\documentclass[10pt,reqno]{amsart}

\usepackage{ulem}

\usepackage{mathtools}                                                      
\mathtoolsset{showonlyrefs=true}                                    

\def\bbp{{\boldsymbol{p}}}
\def\bba{{\boldsymbol{a}}}

\def\cS{{\mathcal S}}
\def\cR{{\mathcal R}}

\def\mL{{\mathbb L}}

\usepackage[utf8]{inputenc}
\usepackage[T1]{fontenc}
\usepackage{lmodern}
\usepackage{xspace}
\usepackage{amsfonts,amsmath,amssymb,amsthm}
\usepackage{enumitem}
\usepackage[small]{titlesec}
\usepackage{comment}
\usepackage{mathrsfs}

\usepackage{fancyhdr}

\usepackage{hyperref}

\usepackage{color}
\usepackage[a4paper,body={160mm,235mm},centering]{geometry}
\newcommand{\Leb}{L}

\newcommand{\Linf}[1]{ \Leb^\infty_T\mathcal{C}_{\bf d}^{#1}}
\newcommand{\Linfa}[2]{ \Leb^{\infty}_{{#2}}\mathcal{C}_{\bf d}^{#1}}
\newcommand{\Linfb}[3]{ \Leb^{\infty}_{{#2},{#3}}\mathcal{C}_{\bf d}^{#1}}

\newcommand{\Linfo}[1]{ \Leb^\infty_T\mathcal{C}_{\bf o}^{#1}}

\newcommand{\pn}{p^{(n)}}

\newcommand{\Cforward}{C_{\text{\bf f}}}
\newcommand{\etaforward}{\eta_{\text{\bf f}}}
\newcommand{\Cbackward}{C_{\text{\bf b}}}
\newcommand{\etabackward}{\eta_{\text{\bf b}}}

\newcommand{\initlaw}{\mu_0}

\def\eps{\epsilon}
\def\wt{\widetilde}
\def\bA{\mathbf{A}}
\def\gR{\mathbf{R}}

\def\B{\mathbb{B}}

\def\d{\mathrm{d}}

\def\bB{{\mathbb{B}}}

\def\E{{\mathbb{E}}}
\def\N{{\mathbb{N}}}
\def\T{{\mathbb{T}}}
\def\K{{\mathbf K}}

\def\R{{\mathbb{R}}}
\def\Rd{{\mathbb{R}^d}}
\def\Rdd{{\mathbb{R}^{2d}}}
\def\gF{{\mathbf{F}}}
\def\bGamma{{\boldsymbol{\Gamma}}}

\def\1{\mbox{1\hspace{-0.25em}l}}

\def\mI{{{\mathbb I}}}
\def\cA{{{\mathcal A}}}
\def\cL{{{\mathcal L}}}

\def \t {\tau}
\def \fy {\mathfrak y}

\def \dif {{\rm d}}
\def\tr{\mathrm {tr}}

\def\bd{{\mathbf{d}}}
\def\be{{\mathbf{e}}}

\def\btheta{{\boldsymbol{\theta}}}
\def\bTheta{{\boldsymbol{\Sigma}}}

\def\bvtheta{{\boldsymbol{\vartheta}}}
\def\bxi{{\boldsymbol{\xi}}}
\def\bsigma{{\boldsymbol{\sigma}}}

\def\pFsig{{p_{\gF,\sigma}}}

\def \cA{\mathcal A}
\def\X{{\bf{X}}}

\def\x{{\boldsymbol{x}}}

\def\y{{\boldsymbol{y}}}
\def\z{{\boldsymbol{z}}}
\def\w{{\boldsymbol{w}}}

\def\<{\langle}
\def\>{\rangle}

\newcommand\Ac{\mathscr{A}}

\newcommand\Fc{\mathscr{F}}

\newcommand\Lc{\mathscr{L}}

\newcommand\sC{\mathscr{C}}

\newcommand\Kc{\mathscr{K}}

\newtheorem{propo}{Proposition}
\newtheorem{definition}{Definition}

\newcommand{\dotafter}[1]{#1.}
\titleformat{\section}[hang]
{\normalfont\large\bfseries}{\thesection.}{.5em}{\dotafter}[]
\titleformat{\subsection}[runin]
{\normalfont\bfseries}{\thesubsection.}{.4em}{}[.]
\titlespacing*{\subsection}{0pt}{3ex plus 1ex minus .2ex}{1em}
\titleformat{\paragraph}[runin]{\normalfont\bfseries}{\theparagraph.}{.4em}{}[.]
\theoremstyle{plain}
\newtheorem{THM}{Theorem}
\newtheorem{lemma}[THM]{Lemma}
\newtheorem{PROP}[THM]{Proposition}
\newtheorem{cor}[THM]{Corollary}

\theoremstyle{definition}
\newtheorem{df}{Definition}

\theoremstyle{remark}
\newtheorem{REM}{Remark}

\usepackage{amsthm}          
\numberwithin{equation}{section}

\allowdisplaybreaks

\begin{document}
\title[Heat kernels for degenerate Brownian SDEs with Besov drift]{On heat kernel estimates for Brownian SDEs with distributional drift}

\author[S. Menozzi]{St\'ephane Menozzi}
\address{Universit\'e d'Evry Val d'Essonne, Paris Saclay, 
Laboratoire de Math\'ematiques et Mod\'elisation d'Evry, UMR CNRS 8071, 23 Boulevard de France 91037 Evry, France}
\email{stephane.menozzi@univ-evry.fr}

\author[S. Pagliarani]{Stefano Pagliarani}
\address{Universit\`a di Bologna, 
Dipartimento di Matematica, Piazza di Porta S. Donato 5, Italy}
\email{stefano.pagliarani9@unibo.it}

\begin{abstract}
We establish heat-kernel bounds and regularity estimates for the {transition densities of the diffusion} associated with the martingale problem corresponding to the generator of a \textit{formal} multi-dimensional Brownian SDE with singular drift. {As a by-product, we also derive Schauder estimates for the associated Kolmogorov (kinetic) Cauchy problem.} We consider both {the cases of non-degenerate and degenerate noise (e.g. kinetic-type models)}, in the so-called Young regime.
Namely, we consider a time inhomogeneous drift in {$L^{\infty}_{[0,T]}\mathcal C_{\bf o}^{\beta}$} for some fixed time horizon $T$, where $ \mathcal C_{\bf o}^{\beta}=B_{\bf o,\infty,\infty}^\beta$ with $\beta\in (-1/2,0) $, with $\bf o$ standing for an underlying distance, namely the usual Euclidean one in the non degenerate setting, and the scale-homogeneous one in the kinetic case. Importantly, the estimates are obtained by employing as \textit{parametrix} the transition density of the SDE (with variable coefficients) without singular perturbation, {as opposed to the standard Levi parametrix obtained by freezing the noise}. 
{Finally, since the noise is multiplicative, the weak well-posedness of the singular SDE is a novel result in itself, and the density estimates directly imply irreducibility and strong Feller property of its solutions.}
\end{abstract}

\maketitle


\vspace{3pt}


\vspace{6pt}

\noindent \textbf{Acknowledgments:} The research of S.P was partially supported by the INdAM - GNAMPA, Italy project CUP\_J53D23003800006. The research of S.P was partially supported by the PRIN22 project CUP\_E53C23001670001. The research of S.M. and S.P was partially supported by the INdAM - GNAMPA, Italy project CUP\_E53C23001670001.

\thispagestyle{fancy}

\vspace{12pt}

\section{Introduction}
\subsection{Statement of the problem and related literature}
We are interested in 
studying the following \textit{formal} degenerate $2d$-dimensional Brownian driven SDE:
\begin{equation}
\label{SDE}
\begin{cases}
\d X^1_t = \big( F_1(t,X_t) + b(t,X_t) \big) \d t+ \sigma(t,X_t) \d W_t,\\
\d X^2_t = F_2(t,X_t) \d t,
\end{cases}\qquad t\in [0,T],
\end{equation}
for a fixed $T> 0 $. In the above equation, $W_t$ formally represents a $q$-dimensional Brownian motion, 
 and the functions 
\begin{equation}\label{eq:coeff_F_sigma}
\sigma:[0,T]\times \Rdd \to 
\mathcal{M}^{d\times q},\qquad
\gF=(F_1,F_2):[0,T]\times \Rdd \to \R^{2d},
\end{equation}
satisfy suitable regularity and weak H\"ormander-type conditions which will be later specified. Critically, we will assume, throughout the paper, that the drift term 
\begin{equation}
b(t,\cdot)\in  \mathcal{C}^{\beta}_{\bd}(\Rdd,\R^d):=\B_{\bd,\infty,\infty}^{\beta} (\Rdd,\R^d),\qquad \beta\in (-1/2,0),
\end{equation}
the latter being the anisotropic Besov space induced by the norm
\begin{equation}\label{HOMO_METRIC}
|{\bf x}|_{\bd} = |{\bf x}_1| + |{\bf x}_2|^{\frac{1}{3}}, \qquad {\bf x}= ({\bf x}_1, {\bf x}_2) \in \Rdd, 
\end{equation}
which reflects the different scaling properties of the components $X^1_t$ and $X^2_t$.
Hereafter, we denote by $|\cdot|$ the Euclidean norm on $\Rd$. We refer to Section \ref{SEC_DEF_BESOV_ANIS} for the precise definition of the anisotropic Besov spaces $\B_{\bd,p,q}^{\beta} (\Rdd,\R^d)$.

We stress that, at a formal level, we can recover a non-degenerate SDE by letting the drift and diffusion coefficients $F_1, b$ and $\sigma$ be independent of the variable $X^2$. In this case the first equation of system \eqref{SDE} is autonomous and reads explicitly as
\begin{equation}\label{eq:sde_kinetic_systems_non_deg}
\d X^1_t = \big( F_1(t,X^1_t) + b(t,X^1_t) \big) \d t+ \sigma(t,X^1_t) \d W_t,\\
\end{equation}
with $F_1:[0,T]\times \Rd \to \Rd$, $\sigma:[0,T]\times \Rd \to \mathcal{M}^{d\times q}
$ satisfying a standard uniform ellipticity condition, and with 
\begin{equation}
b(t,\cdot)\in  \mathcal{C}_{\bf e}^{\beta}(\Rd,\Rd):=\B_{\bf e,\infty,\infty}^{\beta} (\Rd,\R^d),\qquad \beta\in (-1/2,0),
\end{equation}
the latter being the standard isotropic Besov space induced by the Euclidean norm, i.e.
\begin{equation}
|x |_{\bf e} := |x |, \qquad x\in\Rd.
\end{equation}
Once more, we refer to Section \ref{SEC_DEF_BESOV_ANIS} for the precise definition of $\B_{{\bf e},p,q}^{\beta} (\Rd,\R^d)$. Hereafter, we may write for simplicity $\B_{\infty,\infty}^{\beta} $ for $\B_{\bf e,\infty,\infty}^{\beta} (\Rd,\R^d)$, i.e. when the distance is not explicitly specified it is the usual euclidean one and the dimensions are possibly omitted.

The study of the solutions to \eqref{SDE} is deeply linked to the study of the operator 
\begin{equation}\label{eq:op_L}
\Lc = \Kc + \langle b(s,\cdot) ,  \nabla_1   \rangle,
\end{equation}
where we denote by $\nabla_1$ the gradient with respect to the first (non-degenerate) $d$ components, and 
$\Kc$ is the Kolmogorov operator of the non-singular version of \eqref{SDE}, 
namely when $b\equiv 0$. Precisely, 
\begin{equation}\label{gen}
\Kc := \partial_s + \frac 12{\rm Tr} \Big (\sigma\sigma^*(s,\x) \nabla_1^2  \Big) + \langle  \gF(s,\x) , \nabla \rangle=:\partial_s +\Ac_s,
\end{equation} 
where $\nabla=(\nabla_1,\nabla_2) $ stands for the full gradient on $\R^{2d}$ and $\nabla_1^2 $ for the Hessian operator with respect to the first $d$ variables.

A prototypical example of \eqref{SDE} is the kinetic SDE
\begin{equation}\label{eq:singular_SDE_kinetic}
\begin{cases}
\d X^1_t =  b(t,X_t) \d t+ \sigma(t,X_t) \d W_t,\\
\d X^2_t = X^1_t \d t,
\end{cases}
\end{equation}
where the components $X^1$ and $X^2$ represent velocity and position of a particle in the phase space. In the case of additive noise, namely when $\sigma\equiv I_d$, 
\eqref{eq:singular_SDE_kinetic} is a particular case of the system considered in \cite{issoglio2024degenerate}, where existence and uniqueness of martingale solutions were proved for a multi-dimensional linear drift that respects the H\"ormander structure. We also mention the work by Hao \textit{et al.} \cite{hao:zhan:zhu:zhu:24} concerning the extension of \cite{dela:diel:16}, i.e. beyond the current Young regime, to the degenerate kinetic case for a drift in the corresponding anisotropic Besov space with the usual critical threshold $\beta>-2/3$. 

In the non-degenerate case with additive noise, i.e. restricting to $X_t^1$ ($F\equiv0$) with  
$\sigma=I_d$, 
it is known that the formal SDE
\begin{equation}\label{NON_DEG_MARG}
\d X_t^1=b(t,X_t^1) \d t +\sigma \d W_t 
\end{equation}
is well posed (see \cite{flandoli_multidimensional_2017}), for $\beta\in (-\frac 12,0)$, in terms of what the authors therein call virtual solutions and that its dynamics 
can actually be rigorously described in terms of a Dirichlet process. We also refer to \cite{dela:diel:16} for even rougher drifts, i.e. $ \beta>-2/3$, provided the drift has some appropriate structure, and to \cite{chau:meno:22} for extensions to the general symmetric stable setting in the Young regime (i.e. when no specific additional structure is imposed on the drift). We can as well mention the work \cite{krem:perk:22}, which extends to the general symmetric $\alpha $ stable setting,  $\alpha\in (1,2)$, under the condition $\beta>(2-2\alpha)/3 $, the approach of \cite{dela:diel:16}. 

Concerning heat-kernel and regularity bounds for the transition density of diffusions with singular drift, it is known from \cite{perk:vanz:22} that when $b\in \B_{\infty,1}^\beta $, $\beta \in (-1/2,0)$, the heat kernel associated with the formal generator of \eqref{NON_DEG_MARG} admits Gaussian Aronson-type bounds as well as a bound for the gradient w.r.t. the backward variable. Namely, denoting by $P$ the 
semi-group formally associated with 
\begin{equation}\label{GEN}
L= \langle b, \nabla_1 \rangle  +\frac 12\Delta_1  ,
\end{equation}
acting on a bounded measurable function $f$ as
\begin{equation}
P_tf(x)=\E_x[f(X_t^1)]=\int_{\R^{d}}p(0,x,t,y) f(y) \d y, \qquad x\in \R^d, \ t\in (0,T],
\end{equation}
there exist a positive $C$ depending (at most) 
on $T$, $\beta  $ and on some dyadic Littlewood-Paley block decomposition, and $\lambda\ge 1$ depending on $\beta $, 
such that 
\begin{align}\label{RES}
C^{-1} g_{\lambda^{-1}}(t,y-x)\le   p(0,x,t,y)\le C g_\lambda(t,y-x),\\
|\nabla_x p(t,x,y)|\le C
t^{-\frac{1}{2}}
g_\lambda(t,y-x),\label{RES_bis}
\end{align}
where $g_\lambda(t,\cdot)$ is the Gaussian density with covariance matrix $\lambda tI_d$ (see \eqref{THE_DEF_GAUSS_NON_DEG} below). 

In the non-degenerate setting, we prove the result \eqref{RES}-\eqref{RES_bis} in the case of multiplicative noise, namely when
the diffusion coefficient in \eqref{NON_DEG_MARG} is a funtion $\sigma(t,X_t) $  enjoying suitable spatial regularity, and also prove some sharp regularity bounds with respect to both the backward and forward variables, which are novel even in the case of additive noise. {We chose to focus on the case $b\in \B_{\infty,\infty}^\beta$, which is slightly more general than  $\B_{\infty,1}^\beta$ considered in \cite{perk:vanz:22} (recall indeed that $\B_{\infty,1}^\beta \hookrightarrow \B_{\infty,\infty}^\beta $), because a drift $b\in \B_{\infty,\infty}^\beta$ is somehow \textit{easier} to apprehend concretely. Indeed, the elements of $\B_{\infty,\infty}^\beta $ can be viewed as general derivatives of functions belonging to Hölder spaces of regularity $1+\beta $. The analysis below could also be adapted to consider even more general Besov spaces (see Remark \ref{RQ_MAIN_THMS} below)}. \\

Critically, we derive {as well heat kernel and gradient bounds} for the general equation \eqref{eq:sde_kinetic_systems_non_deg}, which allows for an additional, unbounded, drift coefficient $F_1$. This drift term has an impact on the Gaussian bounds through the associated deterministic flow, which reflects the transport of the initial condition (see Theorem \ref{thm-main-non-deg} below), and through the constant $\lambda$, 
which will be dependent on $T$. These two features already appear in the nonsingular controls derived in \cite{meno:pesc:zhan:21} and \cite{chau:meno:pesc:zhan:23} for the dynamics \eqref{eq:sde_kinetic_systems_non_deg} and \eqref{SDE}, respectively, with $b=0$.


Remarkably, we prove our Aronson-type and regularity bounds in the degenerate kinetic-type setting \eqref{SDE}, the controls for the non-degenerate equation \eqref{eq:sde_kinetic_systems_non_deg} being obtained as a direct corollary by letting $F_1, b$ and $\sigma$ be independent of the variable $X^2$.

\subsection{Assumptions and main results}\label{sec:sub_mainresults}

By employing a suitable notation, we will accomodate both the degenerate and non-degenerate settings. In particular, in the set of assumptions below, we let 
\begin{equation}\label{eq:assump_consts}
T>0,\qquad \beta\in (0,-1/2),\qquad \nu\in(-2 \beta, 1)
\end{equation}
 be fixed and set:
\begin{itemize}
\item[-] $(\mathcal{O}, {\bf o}): = (\Rdd, {\bd})$, in the degenerate case;
\item[-] $(\mathcal{O}, {\bf o}): = (\Rd, {\be})$, in the non-degenerate case.
\end{itemize}

The first assumption is a non-degeneracy  H\"ormander-type condition, which is assumed to be satisfied throughout the paper:
\begin{itemize}
\item[{\bf [H-H\"or]}] The function $\sigma$ satisfies the following uniform ellipticity 
condition on $\Rd$:
\begin{equation}
\kappa^{-1} |\xi|^2 \leq \langle \sigma \sigma^\top (t,\x) \xi  , \xi \rangle \leq \kappa |\xi|^2 , \qquad (\x,\xi)\in \mathcal{O}\times\Rd, \quad \text{for a.e. } t\in[0,T],
\end{equation}
for a positive constant $\kappa$. 

Moreover, if $\mathcal{O}=\Rdd$, the function $F_2=F_2(t,\x)$ is differentiable with respect to $x_1$ and there exists a closed convex subset $\mathcal{E} \in GL_d$ (the set of all invertible $d\times d$ matrices with real entries) such that 
\begin{equation}
\nabla_1  F_2 (t,\x) \in \mathcal{E}, \qquad (t,\x)\in(0,T)\times \Rdd,\quad \text{a.e. in time.}
\end{equation}
\end{itemize}
In order to state our assumptions we need to introduce some function spaces which were considered in \cite{chau:meno:pesc:zhan:23} and that we will use in order to state the assumptions concerning the non-singular coefficients $F,\sigma $ in \eqref{SDE} and \eqref{eq:sde_kinetic_systems_non_deg}.

Namely, for $d,l\in\N$, $j\in\{0\}\cup\N$ and $\gamma\in \textcolor{black}{[0,1)}$, given a function $f$ from $\Rd$ to $\R^l$, we say that $f\in \sC_{\mathbf{e}}^{j+\gamma}(\Rd;\R^l) = \sC^{j+\gamma}(\Rd;\R^{l})$ if
\begin{equation}
[f]_{\sC_{\mathbf{e}}^{j+\gamma}(\R^d;\R^{l})}=[f]_{\sC^{j+\gamma}(\R^d;\R^{l})}:=\sum_{k=1}^j \|\nabla^k f\|_{L^\infty(\R^d;\R^l)}+\sup_{x\not=y,|x-y|\leq 1}\frac{|\nabla^j f(x)-\nabla^jf(y)|}{|x-y|^\gamma}<\infty,
\end{equation}
where $\nabla^k$ stands for the $k$-order gradient. \textcolor{black}{Note that the functions in $\sC^{j+\gamma}(\R^d;\R^{l})$ can be unbounded and have sublinear growth}. 

Importantly the functions in  $\sC^0(\R^d;\R^{l})$ can  be possibly discontinuous and also satisfy (see \cite{wang:zhan:16}, Lemma 2.3)
\begin{equation}\label{eq:sub_linear_growth}
\sup_{x\neq y} \frac{|f(x)-f(y)|}{1+|x-y|}<+\infty.
\end{equation}
We now recall the definition of a \emph{anisotropic H\"older spaces} associated with \eqref{HOMO_METRIC} (see e.g. \cite{lun:97}).
Given a function $f$ from $\Rdd$ to $\R^l$, we say that $f\in \sC_{\mathbf{d}}^{j+\gamma}(\R^{2d};\R^l)$ if
\begin{equation}
[f]_{\sC_{\mathbf{d}}^{j+\gamma}(\R^{2d};\R^l)}:=\sup_{x_2\in\R^{d}} [f(\cdot,x_2)]_{\sC^{j+\gamma}(\R^d;\R^l)}
+\sup_{x_1\in\R^{d}} [f(x_1,\cdot)]_{\sC^{(j+\gamma)/3}(\R^d;\R^l)}<\infty.
\end{equation}
In particular, for $f\in\sC_{\mathbf{d}}^{1+\gamma}(\R^{2d};\R)$ the following Taylor's expansion holds:
\begin{align}\label{Taylor}
|\mathcal T_f(\x,\y):=f(\x)-f(\y)-\langle \nabla_{x_1}f(\y) , (\x-\y)_1 \rangle|\leq C_\gamma\|f\|_{\sC^{1+\gamma}_{\bf d}}|\x-\y|^{1+\gamma}_{\bf d}, \qquad \x ,\y \in\Rdd.
\end{align}

Finally, for a measurable function $g:[0,T]\times \mathcal{O} \to \R^l$, we set
\begin{equation}
[g]_{{\bf o},T,j+\gamma} := \sup_{t\in[0,T]} [g{(t,\cdot)}]_{\sC_{\mathbf{o}}^{j+\gamma}(\mathcal{O};\R^l)}.
\end{equation}
Note that, in light of \eqref{eq:sub_linear_growth}, {if $\sup_{t\in [0,T]}|g(t,0)|$ is finite}, then $[g]_{{\bf o},T,0} <\infty$ implies that $g$ is bounded with respect to the time variable, for any fixed spatial point
.

The following standing regularity assumptions will be recalled in each statement, when needed, with $T, \beta$ and $\nu$ as in \eqref{eq:assump_consts}:
\begin{itemize}
\item[{\bf [H-${\bf F}$]}] $F_1:[0,T]\times \mathcal{O} \to \Rd$ is a measurable function such that {$\sup_{t\in [0,T]}|g(t,0)|$ is finite}
and $[F_1]_{{\bf o},T,0} < \infty$. 
Moreover, if $(\mathcal{O}, {\bf o}) = (\Rdd , {\bf d})$, then $F_2:[0,T]\times \mathcal{O} \to \Rd$ is a measurable function such that $[F_2]_{{\bf o},T,1+\beta+\nu}<\infty$.
\item[{\bf [H-$\sigma$]}] $\sigma \in L^{\infty}_T {\mathcal C}_{\bf o}^{\beta+\nu}(\mathcal{O},\Rd)$.
\item[{\bf [H-$b$]}] $b\in \Linfo{\beta}(\mathcal{O},\Rd)$.
\end{itemize}
In {\bf [H-$\sigma$]} and {\bf [H-$b$]} above we set $\Linfo{\vartheta}(\mathcal{O},\Rd):=L_{[0,T]}^\infty(\mathbb B_{\mathbf o,\infty,\infty}^\vartheta(\mathcal{O},\Rd))$, the precise definition of (possibly anisotropic) Besov spaces $\mathbb B_{\mathbf o,p,q}^\vartheta$ being provided in Section \ref{SEC_DEF_BESOV_ANIS}. We also emphasize that no regularity is assumed with respect to the time variable, apart from boundedness.

%

Hereafter, we say that Assumption {\bf [H]} is in force whenever {\bf [H-H\"or]}, {\bf [H-$\gF$]}, {\bf [H-$\sigma$]} and {\bf [H-$b$]} hold true.

We introduce now four sets of parameters (which encode the dependency on the underlying constants):
\begin{equation}\label{DEF_THETAS}
\Theta:=\{d,\beta,\nu+\beta,\kappa,  \|\sigma\|_{\Linfo{\beta+\nu}},[F_1]_{{\bf o},T,0}, [F_2]_{{\bf o},T,1+\beta+\nu}   ,\mathcal E \},  \quad \Theta_{T}:=\Theta \cup \{ T \},
\end{equation}
and
\begin{equation}
\Theta_{b}:=\Theta \cup \{\|b\|_{\Linfo{\beta}},\beta\}, \quad \Theta_{T,b}:=\Theta_T \cup (\|b\|_{\Linfo{\beta}},\beta\}.
\end{equation}
Our heat-kernel and gradient bounds will be stated in terms of the Gaussian density $g_\lambda$ defined, for $\lambda>0$, as follows: 
\begin{itemize}
\item in the degenerate case, i.e. $(\mathcal{O}, {\bf o}): = (\Rdd, {\bd})$, 
\begin{equation}\label{THE_DEF_GAUSS_DEG}
g^{ {\bf d}}_\lambda(u,\z):=\frac{1}{(2\pi \lambda)^du^{2d}}\exp\left(-\left\{\frac{|\z_1|^2}{\lambda u}+\frac{|\z_2|^2}{\lambda u^3}\right\}\right), \qquad u>0,\ \z\in \Rdd,
\end{equation}
\item in the non-degenerate case, i.e. $(\mathcal{O}, {\bf o}): = (\Rd, {\bf e})$, 
\begin{equation}\label{THE_DEF_GAUSS_NON_DEG}
g^{\bf e}_\lambda(u,z):=\frac{1}{(2\pi \lambda u)^{\frac d2}}\exp\left(-\frac{|z|^2}{\lambda u}\right), \qquad u>0,\ z\in \Rd.
\end{equation}

\end{itemize}
We also need an appropriate notion of \textit{flow} associated with the drift $\gF$ in \eqref{SDE}, which is \textit{rough} under {\bf [H-$\gF$]}. Namely, in the degenerate case $(\mathcal{O}, {\bf o}): = (\Rdd, {\bd})$, we introduce $\wt\btheta_{t,s}$ such that 
\begin{equation}\label{FLOW_MACRO_MACRO_KIN}
{\dot{\widetilde \btheta}}_{t,s}(\x)=\wt\gF\big(t,\wt\btheta_{t,s}(\x)\big), \quad \wt\btheta_{s,s}(\x)=\x,\qquad \x\in\Rdd, \ 0\leq s \leq t \leq T
\end{equation}
where 
$$
\wt\gF(t,\cdot):=\big(F_1(t,\cdot)*\rho_1, F_2(t,\cdot)*\rho_{|t-s|^{3/2}}\big).
$$
Here,
\begin{equation}
\rho_\eps(\x):=\eps^{-2d}\rho(\eps^{-1}\x), \qquad \eps>0, \ \xi\in\Rdd, 
\end{equation}
with $\rho$ being a smooth density function with compact support, and $*$ stands for the usual spatial convolution. The previous regularizations allow to have a flow univocally defined in the classical sense. 

Similarly, in the non degenerate setting $(\mathcal{O}, {\bf o}): = (\Rd, {\be})$, we {introduce $\wt\theta_{t,s}^{(1)}$ such that}: 
\begin{equation}\label{FLOW_MACRO_NON_DEG}
{\dot{\widetilde \theta}}_{t,s}^{(1)}(x)=\wt F_1\big(t, \wt\theta_{t,s}^{(1)}(x)\big), \quad \wt\theta_{s,s}^{(1)}(x)=x, \qquad x\in\Rd, \ 0\leq s \leq t \leq T,
\end{equation}
where
\begin{equation}
\wt F_1(t,\cdot):= F_1(t,\cdot)*\rho_1.
\end{equation}


We now state the main result of this article.
%
\begin{THM}[The degenerate case]\label{thm-main-deg} Under Assumption {\bf [H]}, with $(\mathcal{O}, {\bf o}): = (\Rdd, {\bd})$,  we have: 
\begin{itemize}
\item[(i)] \emph{Well-posedness:} For any $t_0\in [0,T)$ and any probability distribution $\initlaw$ on $\R^{2d}$, there exists a unique solution, $P_{t_0,\initlaw}$, to the singular \textit{generalized} martingale problem $\text{MP}(b,\gF,\sigma;t_0,\initlaw)$, in the sense of Ethier and Kurtz \cite{EK:86} (see Definition \ref{def_singular_mp} below).
\item[(ii)] \emph{Markov property:} Denoting by $(\x_t)_{t\in [ t_0,T]} $ the canonical process defined on the space of the continuous functions from $[t_0,T]$ to $\Rdd$, equipped with the Borel $\sigma$-algebra of the topology induced by the uniform convergence, and by $(\Fc^{t_0}_s)_{s\in[t_0,T]}$ its natural filtration, we have 
\begin{equation}
P_{t_0,\initlaw}(\x_t \in H | \Fc^{t_0}_s) 
= P(s,\x_s; t,  H) , \quad P_{t_0,\initlaw}\text{-a.s.}, 
\end{equation} 
for any $t_0\leq s\leq t \leq T$ and any Borel set $H\subset \Rdd$, 
where 
\begin{equation}
P(s,\x; t,  H) := P_{t_0,\delta_{\x}} (\x_t \in  H) , \quad \quad 
\x\in\Rdd.
\end{equation}

\item[(iii)] \emph{Transition density estimates:}  
For any $0\leq s< t \leq T$ and $\x\in\Rdd$, the transition probability $P(s,\x; t, \d y)$ admits a density $p(s,\x,t,\cdot)$. Furthermore, there exist two positive constants $(\lambda,C)=(\lambda,C)(\Theta_{T,b})$
such that
{\mathtoolsset{showonlyrefs=false}\begin{align}\label{eq:THM-deg_HK_1} 
C^{-1} {g^{ {\bf d}}_{\lambda^{-1}}}(t-s, \tilde \btheta_{t,s}(\x) - \y)   \leq p (s,\x, t, \y  ) &\leq C {g^{ {\bf d}}_\lambda}(t-s, \tilde \btheta_{t,s}(\x) - \y), \\
| \nabla_{\x_1} p(s, \x, t, \y  ) | & \leq \frac{C}{(t-s)^{\frac{1}{2}}} {g^{ {\bf d}}_\lambda}(t-s, \tilde \btheta_{t,s}(\x) - \y),\label{eq:THM-deg_HK_2} 
\end{align}}
for any $0\le s<t\le T$ and $ \x,\y\in \R^{2d}$, with $g^{ {\bf d}}_\lambda$ as defined in \eqref{THE_DEF_GAUSS_DEG}.

Finally, for all $\etabackward\in (0,{1+\beta}),\ \etaforward\in  (0,\nu+\beta)$ and $\delta=0,1$, there exist two positive constants $\Cbackward:=\Cbackward(\Theta_{T,b},\etabackward)$ and $\Cforward:=\Cforward(\Theta_{T,b},\etaforward) $ such that 
{\mathtoolsset{showonlyrefs=false}\begin{align}%
\hspace{40pt}\left| p(s,\x,t,\y)-  p(s,\x_1, \x'_2,t,\y) \right| & \le \Cbackward \frac{|\x_2-\x_2'|^{\frac{1 + \etabackward}{3}}}{(t-s)^{\frac{1 + \etabackward}{2}}} \notag \\ \label{eq:THM-deg_HK_3}
&\qquad\times \Big({g^{ {\bf d}}_\lambda}(t-s, \tilde \btheta_{t,s}(\x) - \y)+{g^{ {\bf d}}_\lambda}(t-s, \tilde \btheta_{t,s}(\x_1 , \x_2') - \y)\Big),\\ \notag
\left|\nabla_{\x_1}  p(s,\x,t,\y)-\nabla_{\x_1}  p(s,\x',t,\y) \right| & \le  \frac{\Cbackward}{(t-s)^{\frac{1}{2}}}\left(\frac{|\x-\x'|_{\bf d}}{(t-s)^{\frac{1}{2}}}\right)^{\etabackward}\\ \label{eq:THM-deg_HK_4}
&\qquad\times \Big({g^{ {\bf d}}_\lambda}(t-s, \tilde \btheta_{t,s}(\x) - \y)+{g^{ {\bf d}}_\lambda}(t-s, \tilde \btheta_{t,s}(\x ') - \y)\Big),\\ \notag
\left|\nabla_{\x_1}^\delta p(s,\x,t,\y)-\nabla_{\x_1}^\delta p(s,\x,t,\y') \right| & \le\frac{\Cforward}{(t-s)^{\frac{\delta}{2}}}\left(\frac{|\y-\y'|_{\bf d}}{(t-s)^{\frac{1}{2}}}\right)^{\etaforward}\\ \label{eq:THM-deg_HK_5}
&\qquad\times \Big({g^{ {\bf d}}_\lambda}(t-s, \tilde \btheta_{t,s}(\x) - \y)+{g^{ {\bf d}}_\lambda}(t-s, \tilde \btheta_{t,s}(\x) - \y ')\Big),
\end{align}}
for any $0\le s<t\le T$ and $ \x,\x',\y,\y'\in \R^{2d}$.
\end{itemize}
\end{THM}
As a direct corollary of the previous theorem, the same results hold in the non-degenerate case. For the sake of clarity, we state them below in a separate theorem.
\begin{THM}[The non-degenerate case]\label{thm-main-non-deg} 
Under Assumption {\bf [H]}, with $(\mathcal{O}, {\bf o}): = (\Rd, {\be})$, 
we have: 
\begin{itemize}
\item[(i)] \emph{Well-posedness:} For any $t_0\in [0,T)$ and any distribution $\initlaw$ on $\R^{d}$, there exists a unique solution, $P_{t_0,\initlaw}$, to the singular generalized martingale problem $\text{MP}(b,F^1,\sigma,t_0,\initlaw)$, in the sense of \cite{EK:86} (see Definition \ref{def_singular_mp} below, up to the corresponding modification of the generator). 
\item[(ii)] \emph{Markov property:} Denoting by $(x_t)_{t\in [ t_0,T]} $ the canonical process defined on the space of the continuous functions from $[t_0,T]$ to $\Rd$, equipped with the Borel $\sigma$-algebra of the topology induced by the uniform convergence, and by $(\Fc^{t_0}_s)_{s\in[t_0,T]}$ its natural filtration, we have 
\begin{equation}
P_{t_0,\initlaw}(x_t \in H | \Fc^{t_0}_s) 
= P(s,x_s; t,  H) , \quad P_{t_0,\initlaw}\text{-a.s.}, 
\end{equation} 
for any $t_0\leq s\leq t \leq T$ and any Borel set $H\subset \Rd$, 
where 
\begin{equation}
P(s,x; t,  H) := P_{t_0,\delta_{x}} (x_t \in  H) , \quad \quad 
x\in\Rdd.
\end{equation}

\item[(iii)] \emph{Transition density estimates:}  
For any $0\leq s < t \leq T$ and $x\in\Rd$, the transition probability $P(s,x; t, \d y)$ admits a density $p(s,x,t,\cdot)$. Furthermore, there exist two positive constants $(\lambda,C)=(\lambda,C)(\Theta_{T,b})$
such that
\begin{align}\label{eq:THM-deg_HK} 
C^{-1} {g^{\bf e}_{\lambda^{-1}}}\big(t-s,  \tilde \theta_{t,s}^{(1)}(x) - y\big)   \leq p (s,x, t, y  ) &\leq C {g^{\bf e}_\lambda}\big(t-s,  \tilde \theta_{t,s}^{(1)}(x) - y\big), \\
| \nabla_{x} p(s, x, t, y  ) | & \leq \frac{C}{(t-s)^{\frac{1}{2}}} {g^{\bf e}_\lambda}\big(t-s,  \tilde \theta_{t,s}^{(1)}(x) - y\big),
\end{align}
for any $0\le s<t\le T$ and $x,y\in \Rd$, with $g^{\bf e}_\lambda$ as defined in \eqref{THE_DEF_GAUSS_NON_DEG}.

Finally, for all $\etabackward\in (0,1+\beta),\ \etaforward\in  (0,\nu+\beta)$ and $\delta=0,1$, there exist two positive constants $\Cbackward:=\Cbackward(\Theta_{T,b},\etabackward)$ and $\Cforward:=\Cforward(\Theta_{T,b},\etaforward) $ such that 
\begin{align}\label{eq:THM-deg_SENSI}
\hspace{40pt}\left|\nabla_{x}  p(s,x,t,y)-\nabla_{x}  p(s,x',t,y) \right| & \le  {\frac{\Cbackward}{(t-s)^{\frac{1}{2}}}}\left(\frac{|x-x'|}{(t-s)^{\frac{1}{2}}}\right)^{\etabackward}\\
&\qquad\times \Big({g^{\bf e}_\lambda}\big(t-s,\tilde \theta_{t,s}^{(1)}(x) - y\big)+{g^{\bf e}_\lambda}\big(t-s,\tilde \theta_{t,s}^{(1)}(x') - y\big)\Big),\\
\left|\nabla_{x}^\delta p(s,x,t,y)-\nabla_{x}^\delta p(s,x,t,y') \right| & \le\frac{\Cforward}{(t-s)^{\frac{\delta}{2}}}\left(\frac{|y-y'|}{(t-s)^{\frac{1}{2}}}\right)^{\etaforward}\\
&\qquad\times \Big({g^{\bf e}_\lambda}\big(t-s,\tilde \theta_{t,s}^{(1)}(x) - y\big)+{g^{\bf e}_\lambda}\big(t-s,\tilde \theta_{t,s}^{(1)}(x) - y'\big)\Big),
\end{align}
for any $0\le s<t\le T$ and $x,x',y,y'\in \Rd$.
\end{itemize}
\end{THM}

\begin{REM}[About the results]\label{RQ_MAIN_THMS}
Let us emphasize that the heat kernel and gradient estimates 
 in Theorems \ref{thm-main-deg} and \ref{thm-main-non-deg}  
 are actually similar to the ones obtained in \cite{meno:pesc:zhan:21} and \cite{chau:meno:pesc:zhan:23}, respectively, for the corresponding dynamics without the singular perturbation (i.e. with $b=0$). Hence, as far as the marginal laws are concerned, the singular drift does not impact the structure of the estimates but only impacts the regularity thresholds, which are clearly dependent on $\beta$. 

We also mention that the flows defined in \eqref{FLOW_MACRO_MACRO_KIN} and \eqref{FLOW_MACRO_NON_DEG}, which appear in the statement of the theorem, are chosen for coherence with the statement of the main results in \cite{meno:pesc:zhan:21} and \cite{chau:meno:pesc:zhan:23}, respectively. They actually illustrate that the deviations have to be taken into account with respect to the transport of the initial condition by the underlying mollified flow. We anyhow mention that the statement could have been stated with any Peano flow associated with the drifts $\gF $ or $F_1$, respectively, if $F_1$ is continuous in space.
\end{REM}

Let us point out that the results of Theorem \ref{thm-main-non-deg} can be seen as a Brownian extension, in the multiplicative case 
of the results obtained by Fitoussi in \cite{fito:23} for a non-degenerate SDE driven by a symmetric $\alpha $ stable additive noise, with $\alpha\in (1,2)$, and a Besov drift in  $b\in L^r\big([0,T], \mathbb B_{p,q}^\beta\big)$. The parameters $ r,p,q,\beta$ were chosen in order to guarantee well-posedness of the associated martingale problem as established in \cite{chau:meno:22}, namely, $ \beta>-\frac{\alpha-1-\frac dp-\frac \alpha r}{2}$. When $p,q=\infty $, the constraints match the ones of the current work. We restricted to this case for notational simplicity and believe the approach developed hereafter could be adapted to consider, in both the non-degenerate and degenerate setting, drifts in $L^r((0,T),\mathbb B_{p,q,\mathbf o}^\beta)$, under the same previous constraints with $\alpha=2 $. While this would lead to additional technicalities, the main ideas would remain unchanged.
On the technical side, we also point out that one of the differences with the approach in \cite{fito:23} is that, in our case, control of the concentration constants (variance) must be handled with additional care due the (Gaussian) exponential tails, as opposed to the polynomial ones arising in the pure-jump case.

One could as well wonder if Theorem \ref{thm-main-deg} also extends somehow to the pure jump case. Let us mention that, in the degenerate kinetic setting, density estimates are rather delicate to obtain.  The sharpest results obtained in that direction are due to Hou and Zhang, \cite{hou:zhan:24}, who derive rather precise estimates for the density of the couple  $(Z_t,\int_0^t Z_s ds) $ where $Z_s$ is an isotropic $\alpha $-stable process which constitutes the prototypical example of kinetic dynamics. It would be interesting and challenging as well, to check whether the estimates therein remain stable under perturbation by a singular drift.

From the probabilistic point of view, the first two items of Theorems \ref{thm-main-deg} and \ref{thm-main-non-deg} extend the literature concerned with weak regularization by noise, in that the weak well-posedness of an It\^o SDE with singular drift and multiplicative diffusion had not been established before, to the best of our knowledge, even in the non-degenerate case.
One could as well wonder about the link with the \textit{formal} dynamics written in \eqref{SDE}. We insist here that the estimates obtained in Theorems \ref{thm-main-deg} and \ref{thm-main-non-deg} are related to the transition densities of the canonical processes associated with the corresponding martingale problems. Under the former assumptions, a specific meaning to the dynamics was given in \cite{chau:meno:22}, in the non-degenerate case, in the particular case when $F_1\equiv 0 $ and $\sigma \equiv I_d$. 
One of the difficulties in that setting is that one must somehow reconstruct the drift, which is to be understood as a Dirichlet process, from the paths of the canonical process (which in particular requires keeping track of the noise). The reconstruction of the dynamics in the current setting will be the subject of further research.


Eventually, the following corollary is a straightforward consequence of our main result.
\begin{cor}
Under the assumptions of Theorem \ref{thm-main-deg} (or Theorem \ref{thm-main-non-deg}), for any $t_0\in [0,T)$ and any distribution $\initlaw$ on $\Rdd$ (or on $\Rd$), the Markov process $(\x_t)_{t\in{[t_0,T]}}$ (or $(x_t)_{t\in{[t_0,T]}}$), under $P_{t_0,\initlaw}$, is irreducible. Furthermore, the semigroup defined by $P(s,\cdot; t , \d y)$, $0\leq s <t \leq T$, is strong Feller.

In particular, (see \cite{peszat1995strong}) there exists at most one invariant distribution for $P(s,\cdot; t , \d y)$. If the latter exists, then the Markov process $(\x_t)_{t\in{[t_0,T]}}$ is ergodic.
\end{cor}

The paper is organized as follows: we recall in Section \ref{SEC_CONTROLS_WITHOUT_SING} the useful estimates obtained in \cite{chau:meno:pesc:zhan:23} for the transition density of SDE \eqref{SDE} when $b=0$. This density will be here the proxy in the one-step-parametrix type/Duhamel expansion leading to our main results. This approach is somehow, to the best of our knowledge, new and seems promising for further investigations, i.e. it emphasizes it is not necessary to set up a full parametrix expansion starting form the perturbation by the noise. A more sophisticated proxy for which appropriate controls are available can be considered as well.
In Section \ref{SEC_CONTROLS_WITHOUT_SING} we also recall the specific definitions and properties of (possibly anisotropic) Besov spaces that we need for the analysis. Section \ref{SEC_DENS_ESTIMATES_MOLL} is specifically dedicated to the density estimates associated with mollified coefficients. Importantly, we show in Theorem \ref{thm-main-mollified} that estimates similar to those stated in Theorem \ref{thm-main-deg} hold, with constants that only depend on parameters appearing under \textbf{[H]}, i.e. all the estimates obtained therein are uniform with respect to the mollification parameter.
Section \ref{CAUCHY_PB} is dedicated to the  Cauchy problem associated with \eqref{eq:op_L}. 
In particular, well-posedness and related Schauder-type controls are given in Theorem \ref{th:backward_CP_wellposedenss}.
All these results are then used in Section \ref{SEC_PB_MART} to establish well-posedeness of the generalized martingale problem and the transition density estimates corresponding to the canonical process as stated in Theorem \ref{thm-main-deg}.
Eventually, some technical results are gathered in Appendices {\ref{app:proof_lemma_normal}, \ref{AUX_EST} and \ref{EQUIV_CAR_ANIS_BESOV_SPACES}. In particular, in Appendix \ref{AUX_EST} we prove some regularity estimates in the forward variables for the transition density of the model considered in \cite{chau:meno:pesc:zhan:23}, i.e. equation \eqref{SDE}  with $b=0$. These are crucial for the proof of our main result, which is based on duality arguments that require such controls.


\section{Preliminaries}\label{SEC_CONTROLS_WITHOUT_SING}
For the sake of simplicity, throughout the rest of the paper we will concentrate on the degenerate, kinetic, case. When the equation for $X^1$ is autonomous, namely the coefficients in the equation only depend on $X^1$, the non-degenerate setting can be derived as a particular case, 
by considering the corresponding kinetic equation and by integrating along the degenerate variable. Alternatively, the non-degenerate case could also be addressed directly by replacing, in the arguments below, the intrinsic distance  $|\cdot|_{\bf d}$ with the Eclidean one $|\cdot|_{\bf e}$, and working with the density estimates provided in \cite{meno:pesc:zhan:21} instead of those of Proposition \ref{prop-HK-proxy-sde}, from \cite{chau:meno:pesc:zhan:23}. We thus focus on the most delicate case. 
\subsection{Estimates for the non-singular density}\label{sec:prel_nonsingular}

Let us consider the following SDE, obtained from \eqref{SDE} by letting $b\equiv 0$, which we will be using as a proxy:
\begin{equation}\label{sde-proxy}
\begin{cases}
		\d X_t^{1} =F_1 (t,X_t) \d t+\sigma (t,X_t) \d W_t\\
		\d X_t^{2} =F_2 (t,X_t) \d t.
\end{cases}
\end{equation}
Here, the coefficients ${\bf F}=(F_1,F_2),\sigma$ are as in \eqref{eq:coeff_F_sigma} and will be assumed to satisfy {\bf [H-H\"or]}, {\bf [H-${\bf F}$]} and {\bf [H-$\sigma$]} {with $(\mathcal{O}, {\bf o}): = (\Rdd, {\bd})$}. This SDE has been thoroughly studied in \cite{chau:meno:pesc:zhan:23}, where its weak well-posedness, together with some uppper/lower bounds and regularity estimates for its transition density, hereafter denoted by $p_{\gF,\sigma}(s,\x,t,\y)$, were established. These results are reported below in Propositions \ref{prop-HK-proxy-sde_new} and \ref{prop-HK-proxy-sde}.

To shorten notation, for any given $\lambda>0$ we also denote by $p_{\gF,\lambda}= p_{\gF,\lambda}(s,\x,t,\y)$ the transition density of the solutions to \eqref{sde-proxy} with $\sigma \equiv \lambda \mathbb I_{d\times d}$.

\begin{propo}\label{prop-HK-proxy-sde} 
Let assumptions {\bf [H-H\"or]}, {\bf [H-${\bf F}$]} and {\bf [H-$\sigma$]}, {with $(\mathcal{O}, {\bf o}): = (\Rdd, {\bd})$}, be in force. 
For any initial probability measure $\initlaw$ on $\Rdd$, and initial time $t_0\in [0,T]$, there exists a unique solution, $P^{\gF,\sigma}_{t_0,\initlaw}$, to the martingale problem 
associated to \eqref{sde-proxy}, in the classical sense of Stroock-Varadhan, which is a Markov process. Furthermore, the corresponding transition probability function $P^{\gF,\sigma}(s, \x , t, \d \y)$ has a density $p_{\gF,\sigma}(s,\x,t,\y)$, which is a measurable function satisfying the following estimates.

There exist three positive constants $(C,C',\lambda):=(C,C',\lambda)(\Theta_T)$
such that
	\begin{align}
	C^{-1}p_{\gF,\lambda^{-1}}(s,\x,t,\y)	\leq p_{\gF,\sigma}(s,\x,t,\y) &\leq C\, p_{\gF,\lambda}(s,\x,t,\y) \leq C' p_{\gF,2 \lambda}(s,\x,t,\y)
	,\label{HK-density_upper} \\
		|\nabla_{x_1}p_{\gF,\sigma}(s,\x,t,\y)|&\leq \frac{C}{(t-s)^{\frac{1}{2}}} p_{\gF,\lambda}(s,\x,t,\y)\label{HK-gradient-back},
		\end{align}
for any $0\leq s \leq t \leq T$ and $\x,\y\in \R^{2d}$.

		Furthermore, for any $\etabackward\in (0,1)$ and $\etaforward\in (0,\nu+\beta) $, 
		there exist two positive constants $\Cbackward:=\Cbackward(\Theta_{T},\etabackward)$ and $\Cforward:=\Cforward(\Theta_{T},\etaforward) $ 
such that
\begin{align}
\left| p_{\gF,\sigma}(s,\x,t,\y)-  p_{\gF,\sigma}(s, \x_1 , \x_2',t,\y) \right| & \le \Cbackward \frac{|\x_2-\x_2'|^{\frac{1 + \etabackward}{3}}}{(t-s)^{\frac{1 + \etabackward}{2}}}  \Big(p_{\gF,\lambda}(s,\x,t,\y)+ p_{\gF,\lambda}(s,\x_1 , \x_2',t,\y) \Big), \label{HK-holder-back} \\
		|\nabla_{x_1} p_{\gF,\sigma}(s,\x,t,\y)-\nabla_{x_1} p_{\gF,\sigma}(s,\x',t,\y)| &\leq  \frac{\Cbackward}{(t-s)^{\frac{1}{2}}} \left(\frac{|\x-\x'|_{\bf d}}{(t-s)^{\frac{1}{2}}}\right)^{\etabackward}   \Big( p_{\gF,\lambda}(s,\x,t,\y)+p_{\gF,\lambda}(s,\x',t,\y)\Big),\label{HK-holder-gradient-back}\notag
		\\
		&\\
		|\nabla_{x_1}^\delta p_{\gF,\sigma}(s,\x,t,\y)-\nabla_{x_1}^\delta p_{\gF,\sigma}(s,\x,t,\y')| &\leq  \frac{\Cforward}{(t-s)^{\frac{\delta}{2}}} \left(\frac{|\y-\y'|_{\bf d}}{(t-s)^{\frac{1}{2}}}\right)^{\etaforward}   \Big( p_{\gF,\lambda}(s,\x,t,\y)+p_{\gF,\lambda}(s,\x,t,\y')\Big),\notag\\
		\label{HK-holder-gradient-forward}
	\end{align}
\begin{align} 
|\nabla_{x_1}^{\delta} &p_{\gF,\sigma}(s,\x,t,\y) -\nabla_{x_1}^{\delta} p_{\gF,\sigma}(s,\x,t,\y') - \nabla_{x_1}^{\delta} p_{\gF,\sigma}(s,\x',t,\y) +\nabla_{x_1}^{\delta} p_{\gF,\sigma}(s,\x',t,\y')| \\
& \leq \frac{\Cbackward + \Cforward}{(t-s)^{\frac{\delta}{2}}} \left(\frac{|\y-\y'|_{\bf d}}{(t-s)^{\frac{1}{2}}}\right)^{^{\etaforward}} 
\left(\frac{|\x-\x'|_{\bf d}}{(t-s)^{\frac{1}{2}}}\right)^{\etabackward}\\
&\qquad\quad \times \Big( p_{\gF,\lambda}(s,\x,t,\y)+p_{\gF,\lambda}(s,\x,t,\y')+p_{\gF,\lambda}(s,\x',t,\y)+p_{\gF,\lambda}(s,\x',t,\y')\Big) ,\label{eq:ste_new}
\end{align}
\begin{align} | &p_{\gF,\sigma}(s,\x,t,\y) - p_{\gF,\sigma}(s,\x,t,\y') -  p_{\gF,\sigma}(s,\x_1,\x_2',t,\y) + p_{\gF,\sigma}(s,\x_1,\x_2',t,\y')| \\
& \leq \frac{\Cbackward + \Cforward}{(t-s)^{\frac{\delta}{2}}} \left(\frac{|\y-\y'|_{\bf d}}{(t-s)^{\frac{1}{2}}}\right)^{^{\etaforward}} 
\left(\frac{|\x_2-\x_2'|^{\frac{1+\etabackward}3}}{(t-s)^{\frac{1+\etabackward}{2}}}\right)\\
&\qquad\quad \times \Big( p_{\gF,\lambda}(s,\x,t,\y)+p_{\gF,\lambda}(s,\x,t,\y')+p_{\gF,\lambda}(s,\x_1,\x_2',t,\y)+p_{\gF,\lambda}(s,\x_1,\x_2',t,\y')\Big) ,\label{eq:ste_M_DEG_new}
\end{align}

for any $0\leq s \leq t \leq T$, $\x,\x',\y,\y'\in \R^{2d}$ and $\delta\in \{0,1\} $.

\end{propo}
\begin{propo}\label{prop-HK-proxy-sde_new}
Let assumptions {\bf [H-H\"or]} and {\bf [H-${\bf F}$]}, {with $(\mathcal{O}, {\bf o}): = (\Rdd, {\bd})$}, be in force. For any $\lambda>0$, there exist two positive constants $(C,\kappa):=(C,\kappa)(\Theta_T,\lambda)$ such that
\begin{equation}\label{eq:HK-density_upper_lower}
C^{-1} {g^{ {\bf d}}_{\kappa^{-1} \lambda}}\big(t-s, \tilde \btheta_{t,s}(\x) - \y\big) \leq p_{\gF,\lambda}(s,\x,t,\y) \leq  C {g^{ {\bf d}}_{\kappa \lambda}}\big(t-s, \tilde \btheta_{t,s}(\x) - \y\big),
\end{equation}
for any $0\leq s \leq t \leq T$ and $\x,\y\in \R^{2d}$.
\end{propo} 

Propositions \ref{prop-HK-proxy-sde} and \ref{prop-HK-proxy-sde_new} stem almost entirely from \cite[Theorem 1]{chau:meno:pesc:zhan:23}, except for the regularity estimate with respect to the forward variables, namely \eqref{HK-holder-gradient-forward}, and the mixed, w.r.t. the backward and forward variables, regularity estimate \eqref{eq:ste_new}. The proof of the latter two, which were not explicitly proven in \cite{chau:meno:pesc:zhan:23}, is postponed until Appendix \ref{AUX_EST} for the sake of clarity.

\subsection{Besov spaces: the thermic characterization} 
\label{SEC_DEF_BESOV_ANIS}
\paragraph{A brief recall about the thermic characterization for isotropic Besov spaces}\phantom{SPAZIO!}\\
We first recall here the so-called thermic characterization of Besov spaces in the isotropic case, see e.g. Chapter 5, Section 3 of \cite{lema:02} or Section 2.6.4 in Triebel \cite{trie:83}. Namely, in this setting, we can characterize Besov spaces using the following:

\begin{definition}[Isotropic Besov spaces]\label{DEF_BESOV_THERMIC_ISO}
For $\vartheta \in \R$ and $p,q\in (0,+\infty]$, set
$$\B_{p,q,\bf e}^\vartheta (\R^d):= \left\{ f \in \mathcal{S}'(\R^d) : \Vert f\Vert_{\mathcal{H}_{p,q{,\bf e}}^\vartheta} < \infty\right\},$$
where: 
	\begin{align}\label{thermic-char_ND}
		\Vert f\Vert_{\mathcal{H}_{p,q,\bf e}^{\vartheta}} &:= \Vert g(1,\cdot)* f \Vert_{L^p} + \left\{ \begin{aligned} &\left( \int_0^1 \frac{\d v}{v} v^{(n-\frac{\vartheta}{2})q} \Vert \partial_v^n {g (v,\cdot)} * f \Vert_{L^p}^q\right)^{\frac{1}{q}}, \qquad &q<\infty,\\
			&\sup_{v\in (0,1]} v^{n-\frac{\vartheta}{2}} \Vert \partial_v^n {g (v,\cdot)} * f \Vert_{L^p}, &q=\infty,
		\end{aligned}  \right. \nonumber \\
		&=: \Vert g(1,\cdot)* f \Vert_{L^p} + \mathcal{T}_{p,q,\bf e}^{\vartheta} [f],
	\end{align}
with	$n$ being any non-negative integer (strictly) greater than $\vartheta/2$, 
and with
\begin{equation}
g(v,x) := g^{ {\bf e}}_{1}(v,x), \qquad v>0,\ x\in \Rd ,
\end{equation}
where $g^{ {\bf e}}_{1}$ is the Gaussian density defined in \eqref{THE_DEF_GAUSS_NON_DEG}.
\end{definition}
It is known that the norm above is equivalent with the Besov norm associated with the Littlewood-Paley decomposition (see again Chapter 5, Section 3 of \cite{lema:02} or Appendix \ref{EQUIV_CAR_ANIS_BESOV_SPACES}). Also, for $\vartheta\in \R\setminus\N $, $\B_{p,q,\bf e}^\vartheta (\R^d) $ coincides with the usual Hölder space $\mathcal C^\vartheta(\R^d)$
, see e.g. Theorem 2.7 in \cite{sawa:18}.

\paragraph{A thermic characterization for anisotropic Besov spaces}\phantom{SPAZIO!}\\
To handle the kinetic degenerate setting we will need to introduce anisotropic Besov spaces whose regularity reflects the scaling properties of the underlying homogeneous distance introduced in \eqref{HOMO_METRIC}. These spaces are typically characterized through an appropriate anisotropic Littlewood-Paley decomposition. The interested reader can refer to \cite[Chapter 5]{triebel06} for the general definition, and to the references therein for an account on the historical developments in this subject. Recently, such characterization was considered also in \cite{issoglio2024degenerate}, \cite{hao:rock:zhan:26} and \cite{hao:jabi:meno:rock:zhan:26}, in settings related to the one of this paper. 
We here extend the previous thermic approach.

To this end we will consider, for $v>0$, the anisotropic heat kernel $g(v,\cdot):=g^{ {\bf d}}_{1}(v,\cdot)$ with $g^{ {\bf d}}_{1}$ being the anisotropic Gaussian density in \eqref{THE_DEF_GAUSS_DEG}, as opposed to isotropic one in \eqref{THE_DEF_GAUSS_NON_DEG} that was utilized in Definition \ref{DEF_BESOV_THERMIC_ISO}. 
%
Note that the density $g(v,\cdot) = g^{ {\bf d}}_{1}(v,\cdot)$ has the same scaling behavior as the density of the couple $(W_v,\int_0^v W_r dr) $ formed by the standard $d$-dimensional Brownian motion and its integral first investigated, from the operator viewpoint, by Kolmogorov in \cite{kolm:34}. In particular, the main diagonals of their covariance matrices coincide with vector given by 
\begin{equation}
(\underbrace{v,\dots, v}_{d\text{ times}}, \underbrace{v^3, \dots, v^3}_{d\text{ times}}).
\end{equation}

In this setting, we then readily have the following moments estimate:
\begin{equation}\label{spatial-moments-kolm-kern}
	\forall  \delta>0,\, \exists C_\delta \ge 1 \text{ such that}
	\quad \int_{\R^{2d}} {g^{ {\bf d}}_{1}}(v,\x)|\x|^\delta_{\mathbf{d}} \d \x \le C_\delta v^{\frac{\delta}{2}}, \quad v>0,
\end{equation}
which specifically reflects the same regularizing behavior as in the non-degenerate case once the distance $|\cdot|_{{\mathbf d}} $, introduced in \eqref{HOMO_METRIC}, reflecting the scales of the underlying system is taken into account.

{More generally, for any multi-indicex $(\delta, i , \mathbf j)\in \N_0 \times \N_0\times\N_0^{2d} $, 
there exist two positive constants $\kappa$ and $C = C(\delta, i , \mathbf j)$ such that
\begin{equation}\label{BD_THERM_DEG_DER_TEMP}
|\x|^\delta_{\mathbf{d}}\, |\partial_v^i D_\x^{\mathbf j} {g^{ {\bf d}}_{\lambda}}(v,\x)|\le C\, v^{\frac{\delta}{2}-(i+\frac{1}{2}|{j}_1|+\frac 32|{j}_2|)}\,{g^{ {\bf d}}_{\kappa \lambda}}(v,\x), \qquad \lambda, v>0, \ \x\in \R^{2d} ,
\end{equation}
where $  D_\x^{\mathbf{j}}=D_{x_1}^{j_1}D_{x_2}^{j_2}$ is the multi-derivative operator that differentiates the variable $x_1$ according to the multi-index $j_1$ and $x_2$ according to $j_2$, while $|j_1|$, $|j_2| $ denote here the sum of the components of the corresponding multi-index, with a slight abuse of notation.
} 
Note that the concentration constant is slightly deteriorated in order to absorb the monomials in the spatial variable arising from the differentiation. With these tools at hand we extend the previous definition. 
\begin{definition}[Anisotropic Besov spaces]\label{DEF_BESOV_THERMIC_ANISO}
For $\vartheta \in \R$ and $p,q\in (0,+\infty]$, set
$$\B_{p,q,\bf d}^\vartheta (\R^{2d},\R^d):= \left\{ f \in \mathcal{S}'(\R^{2d},\R^d) : \Vert f\Vert_{\mathcal{H}_{p,q,{\bf d}}^\vartheta} < \infty\right\},$$
where: 
	\begin{align}\label{thermic-char}
		\Vert f\Vert_{\mathcal{H}_{p,q,\bf d}^{\vartheta}} &:= \Vert g(1,\cdot) * f \Vert_{L^p} + \left\{ \begin{aligned} &\left( \int_0^1 \frac{\d v}{v} v^{(n-\frac{\vartheta}{2})q} \Vert \partial_v^n g(v,\cdot) * f \Vert_{L^p}^q\right)^{\frac{1}{q}}, \qquad &q<\infty,\\
			&\sup_{v\in (0,1]} v^{n-\frac{\vartheta}{2}} \Vert \partial_v^n  g(v,\cdot) *  f \Vert_{L^p}, &q=\infty,
		\end{aligned}  \right. \nonumber \\
		&=: \Vert g(1,\cdot) * f \Vert_{L^p} + \mathcal{T}_{p,q,\bf d}^{\vartheta} [f],
	\end{align}
with	$n$ being any non-negative integer (strictly) greater than $\vartheta/2$, 
and with
\begin{equation}
g(v,\x) := g^{ {\bf d}}_{1}(v,\x), \qquad v>0,\ \x\in \Rdd ,
\end{equation}
where $g^{ {\bf d}}_{1}$ is the Gaussian density defined in \eqref{THE_DEF_GAUSS_DEG}.
\end{definition}

Following e.g. \cite{trie:83} or \cite{baho:chem:danc:11}, which provide equivalence of the two characterizations in the isotropic case, it can be shown that the thermic characterization of Definition \ref{DEF_BESOV_THERMIC_ANISO} is indeed  equivalent to the one in \cite[Chapter 5]{triebel06} based on Littlewood-Paley decomposition, the latter coinciding in turn with the one in \cite{hao:rock:zhan:26} provided that the integrability index therein be the same for the degenerate and non-degenerate variables. For the sake of completeness, a proof of this equivalence is provided in Appendix \ref{EQUIV_CAR_ANIS_BESOV_SPACES}. Again, when $\vartheta \in \R\setminus\N$ it can be checked that $\B_{p,q,\bf d}^\vartheta (\R^{2d},\R^d)) $ can be identified with the anisotropic Hölder space $C_{\mathbf d}^\vartheta(\R^{2d},\R) $ associated with $\mathbf d$.

As a consequence of Lemma 2.6, \textit{ii)} in  \cite{hao:rock:zhan:26} (generalized Young convolution inequality in anisotropic Besov spaces), and the related embeddings (point \textit{i)} of the same Lemma) we have the following proposition
\begin{propo}[Duality in anisotropic Besov spaces]\label{PROP_DUALITY}
For any $\vartheta \in \R$ and $p,q\in (0,+\infty]$,  we have
\begin{align*}
\|f * g\|_{L^\infty}\le \|f\|_{\B_{p,q,\bf d}^\vartheta}\|g\|_{\B_{p',q',\bf d}^{-\vartheta}},\qquad \text{with}\quad \frac1p+\frac 1{p'}=1,\quad \frac 1q+\frac 1{q'} = 1.
\end{align*}
In particular, provided the left-hand-side term below makes sense, one obtains from the above control
$$\bigg|\int_{\R^{2d}} f(\y)  g(\y) \d \y\bigg| \le \|f\|_{\B_{p,q,\bf d}^\vartheta}\|g\|_{\B_{p',q',\bf d}^{-\vartheta}}.$$
\end{propo}

{To make sense of the product between distributions we will consider the so-called {Bony's product} (see \cite{bony1981interaction}) between Besov distributions, which extends the product between functions (which are also tempered distributions). The following result can be proved by proceeding exactly as in \cite[Lemma 2.1]{gubinelli2015paracontrolled} (see also \cite{hao:zhan:zhu:zhu:24} for the anisotropic case).
\begin{PROP}\label{prop:bony_prod}
Let $\alpha,\gamma \in \R$  
with  
$\alpha+\gamma >0$.  
 Then Bony's product restricted to   $\mathcal{C}_{\bf d}^{\alpha} \times \mathcal{C}_{\bf d}^{\gamma}$ is a bilinear form from $\mathcal{C}_{\bf d}^{\alpha} \times \mathcal{C}_{\bf d}^{\gamma}$ to $\mathcal{C}_{\bf d}^{\alpha \wedge \gamma}$ and 
there exists a positive constant $C=C(\alpha,\gamma)$ such that
\begin{equation}
\label{eq: prod estim}
\| fg\|_{\alpha \wedge \gamma}\leq C\, \|f\|_\alpha \|g\|_\gamma, \qquad f\in \mathcal{C}_{\bf d}^{\alpha},\ g\in \mathcal{C}_{\bf d}^{\gamma}.
\end{equation}
\end{PROP}}

\section{Density estimates for the mollified equation via one-step parametrix}\label{SEC_DENS_ESTIMATES_MOLL}

The goal of this section is to prove some global upper/lower bounds, together with regularity estimates, analogous to those in Proposition \ref{prop-HK-proxy-sde}, for the transition density of the solution to the regularized SDE
\begin{equation}
	\begin{cases}
		\d X_t^{(n),1}=\left(F_1 \big(t,X^{(n)}_s\big)+b^{(n)}\big(t,X_t^{(n)}\big)\right) \d t+\sigma \big(t,X_t^{(n)}\big) \d W_t , \\
		\d X_t^{(n),2}=F_2 \big(t,X_t^{(n)}\big) \d t,
	\end{cases}\label{sde-mollified}
\end{equation}
where $b^{(n)}(t,\cdot)$ is a suitable smooth mollification of the distribution $b(t,\cdot)$. Critically, these estimates will be uniform w.r.t. the regularization parameter $n$.

We start by defining the sequence $(b^{(n)})_n$. Let $\Phi$ be a standard mollifier on $\Rdd$ with compact support. 
For any $g \in  \mathcal{C}^{\alpha}_{\bd}$, with $\alpha\in \R$, it is well defined, by duality, the family
\begin{equation}
\label{eq:b_n bis}
g^{(n)}:=\Phi_n * g, \quad  \Phi_n: = n^{2d} \Phi(n\, \cdot), \qquad n\in \N.
\end{equation}  

Therefore, under the assumption $b\in \Linf{\beta}$, we can set
\begin{equation}\label{eq:def_mollif_b}
b^{(n)}(t,\cdot) : = (b(t,\cdot))^{(n)}, \qquad t\in[0,T], \qquad  n\in \N.
\end{equation}

We recall the following result (\cite[Lemma 2.16]{issoglio2024degenerate}).
\begin{lemma}\label{lemm:moll_b}
Under assumption {\bf [H-$b$]}, we have the following properties:
\begin{itemize}
\item[(i)] $b^{(n)}\in \Linf{\gamma}$ for any $n\in\N$ and $\gamma>0$. In particular, $b^{(n)}\in L^{\infty}_T {\mathscr C}_{\bf o}^{0}$. 
\item[(ii)] For any $\eta < \beta$, there exists $C=C(
\beta,\eta)>0$ such that 
\begin{equation}\label{eq:bound_gn_g}
\|b^{(n)}-b\|_{\Linf{\eta}}
\leq C \|b\|_{\Linf{\beta}} 
, \qquad n\in\N,
\end{equation}
and in particular,
\begin{equation}
 \label{eq:conv sequence}
 \|b^{(n)} - b \|_{{\Linf{\eta}}} \longrightarrow 0, \qquad \text{as } n\to \infty.
\end{equation}
\end{itemize}
\end{lemma}

As a consequence of the previous lemma, we have the following
\begin{REM}\label{rem:solution_regularized_sde}
Under assumption {\bf [H]}, Lemma \ref{lemm:moll_b} implies that the coefficients of \eqref{sde-mollified} satisfy the assumptions in \cite{chau:meno:pesc:zhan:23}. Therefore, in light of Section \ref{sec:prel_nonsingular}, for any initial probability measure $\initlaw$ on $\Rdd$, and initial time $t_0\in [0,T]$, there exists a unique solution, $P^{(n)}_{t_0,\initlaw}$, to the martingale problem 
associated to \eqref{sde-mollified}, in the sense of Stroock-Varadhan, which is a Markov process. Furthermore, the transition probability function $P^{(n)}(s, \x , t, \d \y)$ has a density $\pn(s,\x,t,\y)$, for which the estimates in Proposition \ref{prop-HK-proxy-sde} are satisfied, with a constant $\lambda$ that is independent of $n$ and with constants $C$, $C_{\etabackward}$, $C_{\etaforward}$ that depend, a priori, on $n$. In the next theorem, we show that these constants are actually independent of $n$.
\end{REM}

\begin{THM}\label{thm-main-mollified}
Let assumption {\bf [H]} be in force. Let $n\in\N$ and denote by $p^{(n)}=p^{(n)}(s,\x,t,\y)$ the transition density of the solutions to \eqref{sde-mollified} (cf. Remark \ref{rem:solution_regularized_sde}).
For all $\etabackward\in (0,{1+\beta}),\ \etaforward\in  (0,\nu+\beta)$ and $\delta=0,1$, there exist four positive constants $\Cbackward:=\Cbackward(\Theta_{T,b},\etabackward)$, $\Cforward:=\Cforward(\Theta_{T,b},\etaforward) $, $C(\Theta_{T,b})$ and $\lambda:= \lambda(\Theta_{T,b})$, such that
%
{\mathtoolsset{showonlyrefs=false}\begin{align}%
\label{eq:THM-deg_HK_1_reg} 
C^{-1} p_{\gF,\lambda^{-1}}(s,\x,t,\y)   \leq \pn (s,\x, t, \y  ) &\leq C p_{\gF,\lambda}(s,\x,t,\y), \\
| \nabla_{\x_1} \pn(s, \x, t, \y  ) | & \leq \frac{C}{(t-s)^{\frac{1}{2}}} p_{\gF,\lambda}(s,\x,t,\y),\label{eq:THM-deg_HK_2_reg}\\
\left| \pn(s,\x,t,\y)-  \pn(s, \x_1 , \x_2',t,\y) \right| & \le \Cbackward \frac{|\x_2-\x_2'|^{\frac{1 + \etabackward}{3}}}{(t-s)^{\frac{1 + \etabackward}{2}}}  \Big(p_{\gF,\lambda}(s,\x,t,\y)+ p_{\gF,\lambda}(s,\x_1 , \x_2',t,\y) \Big),  \label{eq:THM-deg_HK_3_reg} \\
\left|\nabla_{\x_1}  \pn(s,\x,t,\y)-\nabla_{\x_1}  \pn(s,\x',t,\y) \right| & \le  \frac{\Cbackward}{(t-s)^{\frac{1}{2}}}\left(\frac{|\x-\x'|_{\bf d}}{(t-s)^{\frac{1}{2}}}\right)^{\etabackward} \Big(p_{\gF,\lambda}(s,\x,t,\y)+ p_{\gF,\lambda}(s, \x',t,\y) \Big), \label{eq:THM-deg_HK_4_reg} \\ 
\left|\nabla_{\x_1}^\delta \pn(s,\x,t,\y)-\nabla_{\x_1}^\delta \pn(s,\x,t,\y') \right| & \le\frac{\Cforward}{(t-s)^{\frac{\delta}{2}}}\left(\frac{|\y-\y'|_{\bf d}}{(t-s)^{\frac{1}{2}}}\right)^{\etaforward} \Big(p_{\gF,\lambda}(s,\x,t,\y)+ p_{\gF,\lambda}(s, \x,t,\y') \Big), \label{eq:THM-deg_HK_5_reg}
\end{align}}
for any $0\le s<t\le T$ and $ \x,\x',\y,\y'\in \R^{2d}$. Recall that $p_{\gF,\lambda}=p_{\gF,\lambda}(s,\x,t,\y)$ denotes transition density of the solutions to the proxy SDE \eqref{sde-proxy} with $\sigma \equiv \lambda \mathbb I_{d\times d}$. Note that all the constants $\Cbackward$, $\Cforward$, $C$ and $\lambda$ are independent of $n$.
\end{THM}

The proof of Theorem \ref{thm-main-mollified} is based on the following lemma, whose proof is deferred to Appendix \ref{app:proof_lemma_normal}.
In the lemma we assume that $\etaforward \in (-\beta,\nu+\beta) $. This restriction is due to the fact that we strongly rely on duality arguments which involve the drift $b^{(n)} $, which is regular, but which must be evaluated, in order to get uniform controls, in $\B_{\infty,\infty,\bd}^{-\beta} $. The assertion of Theorem \ref{thm-main-mollified} for the range $(0,-\beta] $ can then be easily derived from the previous controls and usual interpolation arguments.


	\begin{lemma}
	\label{lemma-besov-convo}
	Under the assumptions of Theorem \ref{thm-main-mollified}, 
for all  $\etaforward\in  (-\beta,\nu+\beta)$ and $\delta=0,1$, there exist two positive constants $C(\Theta_{T,b})$ and $\Cforward:=\Cforward(\Theta_{T,b},\etaforward) $, such that

	\begin{equation}\label{maj-convo-besov}
 \big\Vert \nabla_{1}^{\delta} \pn(0,\x,r,\cdot)\nabla_{1} p_{\gF,\sigma}(r,\cdot,t,\y)\big\Vert_{\B^{- \beta}_{1,1,\bf d}}  \leq C \, h_{\x}^{\etaforward,n}(r) \frac{ p_{\gF,2 \lambda}(0,\x,t,\y)}{r^{\frac{\delta}{2}} (t-r)^{\frac{1-\beta-\etaforward}{2}}}\bigg( \frac{1}{(t-r)^{\frac{\etaforward}{2}}}+\frac{1}{r^{\frac{\etaforward}{2}}}\bigg) 
\end{equation}
and
	\begin{align}
 &\big\Vert \nabla_{1}^{\delta} \pn(0,\x,r,\cdot)\big(\nabla_{1} p_{\gF,\sigma}(r,\cdot,t,\y)-\nabla_{1} p_{\gF,\sigma}(r,\cdot,t,\y')\big)\big\Vert_{\B^{- \beta}_{1,1,\bf d}}  \\
				&\qquad \leq \Cforward\, h_{\x}^{\etaforward,n}(r)\big(p_{\gF,2 \lambda}(0,\x,t,\y)+p_{\gF,2 \lambda}(0,\x,t,\y')\big)\frac{1}{r^{\frac{\delta}{2}}}\frac{|\y-\y'|_{\bf d}^{\etaforward}}{(t-r)^{\frac{1-\beta}{2}}}\bigg( \frac{1}{(t-r)^{\frac{\etaforward}{2}}}+\frac{1}{r^{\frac{\etaforward}{2}}}\bigg) 
				, \label{maj-convo-besov-2}
	\end{align}
for any $0< r <t\leq T$, $\x , \y\in\Rdd$ and $\delta=0,1$, where
	\begin{equation}
	h_{\x}^{\etaforward,n}(r) := \sup_{\y,\y'\in \Rdd} \bigg\{ \frac{\pn(0,\x,r,\y)}{p_{\gF, 2 \lambda}(0,\x,r,\y)} +\frac{|\pn(0,\x,r,\y)-\pn(0,\x,r,\y')|}{p_{\gF, 2 \lambda}(0,\x,r,\y)+p_{\gF, 2 \lambda}(0,\x,r,\y')} \frac{r^{\frac{\etaforward}{2}}}{|\y-\y'|^{\etaforward}_{\bf d}}\bigg\}\label{THE_NOR},
	\end{equation}
and where $\lambda=\lambda(\Theta_T)$ is the positive constant in Proposition \ref{prop-HK-proxy-sde}.
	\end{lemma}
	
We are now in the position to prove Theorem \ref{thm-main-mollified}.

\begin{proof}[Proof of Theorem \ref{thm-main-mollified}]
We prove the statement for $s=0$. This is not restrictive, up to shifting in time all the coefficients. The strategy we implement is the following:
\begin{itemize}
\item \underline{Step 1}: prove estimates \eqref{eq:THM-deg_HK_1_reg}-\eqref{eq:THM-deg_HK_5_reg} with positive constants $C$, $C_{\etabackward}$, $C_{\etaforward}$ that also depend on $n$, and with $\lambda=\lambda(\Theta_T)$ (independent of $n$) being the same concentration constant appearing in Proposition \ref{prop-HK-proxy-sde}. Notice that the independence of the concentration constant on the mollification parameter is crucial in order for the quantity $h_{\x}^{\etaforward,n}(r)$ defined in \eqref{THE_NOR} to be finite (because of the Gaussian tails).
\item \underline{Step 2}: prove that $C$, $C_{\etabackward}$, $C_{\etaforward}$ can be made independent of $n$, up to raising the concentration constant $\lambda$.
\end{itemize}

\underline{Step 1:} for this step we only provide a sketch of the proof, as the technique we rely on is quite standard. We employ a full parametrix expansion around the kernel (parametrix) $p_{\gF,\sigma}(0,\x,t,\y)$. Recalling Remark \ref{rem:solution_regularized_sde}, Duhamel representation yields:
\begin{equation}\label{eq:rep_par_full}
\pn(0,\x,t,\y)=p_{\gF,\sigma}(0,\x,t,\y) + \int_0^t \int_{\Rdd} p_{\gF,\sigma} (0,\x,r,\z) \Phi^{(n)}(r,\z,t,\y) \d \z \d r,
\end{equation}
for any $\x,\y\in\Rdd$ and $0<t\leq T$, with
\begin{equation}
\Phi^{(n)}(r,\z,t,\y) = \sum_{k\in\N_0} \varphi_k^{(n)}(r,\z,t,\y) ,
\end{equation}
and where
\begin{align}
\varphi_0^{(n)}(r,\z,t,\y) &= b^{(n)}(r,\z)\nabla_1 p_{\gF,\sigma}(r,\z,t,\y), \\
\varphi_k^{(n)}(r,\z,t,\y) &= \int_r^{t} \int_{\Rdd}  b^{(n)}(r,\z)\nabla_1 p_{\gF,\sigma}(r,\z,\tau,\bxi) \varphi_{k-1}^{(n)}(\tau,\bxi,t,\y) \d \bxi \d \tau, \qquad k\in\N.
\end{align}
By \eqref{HK-gradient-back} we directly obtain
\begin{equation}
|\varphi_0^{(n)}(r,\z,t,\y)| \leq \underbrace{C \| b^{(n)} \|_{\infty}}_{=: K_n} (t-r)^{-\frac{1}{2}} p_{\gF,\lambda}(r,\z,t,\y),
\end{equation}
and, again by \eqref{HK-gradient-back} together with Chapman-Kolmogorov identity, we have
\begin{equation}
|\varphi_1^{(n)}(r,\z,t,\y)| \leq \frac{K_n^{2}}{\Gamma_{\text{Euler}}(1)} p_{\gF,\lambda}(r,\z,t,\y).
\end{equation}
Iterating in $k$ yields
\begin{equation}\label{eq:bound_Psi}
|\varphi_k^{(n)}(r,\z,t,\y)| \leq \frac{K_n^{k+1}}{\Gamma_{\text{Euler}}((1+k)/2)} (t-r)^{\frac{k-1}{2}} p_{\gF,\lambda}(r,\z,t,\y), \qquad k\in\N,
\end{equation}
and thus
\begin{equation}\label{eq:bound_Psi_sens}
|\Phi^{(n)}(r,\z,t,\y)| \leq \tilde K_n (t-r)^{-\frac{1}{2}}p_{\gF,\lambda}(r,\z,t,\y).
\end{equation}
Nearly identical arguments, employing also \eqref{HK-holder-gradient-forward} with $\delta=1$, yields
\begin{equation}
|\Phi^{(n)}(r,\z,t,\y) - \Phi^{(n)}(r,\z,t,\y') | \leq \tilde K_n (t-r)^{-\frac{1}{2}} \left(\frac{|\y-\y'|_{\bf d}}{(t-r)^{\frac{1}{2}}}\right)^{\etaforward} \big( p_{\gF,\lambda}(r,\z,t,\y)+ p_{\gF,\lambda}(r,\z,t,\y') \big).
\end{equation}
Plugging estimates \eqref{eq:bound_Psi}-\eqref{eq:bound_Psi_sens} into \eqref{eq:rep_par_full}, and applying \eqref{HK-density_upper}-\eqref{HK-holder-gradient-forward} and Chapman Kolmogorov identity, yield estimates \eqref{eq:THM-deg_HK_1_reg}-\eqref{eq:THM-deg_HK_5_reg}, with $C$, $C_{\etabackward}$, $C_{\etaforward}$ that depend on $n$, and with $\lambda=\lambda(\Theta_T)$ (independent of $n$) being the same concentration constant appearing in Proposition \ref{prop-HK-proxy-sde}.

\underline{Step 2:} Recalling Remark \ref{rem:solution_regularized_sde}, Duhamel formula applies to the transition density $\pn$. Namely, we can write:
	\begin{equation}
		\pn(0,\x,t,\y)=p_{\gF,\sigma}(0,\x,t,\y) + \int_0^t \int_{\Rdd} \pn (0,\x,r,\z)b^{(n)}(r,\z)\nabla_1 p_{\gF,\sigma}(r,\z,t,\y) \d \z \d r.\label{duhamel-mollified-sde}
	\end{equation}
for any $\x,\y \in  \Rdd$ and $0 <t \leq T$. By \eqref{duhamel-mollified-sde} and by the duality bound in Proposition \ref{PROP_DUALITY}, 
	we have:
	\begin{align}
		\pn(0,\x,t,\y)&\leq p_{\gF,\sigma}(0,\x,t,\y)+\int_0^t \Vert b^{(n)} (r,\cdot) \big\Vert_{\B^{\beta}_{\infty,\infty,\bf d}} \Vert \pn (0,\x,r,\cdot)\nabla_1 p_{\gF,\sigma}(r,\cdot,t,\y) \big\Vert_{\B^{-\beta}_{1,1,\bf d}} \d r.\label{dual-ineq-on-duhamel}
	\end{align}
	
	Applying \eqref{HK-density_upper} and \eqref{maj-convo-besov}, with $\delta=0$, to \eqref{dual-ineq-on-duhamel} yields
	\begin{equation}\label{eq:bound_p_pn}
\frac{\pn(0,\x,t,\y)}{p_{\gF,2 \lambda}(0,\x,t,\y)} \lesssim 1+
\int_0^t 
h_{\x}^{\etaforward,n}(r)\frac{1}{(t-r)^{\frac{1-\beta - \etaforward}{2}}}\bigg( \frac{1}{(t-r)^{\frac{\etaforward}{2}}}+\frac{1}{r^{\frac{\etaforward}{2}}}\bigg)  \d r.
	\end{equation}
Now, by applying  once more the Duhamel representation \eqref{duhamel-mollified-sde} and the duality bound in Proposition \ref{PROP_DUALITY}, we obtain
	\begin{align}
		\nonumber &|\pn(0,\x,t,\y)-\pn(0,\x,t,\y')|\leq |p_{\gF,\sigma}(0,\x,t,\y)-p_{\gF,\sigma}(0,\x,t,\y')|\\
		&\qquad \qquad +\int_0^t \big\Vert b^{(n)} (r,\cdot) \big\Vert_{\B^{\beta}_{\infty,\infty,\bf o}} \big\Vert \pn (0,\x,r,\cdot) \big(\nabla_{1} p_{\gF,\sigma}(r,\cdot,t,\y)-\nabla_{1} p_{\gF,\sigma}(r,\cdot,t,\y')\big) \big\Vert_{\B^{-\beta}_{1,1,\bf d}} \d r.
		\label{dual-ineq-on-duhamel-holder-modulus}
	\end{align}
	Applying \eqref{HK-holder-gradient-forward} and \eqref{maj-convo-besov-2} with $\delta=0$, and once more \eqref{HK-density_upper}, yields
\begin{equation}
\frac{|\pn(0,\x,t,\y)-\pn(0,\x,t,\y')|}{p_{\gF, 2 \lambda}(0,\x,t,\y)+p_{\gF,2 \lambda}(0,\x,t,\y')} \frac{t^{\frac{\etaforward}{2}}}{|\y-\y'|^{\etaforward}_{\bf d}}  \lesssim 1 +t^{\frac{\etaforward }{2}}\int_0^t  h_{\x}^{\etaforward,n}(r)\frac{1}{(t-r)^{\frac{1-\beta}{2}}}\bigg( \frac{1}{(t-r)^{\frac{\etaforward}{2}}}+\frac{1}{r^{\frac{\etaforward}{2}}}\bigg)  \d r,
\end{equation}	
which in turn, together with \eqref{eq:bound_p_pn}, yields
	\begin{equation}
		h_{\x}^{\etaforward,n}(t)\lesssim1 +t^{\frac{\etaforward }{2}}\int_0^t  h_{\x}^{\etaforward,n}(r)\frac{1}{(t-r)^{\frac{1-\beta}{2}}}\bigg( \frac{1}{(t-r)^{\frac{\etaforward}{2}}}+\frac{1}{r^{\frac{\etaforward}{2}}}\bigg)  \d r.
	\end{equation}	
By what we proved in Step 1, the quantity $h_{\x}^{\etaforward,n}(t)$ is finite for any $n\in\N$. 
Furthermore, owing to $\etaforward\in (-\beta, \nu+\beta)$ and to \eqref{eq:assump_consts}, we have	
	\begin{equation}
 \etaforward -\beta < (\nu + \beta) -\beta < 1  .
\end{equation}
Thus we can employ a Gronwall-Volterra lemma, see e.g. see e.g. Lemma 2.6 in \cite{jour:meno:24}, and finally obtain
	\begin{equation}
		h_{\x}^{\etaforward,n}(t)\lesssim 1.
	\end{equation}
This proves \eqref{eq:THM-deg_HK_1_reg} and \eqref{eq:THM-deg_HK_5_reg} for $\delta=0$. To obtain \eqref{eq:THM-deg_HK_2_reg} and \eqref{eq:THM-deg_HK_5_reg} for $\delta=1$, one simply needs to apply the same proof to bound the term
	\begin{equation}
	\sup_{\y,\y'\in \Rdd} \bigg\{ \frac{\nabla_{\x_1}\pn(0,\x,r,\y)}{p_{\gF, 2 \lambda}(0,\x,r,\y)}r^{\frac 12} +\frac{|\nabla_{\x_1} \pn(0,\x,r,\y)-\nabla_{\x_1} \pn(0,\x,r,\y')|}{p_{\gF, 2 \lambda}(0,\x,r,\y)+p_{\gF, 2 \lambda}(0,\x,r,\y')} \frac{r^{\frac{1+\etaforward}{2}}}{|\y-\y'|^{\etaforward}_{\bf d}}\bigg\}\label{THE_NOR_bis}.
	\end{equation}
This can be done since the additional time normalization for the gradient still gives integrable singularities in $r^{-\frac{1+\etaforward}2}$. 
A similar proof yields \eqref{eq:THM-deg_HK_3_reg} and \eqref{eq:THM-deg_HK_4_reg}. Pay anyhow attention that for these estimates the forward regularity is still needed, because of the singular drift. One can e.g. use Lemma \ref{lemma-besov-convo} with the somehow \textit{minimal regularity} $\etaforward=-\beta+\varepsilon,\ \varepsilon>0 $.  We omit the details for the sake of brevity.
\end{proof}

\section{The Kolmogorov Cauchy problem}\label{CAUCHY_PB}


Once again, we focus here on the degenerate case, the non-degenerare one being a particular case. The objective of this section is to study the (backward) Kolmogorov Cauchy problem
\begin{equation}\label{eq:kolm_eq}
	\begin{cases}
		\Lc u  = 
		g \qquad \text{on } (0,t) \times \R^{2d}, \\
		u_t = \ell,
	\end{cases}
\end{equation}
for a given $t\in (0,T]$, where the Kolmogorov operator $\Lc$ is defined as in \eqref{eq:op_L}-\eqref{gen}. 
The study of the mild solutions to \eqref{eq:kolm_eq}, together with the corresponding Schauder-type estimates, have been previously performed by several authors in several settings. We refer e.g. to \cite{prio:06}, \cite{chau:hono:meno:21}, \cite{lucertini2023optimal}, \cite{issoglio2022pde} among others. To the best of our knowledge, in the context Kolmogorov PDEs with distributional drift, the results of this section are original as the $\Kc$ component of $\Lc$, i.e. operator \eqref{gen}, has variable, non-smooth, coefficients, in particular in the diffusion terms. This appears to be novel even in the non-degenerate setting, namely when $\Lc$ is the generator of the solution to \eqref{eq:sde_kinetic_systems_non_deg}. 

\subsection{The semigroup of $\Kc$}
To define mild solutions to \eqref{eq:kolm_eq}, we first need to define the semigroup of the (non-singular) Kolmogorov operator $\Kc$,  hereafter $P^{{\bf F},\sigma}_{s,t}$, acting on a distribution in ${\mathcal{C}^{-\nu-\beta^+}_{\bd}\!}$ and to study its smoothing properties. When $\pFsig(s, \cdot; t, \cdot  )$ is smooth, we can define $P^{{\bf F},\sigma}_{s,t}$ by duality. Here, as $\pFsig(s, \cdot; t, \cdot  )$ is only H\"older continuous, we define it by approximation.

For any $0\leq s < t\leq T$, in light of Proposition \ref{prop-HK-proxy-sde}, we can first define $P^{{\bf F},\sigma}_{s,t} :  \mathcal{S} \longrightarrow \mathcal{C}^{2-\eps}_{\bd}$ 
, for any $\eps>0$ suitably small, as:
\begin{equation}
	P^{{\bf F},\sigma}_{s,t}  \varphi := \int_{\R^{2d}}  \pFsig(s, \cdot; t, \y  )  \varphi(\y) \d \y, \qquad \varphi\in\mathcal{S},
\end{equation}
and then extend it to $P^{{\bf F},\sigma}_{s,t} :  {\mathcal{C}^{-\nu-\beta^+}_{\bd}\!} \to  \mathcal{C}^{2-\eps}_{\bd}$ by means of the following

\begin{PROP}\label{prop:limit_semigroup_F_sigma}
	Let assumption {\bf [H]} be in force. For any $0\leq s < t\leq T$ and {$\varphi\in \mathcal{C}_{\bf d}^{\alpha}$ with $\alpha>-(\nu+\beta)$}, there exists $l\in 
	\bigcap_{\delta>0}\mathcal{C}^{2-\delta}_{\bd}$ such that 
	\begin{equation}
		\| P^{{\bf F},\sigma}_{s,t}  \varphi_n - l  \|_{\mathcal{C}_{\bf d}^{2-\delta}} \longrightarrow 0 \qquad \text{as}\quad n\to \infty,
	\end{equation} 
	for any $\delta>0$ and for any sequence $(\varphi_n)_{n\in\N}$ of functions in $\mathcal{S}$ such that $\varphi_n \to \varphi$ in $\mathcal{C}_{\bf d}^{\eta}$ with $\eta\in(-\nu-\beta,\alpha)$. In particular, $l$ does not depend on either $\delta$ or the choice of the sequence $\varphi_n$ {(nor on $\eta$)}.
\end{PROP}

To prove Proposition \ref{prop:limit_semigroup_F_sigma}, we need the following
\begin{lemma}\label{lem:limit_semigroup_F_sigma}
	Let assumption {\bf [H]} be in force. For any $\eta\in (-\nu-\beta, 2)$ and $\gamma\in(\eta,2)$, there exists a positive constant $C=C(\Theta_{T}, \gamma, \nu+\beta-\eta)$, non-decreasing in $\nu+\beta-\eta$, such that
	\begin{equation}\label{eq:ste_cauc_bound}
\| P^{{\bf F},\sigma}_{s,t}  \varphi  \|_{\mathcal{C}_{\bf d}^{\gamma}} \leq C  \| \varphi  \|_{\mathcal{C}^{\eta}_{\bd}} \frac{(t-s)^{-\frac{\gamma-\eta}{2}}}{\nu+\beta+\eta}, \qquad 0\leq s < t\leq T, \quad \varphi\in \mathcal{S}.
\end{equation}
\end{lemma}
\begin{proof} We only show the case $\eta<0$ and $\gamma = 2 - \delta$ with $\delta
\in(0,1)$, the other cases being easier.
 Throughout this proof, we will denote by $C$ and $\tilde \kappa$ any positive constant greater than $1$
 , depending at most on $\Theta_{T}$, $\delta$, $\nu+\beta-\eta$, non-decreasing in the latter quantity.
Recalling the thermic characterization of Besov spaces in Definition \ref{DEF_BESOV_THERMIC_ANISO}, we have
\begin{align}
\mathcal{T}_{\infty,\infty}^{2-\delta} [ P^{{\bf F},\sigma}_{s,t}  \varphi  ]&  = \sup_{v\in (0,1]} v^{1-\frac{2- \delta}{2}} \sup_{\z\in\Rdd} \Big|  \int_{\Rdd}  \partial_v g (v,\z-\x) \int_{\Rdd} \pFsig(s, \x, t, \y  )  \varphi (\y) \d \y \, \d \x    \Big| \\
&=  \sup_{v\in (0,1]} v^{1-\frac{2-\delta}{2}} \sup_{\z\in\Rdd} \Big|  \int_{\Rdd}  \varphi (\y)  \int_{\Rdd} \partial_v g (v,\z-\x) \pFsig(s, \x, t, \y  )  \d \x \, \d \y    \Big| \\
&\leq \| \varphi \|_{\mathcal{C}^{\eta}_{\bd}} \sup_{v\in (0,1]} v^{1-\frac{2-\delta}{2}} \sup_{\z\in\Rdd} \Big\|   \int_{\Rdd}  \partial_v g (v,\z-\x)    \pFsig(s, \x, t, \cdot  )   \d \x  \Big\|_{\B_{\bd,1,1}^{-\eta}},\label{eq:ste_bound_cauch}
\end{align}
using Proposition \ref{PROP_DUALITY} for the last inequality.
Now observe that
\begin{equation}
\mathcal{T}_{1,1}^{-\eta} \Big[  \int_{\Rdd}  \partial_v g (v,\z-\x)    \pFsig(s, \x, t, \cdot  )   \d \x \Big] =   \int_0^1   \tilde v^{\frac{\eta}{2}} \int_{\Rdd}   \big| I(v, \tilde v , \z , \tilde \z) \big| \,  \d \tilde\z  \,  \d\tilde v,
\end{equation}
with
\begin{align}
I(v, \tilde v , \z , \tilde \z) :& = \int_{\Rdd} \partial_{\tilde v} g (\tilde v,
\tilde\z-\y)  \int_{\Rdd} \partial_v g (v,\z-\x)   \pFsig(s, \x, t, \y  ) \, \d\x \, \d \y.
\label{eq:ste_I}
\end{align}
Set now
\begin{align}
J (s, \x , &\z, t, \y ):=
\pFsig(s, \x, t, \y ) - \pFsig(s, \z, t, \y ) 
- \langle \nabla_{\z_1} \pFsig(s, \z, t, \y ) , \x_1 -  \z_1 \rangle.
\end{align}
Observing that 
\begin{equation}
\int_{\Rdd} \partial_v g (v,\z-\x)\Big(\pFsig(s, \z, t, \y ) 
- \langle \nabla_{\z_1} \pFsig(s, \z, t, \y ) , \x_1 -  \z_1 \rangle \Big)\d \x=0 ,
\end{equation}
one obtains
\begin{align}
\int_{\Rdd} \partial_{ v} g ( v,
\z-\x)\pFsig(s, {\x}, t, \y )\d\x= \int_{\Rdd} \partial_v g (v,\z-\x)J (s, \x , \z, t, \y )   \, \d\x.
\end{align}
Write now:
\begin{align}
J (s, \x , \z, t, \y )&=
\pFsig(s, \x, t, \y )-\pFsig(s, \x_1,\z_2, t, \y )\\
&\quad +\pFsig(s, \x_1,\z_2, t, \y ) - \pFsig(s, \z, t, \y ) 
- \langle \nabla_{\z_1} \pFsig(s, \z, t, \y ) , \x_1 -  \z_1 \rangle\\
&=\pFsig(s, \x, t, \y )-\pFsig(s, \x_1,\z_2, t, \y )\\
&\quad +\int_0^1   \langle \nabla_1\pFsig(s, \z_1+\lambda (\x_1-\z_1),\z_2, t, \y ) - \nabla_{1} \pFsig(s, \z, t, \y ) , \x_1 -  \z_1 \rangle \d \lambda.
\end{align}
From the expression of $I(v, \tilde v , \z , \tilde \z)$, we have:
\begin{align*}
I(v, \tilde v , \z , \tilde \z) :& = \int_{\Rdd} \partial_{\tilde v} g (\tilde v,
\tilde\z-\y)  \int_{\Rdd} \partial_v g (v,\z-\x)   \tilde J(s, \x,\z, t, \y,\tilde \z  ) \, \d\x \, \d \y,
\end{align*}
where
\begin{align}
\tilde J(s, \x,\z, t, \y,\tilde \z  )&=\Big(\pFsig(s, \x, t, \y )-\pFsig(s, \x_1,\z_2, t, \y )-(\pFsig(s, \x, t, \tilde \z )-\pFsig(s, \x_1,\z_2, t, \tilde \z ))\Big)\\
&\quad+\Big(\int_0^1 \langle \nabla_1\pFsig(s, \z_1+\lambda (\x_1-\z_1),\z_2, t, \y ) - \nabla_{1} \pFsig(s, \z, t, \y ) , \x_1 -  \z_1 \rangle \d \lambda \\
&\quad-\int_0^1 \langle \nabla_1\pFsig(s, \z_1+\lambda (\x_1-\z_1),\z_2, t, \tilde \z ) - \nabla_{1} \pFsig(s, \z, t, \tilde \z ) , \x_1 -  \z_1 \rangle \d \lambda \Big) \\
&=:(\tilde J_1+\tilde J_2)(s, \x,\z, t, \y,\tilde \z  ).
\end{align}
Observe that the terms with final variables $\tilde \z$ have been added thanks to a cancellation argument with respect to the integration variable $\y$ and Fubini Theorem (double integral associated with the backward and forward variables of the density initially considered).

We can now invoke \eqref{eq:ste_M_DEG_new} and \eqref{eq:ste_new} for the  terms $\tilde J_1 $ and $\tilde J_2 $, respectively, with
$\etabackward = 1-\delta, \etaforward = \frac{\nu+\beta-\eta}{2} =: \eps$.
 One gets
\begin{align}
|\tilde J (s, \x , \z, t, \y,\tilde \z ) | & \leq C (t-s)^{-\frac{2-\delta-\eta}{2}} | \x -  \z|_{\bd}^{2-\delta}|\tilde\z-\y|_{\bd}^{\eps}\\ 
&\quad\times \int_0^1 \Big( g^{\bf d}_{ \tilde \kappa}(t-s, \tilde \btheta_{t,s}(\z_1+\lambda(\x_1-\z_1),\z_2) - \y) + g^{\bf d}_{\tilde  \kappa}(t-s, \tilde \btheta_{t,s}(\z) - \y) \\
&\qquad\qquad+ g^{\bf d}_{ \tilde \kappa}(t-s, \tilde \btheta_{t,s}(\z_1+\lambda(\x_1-\z_1),\z_2) - \tilde\z) + g^{\bf d}_{ \tilde \kappa}(t-s, \tilde \btheta_{t,s}(\z) - \tilde\z) \Big) \d \lambda .\qquad \label{eq:ste_estim_J_bis}
\end{align}

Bound \eqref{eq:ste_estim_J_bis}, together with the Gaussian bounds \eqref{BD_THERM_DEG_DER_TEMP}, now implies
\begin{align}
&|I (v, \tilde v , \z , \tilde \z)|  \leq C (t-s)^{-\frac{2-\delta}{2}} v^{\frac{2-\delta}{2}-1}\tilde v^{-1+\frac{\eps}{2}}  \int_{\Rdd} g^{\bf d}_{\tilde \kappa} (\tilde v,
\tilde\z-\y) \int_{\Rdd} g^{\bf d}_{\tilde \kappa} (v,\z-\x)  \, \d\x \, \d \y \\
&\, \times \int_0^1 \d \lambda \big( g^{\bf d}_{\tilde \kappa}(t-s, \tilde \btheta_{t,s}(\z_1+\lambda(\x_1-\z_1),\z_2) - \y) + g^{\bf d}_{ \tilde \kappa}(t-s, \tilde \btheta_{t,s}(\z) - \y) \\
&\quad + g^{\bf d}_{\tilde \kappa}(t-s, \tilde \btheta_{t,s}(\z_1+\lambda(\x_1-\z_1),\z_2) - \tilde\z) + g^{\bf d}_{\tilde \kappa}(t-s, \tilde \btheta_{t,s}(\z) - \tilde\z) \big) ,
\end{align}
which in turn, by integrating the Gaussian densities (using Fubini Theorem), yields
\begin{equation}
\mathcal{T}_{1,1}^{-\eta} \Big[  \int_{\Rdd}  \partial_v g (v,\z-\x)    \pFsig(s, \x, t, \cdot  )   \d \x \Big]   \leq C (t-s)^{-\frac{2-\delta}{2}} v^{\frac{2-\delta}{2}-1}   \int_{0}^1   \tilde v^{\frac{\eta+\eps}{2}-1} \,  \d\tilde v  \leq C \frac{(t-s)^{-\frac{2-\delta-\eta}{2}} v^{\frac{2-\delta}{2}-1}}{\nu+\beta+\eta}.
\end{equation}
The same bound can be obtained for $ \big\Vert g(1,\cdot) *  \int_{\Rdd}  \partial_v g (v,\z-\x)    \pFsig(s, \x, t, \cdot  )   \d \x \big\Vert_{L^1}$. These, together with \eqref{eq:ste_bound_cauch}, yields
\begin{equation}
\mathcal{T}_{\infty,\infty}^{2-\delta} [ P^{{\bf F},\sigma}_{s,t}  \varphi ] \leq C  \| \varphi  \|_{\mathcal{C}^{\eta}_{\bd}} \frac{(t-s)^{-\frac{2-\delta-\eta}{2}}}{\nu+\beta+\eta}.
\end{equation}
The same bound can be established for $ \big\Vert g(1,\cdot) * \big(  P^{{\bf F},\sigma}_{s,t}  \varphi\big)\big\Vert_{L^{\infty}}$. This concludes the proof of \eqref{eq:ste_cauc_bound}.
\end{proof}

\begin{proof}[Proof of Proposition \ref{prop:limit_semigroup_F_sigma}]
Fix $\eta\in(-\nu-\beta,\alpha)$ and $\delta>0$ suitably small. By Lemma \ref{lem:limit_semigroup_F_sigma}, we have
\begin{equation}\label{eq:ste_cauc_bound_2}
\| P^{{\bf F},\sigma}_{s,t}  \varphi_n - P^{{\bf F},\sigma}_{s,t}  \varphi_m \|_{\mathcal{C}_{\bf d}^{2-\delta}} \leq C  \| \varphi_n -  \varphi_m \|_{\mathcal{C}^{\eta}_{\bd}} \frac{(t-s)^{-\frac{2-\delta-\eta}{2}}}{\nu+\beta+\eta}.
\end{equation}
As $\varphi_n$ converges in $\mathcal{C}^{\eta}_{\bd}$, $P^{{\bf F},\sigma}_{s,t}  \varphi_n$ is a Cauchy sequence in $\mathcal{C}^{2-\delta}_{\bd}$. Therefore, $P^{{\bf F},\sigma}_{s,t}  \varphi_n$ converges to $l\in \mathcal{C}^{2-\delta}_{\bd}$.

As $\mathcal{C}^{2-\delta}_{\bd}$ is an increasing family in $\delta$, the limit $l$ does not depend on $\delta$ and belongs to $\bigcap_{\delta>0}\mathcal{C}^{2-\delta}_{\bd}$. 

To conclude, we show that the limit does not depend on $\eta$, nor on the choice of the sequence $(\varphi_n)_{n\in\N}$. Let $(\tilde\varphi_n)_{n\in\N}$ be a second sequence in $\mathcal{S}$ that converges to $\varphi$ in $\mathcal{C}^{\tilde\eta}_{\bd}$, with $\tilde\eta \in(\eta, \alpha)$. We have
\begin{equation}
\| l - P^{{\bf F},\sigma}_{s,t}  \tilde\varphi_n \|_{\mathcal{C}_{\bf d}^{2-\delta}} \leq \| l - P^{{\bf F},\sigma}_{s,t}  \varphi_n \|_{\mathcal{C}_{\bf d}^{2-\delta}} + \|  P^{{\bf F},\sigma}_{s,t}  \varphi_n - P^{{\bf F},\sigma}_{s,t}  \tilde\varphi_n \|_{\mathcal{C}_{\bf d}^{2-\delta}}.
\end{equation}
The first term on the right-hand side tends to zero by definition of $l$. As for the second term, Lemma \ref{lem:limit_semigroup_F_sigma} yields
\begin{equation}
\|  P^{{\bf F},\sigma}_{s,t}  \varphi_n - P^{{\bf F},\sigma}_{s,t}  \tilde\varphi_n \|_{\mathcal{C}_{\bf d}^{2-\delta}} \leq C  \| \varphi_n -  \tilde\varphi_n \|_{\mathcal{C}^{\eta}_{\bd}} \frac{(t-s)^{-\frac{2-\delta-\eta}{2}}}{\nu+\beta+\eta}.
\end{equation}
which tends to zero because both $\varphi_n$ and $\tilde\varphi_n$ tend to $\varphi$ in $\mathcal{C}^{\eta}_{\bd}$. 
\end{proof}
Hereafter, for any $0\leq s < t\leq T$ and any {$\alpha>-(\nu+\beta)$}, we shall define 
\begin{equation}
	P^{{\bf F},\sigma}_{s,t}:  {\mathcal{C}^{\alpha}_{\bd}} \to  \bigcap_{\eps>0}\mathcal{C}^{2-\eps}_{\bd}, \qquad P^{{\bf F},\sigma}_{s,t}\varphi := l,
\end{equation}
with $l$ being the limit in Proposition \ref{prop:limit_semigroup_F_sigma}. 

The following Schauder-type estimates hold for the semigroup $P^{{\bf F},\sigma}$.
\begin{PROP}[Schauder estimates for $P^{{\bf F},\sigma}$]\label{prop:schauder}
	Let assumption {\bf [H]} be in force. For any $\alpha\in ({-\nu-\beta},2)$ and $\gamma\in (\alpha,2)$, there exists a positive constant {$C=C(\Theta_T, \alpha, \gamma)$} such that
	\begin{equation}\label{eq:ste_schauder}
		\| P^{{\bf F},\sigma}_{s,t} \varphi  \|_{\mathcal{C}_{\bf d}^{\gamma}} \leq C (t-s)^{-\frac{\gamma-\alpha}{2}} \| \varphi \|_{\mathcal{C}_{\bf d}^{\alpha}},
		\qquad 0\leq s < t\leq T, \quad \varphi\in \mathcal{C}_{\bf d}^{\alpha}.
	\end{equation}
\end{PROP}
\begin{proof}
By Proposition \ref{prop:limit_semigroup_F_sigma}, we have
\begin{equation}
\| P^{{\bf F},\sigma}_{s,t} \varphi  \|_{\mathcal{C}_{\bf d}^{\gamma}} = \lim_{n\to+\infty}\| P^{{\bf F},\sigma}_{s,t} \varphi_n  \|_{\mathcal{C}_{\bf d}^{\gamma}},
\end{equation} 
for any sequence $(\varphi_n)_{n\in\N}$ converging to $\varphi$ in $\mathcal{C}_{\bf d}^{\eta}$, with $\eta \in ({-\nu-\beta},\alpha)$.
The proof is already contained in the proof of Proposition \ref{prop:limit_semigroup_F_sigma}, up to replacing $2-\delta$ 
with $\gamma$. However, by Lemma \ref{lem:limit_semigroup_F_sigma}, we also have
\begin{equation}
\| P^{{\bf F},\sigma}_{s,t} \varphi_n  \|_{\mathcal{C}_{\bf d}^{\gamma}} \leq C  \| \varphi_n  \|_{\mathcal{C}^{\eta}_{\bd}} \frac{(t-s)^{-\frac{\gamma-\eta}{2}}}{\nu+\beta+\eta} ,
\end{equation}
and taking the limite as $n\to \infty$ yields
\begin{equation}
\| P^{{\bf F},\sigma}_{s,t} \varphi  \|_{\mathcal{C}_{\bf d}^{\gamma}} \leq  C  \| \varphi  \|_{\mathcal{C}^{\eta}_{\bd}} \frac{(t-s)^{-\frac{\gamma-\eta}{2}}}{\nu+\beta+\eta}.
\end{equation}
Finally, \eqref{eq:ste_schauder} follows from taking the limit as $\eta\to \alpha$, and recalling that the constant $C=C(\Theta_{T}, \gamma, \nu+\beta-\eta)$ is non-decreasing in $\nu+\beta-\eta$.
\end{proof}

\subsection{Well-posedness, regularity estimates and stability properties of the Cauchy problem} With the semigroup $P^{{\bf F},\sigma}_{s,t}$ at hand, we can now give a definition of mild solution to \eqref{eq:kolm_eq}, and then prove the well-posedness of the latter together with some global Schauder etimates. We introduce the following notation, for any $t>s\geq 0$ and $\gamma\in\R$:
\begin{itemize}[leftmargin=10pt]
\item[-] $\mathcal{C}_b$: set of bounded and continuous scalar-valued functions on $\Rdd$, equipped with the topology of the uniform convergence.
\item[-] $C_t\mathcal{C}_b$: continuous functions on $[0,t]$ with values on $\mathcal{C}_b$.
\item[-] $\Linfb{\gamma}{s}{t}$: measurable and (essentially) bounded functions on $[s,t]$ with values on $\mathcal{C}^{\gamma}_{\bf d}$. In particular, $\Linfa{\gamma}{t}:=\Linfb{\gamma}{0}{t}$, in accordance with the notation introduced in Section \ref{sec:sub_mainresults}.
\end{itemize}
\begin{df}[Mild solution to the Kolmogorov Cauchy problem]\label{def:sol_mild_backward_CP}
	Let assumption {\bf [H]} be in force. Let also $t\in(0,T]$, $g\in \Linfa{\beta}{t}$ 
	and $\ell \in {\mathcal{C}_{\bf d}^{2+\beta}}$.
	A 
	function $u: [0,t]\times \R^{2d} \to \R$ is said a mild solution to \eqref{eq:kolm_eq} if:
	\begin{itemize}
		\item[(a)] $u\in C_t\mathcal{C}_b \cap L_t^\infty\mathcal{C}^\gamma_{\bf d}$ for any 
		{$\gamma<2+\beta$};
		\item[(b)] the following Duhamel-type representation holds:
		\begin{equation}
			u(s,\x) =   P^{{\bf F},\sigma}_{{s,t}} \ell (\x)+ \int_{s}^t P^{{\bf F},\sigma}_{s,r} \big[ \langle b(r,\cdot) , \nabla_{1} u(r, \cdot ) \rangle 
			 - g(r,\cdot)\big](\x) \d r, \qquad (s,\x) \in [0,t]\times \R^{2d}.
		\end{equation}
	\end{itemize}
\end{df}
\begin{THM}[Well-posedness and Schauder estimates for the Kolmogorov Cauchy problem]\label{th:backward_CP_wellposedenss}
	\label{MTHM_cauchy} Under the assumptions of Definition \ref{def:sol_mild_backward_CP}, 
	there exists a unique mild solution to \eqref{eq:kolm_eq}, in the sense of Definition \ref{def:sol_mild_backward_CP}. 
	
	Furthermore, for any $\gamma\in (1-\beta,2+\beta)$, 
	there exists a positive constant $C=C(\Theta_{T,b},\gamma,\beta)$ such that the following Schauder estimate holds:
	\begin{equation}\label{eq:schauder_est_sing}
		\| u \|_{\Linfb{\gamma}{s}{t}} \leq C \big( 
		\| \ell \|_{ \mathcal{C}_{\bf d}^{\gamma}} + (t-s)^{\frac{2+\beta-\gamma}{2}} \| g \|_{ \Linfb{\beta}{s}{t}} \big).
	\end{equation}
\end{THM}
\begin{proof}
The proof of existence is basically identical to the case with additive noise (\cite[Theorem 4.7]{issoglio2022pde}), namely when $\Lc$ is the generator of the solution to \eqref{eq:sde_kinetic_systems_non_deg} with $\sigma \equiv I_d$ and $F_1\equiv 0$. 
The idea is to find a unique fixed point $u\in L_t^\infty\mathcal{C}^\gamma_{\bf d}$ for the 
map
\begin{equation}
L_t^\infty\mathcal{C}^\gamma_{\bf d}\ni w \longmapsto \ I_s(w):=P^{{\bf F},\sigma}_{{s,t}} \ell - \int_s^t P^{{\bf F},\sigma}_{{s,r}}\big( g(r,\cdot) -  \langle b(r,\cdot) , \nabla_{1} w(r, \cdot ) \rangle \big) \d r, \qquad s\in [0,t],
\end{equation}
and to show that $u\in C_t\mathcal{C}_b$.

By relying on the Schauder estimate of Proposition \ref{prop:schauder}, one can proceed exactly as in the proof of \cite[Theorem 4.7]{issoglio2022pde} and obtain that $I$ is a contraction from the complete metric space $(\Linfa{\gamma}{t}, \| \cdot \|_{\rho,t,\gamma})$ onto itself for $\rho>0$ suitably large, where $\| f \|_{\rho,t,\gamma}  : = \sup_{r\in[0,t]} e^{-\rho(t-r)} \|f_r\|_{ \mathcal{C}^\gamma_{\bf d}}$ . Banach fixed point theorem then yields that there exists a unique fixed point $u\in\Linfa{\gamma}{t}$ for $I$. Eventually, owing to the upper bounds in \eqref{HK-density_upper}-\eqref{eq:HK-density_upper_lower} and, once more, on the Schauder estimate \eqref{eq:ste_schauder}, it is possible to proceed as in the proof of \cite[Lemma 3.14]{issoglio2024degenerate} and obtain $u\in C_t\mathcal{C}_b$. This addresses the existence part of the statement.

We now prove the Schauder estimate \eqref{eq:schauder_est_sing} (which in turn readily gives uniqueness within the function space considered). For any $s \in [0,t]$, we have
\begin{align}
			\| u(s,\cdot)\|_{ \mathcal{C}^\gamma_{\bf d}} & \leq  \|  P^{{\bf F},\sigma}_{{s,t}} \ell  \|_{ \mathcal{C}^\gamma_{\bf d}} + \int_{s}^t \big\| P^{{\bf F},\sigma}_{s,r} \big[ \langle b(r,\cdot) , \nabla_{1} u(r, \cdot ) \rangle 
			 - g(r,\cdot)\big] \big\|_{ \mathcal{C}^\gamma_{\bf d}} \d r  
			 \intertext{(from the controls in \cite{chau:meno:pesc:zhan:23} and by the Schauder estimate \eqref{eq:ste_schauder}  for the second term)}
			 & \leq_C \| \ell  \|_{ \mathcal{C}^\gamma_{\bf d}} +  \int_{s}^t (r-s)^{\frac{\beta - \gamma}{2}}  \big\| \big[ \langle b(r,\cdot) , \nabla_{1} u(r, \cdot ) \rangle 
			 - g(r,\cdot)\big] \big\|_{ \mathcal{C}^\beta_{\bf d}} \d r 
			 \intertext{{(by Proposition \ref{prop:bony_prod})}}
			 & \leq_C \| \ell  \|_{ \mathcal{C}^\gamma_{\bf d}} +  \int_{s}^t (r-s)^{\frac{\beta - \gamma}{2}} \big(  \|  b(r,\cdot) \|_{ \mathcal{C}^\beta_{\bf d}} \, \|\nabla_{1} u(r, \cdot )\|_{ \mathcal{C}^{\gamma-1}_{\bf d}}  
			 + \| g(r,\cdot) \|_{ \mathcal{C}^\beta_{\bf d}} \big) \d r.
		\end{align}
Employing $\|\nabla_{1} u(r, \cdot )\|_{ \mathcal{C}^{\gamma-1}_{\bf d}} \leq \| u(r, \cdot )\|_{ \mathcal{C}^{\gamma}_{\bf d}}$ and taking the supremum then yield
\begin{equation}
\| u\|_{\Linfb{\gamma}{s}{t}} \leq_C \| \ell  \|_{ \mathcal{C}^\gamma_{\bf d}} +  (t-s)^{\frac{2+\beta-\gamma}{2}} \| g \|_{ \Linfb{\beta}{s}{t}} +  \int_{s}^t (r-s)^{\frac{\beta - \gamma}{2}}  \| b \|_{ \Linfb{\beta}{s}{r}}  \, \| u \|_{ \Linfb{\beta}{s}{r}}    \d r.
\end{equation}	
Thus \eqref{eq:schauder_est_sing} stems from the Gronwall-Volterra Lemma.	
		\end{proof}


We now establish a stability result for the singular Kolmogorov Cauchy problem. For a given $t\in(0,T]$, and for a source term $g\in \Linfa{\beta}{t}$, we consider the Cauchy problems
\begin{equation}\label{eq:kolm_eq_reg}
	\begin{cases}
		\Lc^{(n)} u  = 
		g^{(n)} \qquad \text{on } (0,t) \times \R^{2d}, \\
		u_t = \ell.
	\end{cases}
\end{equation}
Here,
\begin{equation}\label{eq:op_Ln}
\Lc^{(n)} = \Kc + \langle b^{(n)}_s ,  \nabla_1   \rangle,
\end{equation}
and $(b^{(n)}, g^{(n)})$ is a sequence 
that converges to $(b,g)$ in $L_t^\infty\mathcal{C}^\eta_{\bf d}$ for some $\eta<\beta$. 

\begin{THM}[Stability of the KC problem]\label{thm:regul_kolm}
Let Assumption {\bf [H]} be in force. Let also $t\in(0,T]$, $g\in \Linfa{\beta}{t}$ 
	and $\ell \in {\mathcal{C}_{\bf d}^{2+\beta}}$. 
Let $(b^{(n)}, g^{(n)})$ be a sequence 
such that
\begin{equation}
\| b - b^{(n)} \|_{\Linfa{\eta}{t}} + \| g - g^{(n)} \|_{\Linfa{\eta}{t}}\longrightarrow 0 \quad \text{as } n\to \infty, \qquad \eta\in (-\nu-\beta , \beta).
\end{equation}
Then, for any $\gamma \in (0, 2+\beta)$, we have:
\begin{equation} \label{eq:cont_propr}
\|u^{(n)} - u \|_{\Linfa{\gamma}{t}} \longrightarrow 0 \quad \text{as } n\to \infty,
\end{equation}
with  
$u^{(n)},u
$ being the unique mild solutions to \eqref{eq:kolm_eq_reg} and \eqref{eq:kolm_eq}, respectively, in the sense of Definition \ref{def:sol_mild_backward_CP}.
\end{THM}
\begin{proof}
Owing to the anisotropic Schauder estimates of Proposition \ref{prop:schauder}, the proof is a straightforward modification of the one of \cite[Lemma 4.17-(ii)]{issoglio2022pde}.
\end{proof}
\begin{REM}\label{rem:equi}
Under the assumptions of Theorem \ref{thm:regul_kolm}, the family $u^{(n)}$ converges uniformly to $u$ in $C_t\mathcal{C}_b$. Therefore, the family $u^{(n)}$, $n\in\mathbb{N}$,  is uniformly equicontinuous.
\end{REM}
\section{The singular martingale problem}\label{SEC_PB_MART}
Throughout this section, we let Assumption {\bf [H]} be in force. Hereafter, we denote by $(C_{[t_0,T]}(\R^{2d}), \mathcal{B})$ the canonical measurable space of the continuous functions from $[t_0,T]$ to $\R^{2d}$, with $t_0\in [0,T)$, equipped with the usual Borel $\sigma$-algebra. We also denote by $X$ the canonical process on $(C_{[t_0,T]}(\R^{2d}), \mathcal{B})$.

\begin{df}\label{def_singular_mp}
Given $t_0\in [0,T)$ and $\initlaw$ a probability distribution on $\R^{2d}$, we say that a 
probability measure $P=P_{t_0,\initlaw}$ on $(C_{[t_0,T]}(\R^{2d}), \mathcal{B})$ is a {\em solution to $\text{MP}(b,\gF,\sigma;t_0, \initlaw)$} if, denoting by $(\x_r)_{r\in[t_0,T]}$ the canonical process on $(C_{[t_0,T]}(\R^{2d}), \mathcal{B})$,
\begin{itemize}
\item[(i)] $\x_{t_0} \sim \initlaw$ under $P$;
\item[(ii)] for all $g\in L^{\infty}_T \mathcal S$  and all $\ell \in \mathcal S$, the process $(M_r^{u,t_0})_{r\in[ t_0 , T ]}$ defined through
\begin{equation} \label{eq:martingale}
M_r^{u,t_0} 
:= u(r, \x_r) - u(t_0, \x_{t_0}) - \int_{t_0}^r g(s, \x_s) \d s,
\end{equation}
with $u\in  C_T\mathcal{C}_b \cap  \Linf{\eta}$ (for any $\eta\in (0,2+\beta)$) being the unique solution to \eqref{eq:kolm_eq} with $t=T$ (cf Theorem \ref{th:backward_CP_wellposedenss}), is a (local) martingale under $P$ with respect to the natural filtration of $(\x_r)_{r\in[t_0,T]}$.
\end{itemize}
\end{df}


\begin{REM}
The previous definition differs from classical one, by Stroock-Varadhan (\cite[Ch. 6]{stroock_multidimensional_1979}), in that the domain of the test functions is obtained by inverting the Kolmogorov operator $\Kc$ as opposed to taking the smooth functions with compact support.
The next lemma, which will be employed below to prove Theorem \ref{thm-main-deg}, shows that the two notions are equivalent in the non-singular case, namely when $b=0$. In particular, it will allow us to characterize the solution to $\text{MP}(b,\gF,\sigma;t_0, \initlaw)$, in the sense of Definition \ref{def_singular_mp}, a the limit of the classical solutions, in the sense of Stroock-Varadhan, to the regularized martingale problems obtained by replacing $b$ with a regular sequence $b^{(n)}$ that approximates $b$ in $L_T^\infty\mathcal{C}^{\beta^-}_{\bf d}$. 
\end{REM}

{\begin{lemma}\label{lemm:equivalence}
Given an initial time $t_0\in [0,T)$ and an initial probability distribution $\initlaw$ on $\R^{2d}$, the unique solution $P=P_{t_0,\initlaw}$ (see Proposition \ref{prop-HK-proxy-sde}) to the martingale problem associated with $\gF,\sigma$ in the sense of Stroock-Varadhan (\cite[Ch. 6]{stroock_multidimensional_1979}) is also 
a solution to $\text{MP}(0,\gF,\sigma;t_0, \initlaw)$, in the sense of Definition \ref{def_singular_mp}.
\end{lemma}
\begin{proof}

Fix $g\in L^{\infty}_T \mathcal S$, $\ell \in \mathcal S$, and show that $M^{u,t_0}$ as in \eqref{eq:martingale} is a (local) martingale, with $u\in  C_T\mathcal{C}_b \cap  \Linf{\eta}$ (for any $\eta\in (0,2+\beta)$) being the unique mild solution to \eqref{eq:kolm_eq} with $t=T$ (see Theorem \ref{th:backward_CP_wellposedenss}).

Let $(F^{(m)}_1, \sigma^{(m)})$, $m\in\N$, be the sequence defined by
\begin{equation}
F^{(m)}_1(t,\cdot):= \Phi_m \ast {F_1}{(t,\cdot)}, \quad \sigma^{(m)}{(t,\cdot)} := \Phi_m \ast \sigma{(t,\cdot)},\qquad t\in [t_0,T].
\end{equation} 
First notice that Assumptions {\bf [H-$F$]} and {\bf [H-$\sigma$]} yield 
\begin{equation}\label{eq:conv_s_F_m_bis}
[ F_1 - F^{(m)}_1 ]_{{\bf d},T,0} + \| \sigma - \sigma^{(m)} \|_{\Linf{\beta+\nu}}  \longrightarrow 0 \quad \text{as } m\to \infty.
\end{equation}
For any $m\in\N$, let $u^{(m)}$ denote the unique mild solution to the Kolmogorov Cauchy problem
\begin{equation}\label{eq:kolm_eq_reg_zero__m_bis}
	\begin{cases}
		\big(\partial_s + \Ac^{(m)}_s\big) u  = 
		g \qquad \text{on } (t_0,T) \times \R^{2d}, \\
		u_T = \ell,
	\end{cases}
\end{equation}
where
\begin{equation}\label{eq:op_Lnm_bis}
 \Ac^{(m)}_s := 
\frac 12{\rm Tr} \Big (\sigma^{(m)}\sigma^{(m)*}(s,\x) \nabla_1^2  \Big) + \langle  (F_1^{(m)}(s,\x) ,F_2(s,\x)  )  ,\nabla \rangle.
 \end{equation}

In light of Assumption {\bf [H]}, f
the coefficients, $F^{(m)}_1$, $F_2$, and $\sigma^{(m)}$ satisfy the assumptions in \cite[Theorem 6.11]{MR4124429}, up to performing and additional localization. Therefore, $u^{(m)}$ is a classical solution to \eqref{eq:kolm_eq_reg_zero__m_bis} in the sense of \cite[Definition 6.1]{MR4124429}. In particular, we have
\begin{equation}
\nabla_{\x} u^{(m)},\nabla^2_{\x_1} u^{(m)} \in C_b([t_0,T]\times\Rdd)
\end{equation}
and 
\begin{equation}\label{eq:sol_clas_zhang}
u^{(m)}(s,\x) =   \ell(\x)+ \int_s^T \big[\Ac^{(m)}_r u^{(m)} - g\big](r,\x)  \d r, \qquad (s,\x)\in[t_0,T]\times\Rdd.
\end{equation}
Therefore, a straightforward extension of Theorem 4.2.1 in \cite{stroock_multidimensional_1979} shows that the process
\begin{equation} \label{eq:martingale_SV_bis}
M_r^{u,t_0,(m)} 
:= u^{(m)}(r, \x_r) - u^{(m)}(t_0, \x_{t_0}) - \int_{t_0}^r \Big({(\partial_s+\Ac^{(m)}_s)}  u^{(m)}(s, \x_s)\Big) \d s, \qquad r\in[0,T],
\end{equation}
is a (local) martingale under $P$. Now, \eqref{eq:sol_clas_zhang} implies that the PDE in \eqref{eq:kolm_eq_reg_zero__m_bis} is satisfied almost everywhere on $(t_0,T) \times \R^{2d}$. 
Furthermore, by Proposition \ref{prop-HK-proxy-sde}, the time-marginal laws of $\x_r$, under $P$, are absolutely continuous with respect to the Lebesgue measure, and thus we obtain
\begin{equation}
M_r^{u,t_0,(m)} 
= u^{(m)}(r, \x_r) - u^{(m)}(t_0, \x_{t_0}) - \int_{t_0}^r g(s, \x_s) \d s, \qquad r\in[0,T],
\end{equation}  
up to a modification. On the other hand, existing results (see e.g. the passage to the limit from smooth to Hölder coefficients in \cite{chau:hono:meno:21}) yield
\begin{equation}
\| u - u^{(m)} \|_{L^{\infty}([t_0,T] \times \Rdd)}  \longrightarrow 0 \quad \text{as } m\to \infty.
\end{equation} 
Therefore, $M^{u,t_0,(m)}$ converges to $M^{u,t_0}$ ucp and thus $M^{u,t_0}$ is a (local) martingale. 
\end{proof}}
The rest of the section is devoted to proving Theorem \ref{thm-main-deg}.

\subsection{Uniqueness}

Proceeding with classical arguments (e.g. in \cite[Appendix A]{issoglio2024degenerate} for a kinetic setting with singular drift, {or in \cite{chau:meno:22} for a non-degenerate setting with singular drift}), one can show that it is enough to prove the uniqueness of the marginals. Namely, it is enough to show that, for any $t_0\in[0,T)$ and $\initlaw$ probability distribution on $\Rdd$, and for any pair $P^1$, $P^2$ of solutions to $\text{MP}(b,\gF,\sigma;t_0, \initlaw)$, the following holds:
\begin{equation}\label{eq:M_prop}
P^1(\x_s \in \d\x) = P^2 (\x_s \in \d\x)
\end{equation}
for any $s\in[t_0, T]$.

 Let 
 $g \in C_T \mathcal S$
 and $u\in  C_T\mathcal{C}_b \cap  \Linf{\eta}$ (for any $\eta\in (0,2+\beta)$) be the unique solution to \eqref{eq:kolm_eq} with $t=T$ and $\ell = 0$ (cf Theorem \ref{th:backward_CP_wellposedenss})
 . By the Definition \ref{def_singular_mp}, we have 
\begin{equation}
\mathbb E^{P^i} \left[ \underbrace{ u_T( \x_T)}_{=0} -  u_{t_0}(\x_{t_0}) - \int_{t_0}^T g_s(\x_s) d s\right] =0,\qquad \text{for } i=1,2,
\end{equation}
and $P^1(\x_0 \in \d\x)=P^2(\x_0 \in \d\x)=\mu_0$. As a result,
\begin{equation}
\mathbb E^{P^1} \left[\int_{t_0}^T g_s(\x_s)ds\right]=\mathbb E^{P^2} \left[\int_{t_0}^T g_s(\x_s)ds\right].
\end{equation}
As 
 $g \in C_T \mathcal S$ is arbitrary, then \eqref{eq:M_prop} holds for 
 almost all $s\in[t_0,T]$, and since  $\x$ is continuous, it holds for all $s \in[t_0,T]$.  
 This concludes the proof of uniqueness for the martingale problem.  
 \qed

\subsection{Existence
}\label{sec:existence_MP}
We prove the existence of a solution to $\text{MP}(b,\gF,\sigma;t_0, \initlaw)$. 
Hereafter, $b^{(n)}$ and $P^{(n)}_{t_0,\initlaw}$ will denote, respectively, the sequence 
of smooth functions defined by \eqref{eq:def_mollif_b}, for which we recall the estimates of Lemma \ref{lemm:moll_b}, and the unique solution to the Stroock-Varadhan martingale problem associated to \eqref{sde-mollified}.
Note that, in light of Lemma \ref{lemm:equivalence}, $P^{(n)}_{t_0,\initlaw}$ is also the unique solution to $\text{MP}(b^{(n)},\gF,\sigma;t_0,\initlaw)$ in the sense of Definition \ref{def_singular_mp}. For a given $\x_0\in\Rdd$, we also set $P^{(n)}_{t_0,\x_0} : = P^{(n)}_{t_0,\delta_{\x_0}}$.

 We start by establishing the tightness of the sequence of probability measures $P^{(n)}_{t_0,\x_0}$. For given $t_0\in [0,T)$ and $\x_0\in\Rdd$, we let $u^{(n)}$ be the $\Rd$-valued unique mild solution to 
\begin{equation}\label{eq:kolm_eq_reg_zero}
	\begin{cases}
		\Lc^{(n)} u  = 
		-b^{(n)} \qquad \text{on } (0,T) \times \R^{2d}, \\
		u_T = 0,
	\end{cases}
\end{equation}
in the sense of Definition \ref{def:sol_mild_backward_CP}, with $\Lc^{(n)}$ as in \eqref{eq:op_Ln}. This is well-defined in light of Theorem \ref{thm:regul_kolm} and because of \eqref{eq:bound_gn_g}. We also point out that the conclusion of Theorem \ref{thm:regul_kolm}, and in particular of Remark \ref{rem:equi}, applies to $u^{(n)}$ thanks to \eqref{eq:conv sequence}.

We finally define the $\Rd$-valued function
\begin{equation}
\phi^{(n)}(t,\x) : = \x_1 +  u^{(n)}(t,\x) , \qquad (t,\x)\in[0,T]\times\Rdd. 
\end{equation}

\begin{PROP}\label{prop:tightness}
Let $t_0\in [0,T)$, with $T-t_0$ sufficiently small, and 
$\x_0\in\Rdd$. 
Then, the family of $P^{(n)}_{t_0,\x_0}$, $n\in\N$, is tight.
\end{PROP}
The proof of Proposition \ref{prop:tightness} is based on the following 
\begin{lemma}\label{lemm:kol}
Let $t_0\in [0,T)$ and 
$\x_0\in\Rdd$. 
For any $n\in \N$ we have
\begin{equation}\label{eq:bound_kolm}
\mathbb{E}_{P^{(n)}_{t_0,\x_0}}\big[ | \phi^{(n)}(t, \x_t) - \phi^{(n)}(s, \x_s)  |^4 +|  \x^{2}_t - \x^{2}_s  |^4 \big] \leq C{(1+|\x_0|^4)} |t-s|^{2}, \qquad t_0\leq s <t \leq T,
\end{equation}
for some positive $C=C(\Theta_{T,b},\gamma,\beta)$. In particular, $C$ is independent of $n$.
\end{lemma} 
\begin{proof}
We want to use It\^o calculus. Therefore, for any $n\in\N$, we let $(\Omega^{(n)}, \Fc^{(n)} , (\Fc^{(n)}_t)_{t\in[t_0,T]}, \mathbb{P}^{(n)})$ be a filtered probability space, $W^{(n)}$ be a Brownian motion, and $X^{(n)}$ be a continuous and adapted process such that
\begin{equation}
\label{SDE_reg}
\begin{cases}
X^{(n),1}_t = X^{(n),1}_0 + \int_{t_0}^t \big(F_1(r,X^{(n)}_r)+ b^{(n)}(r,X^{(n)}_r)\big) \d r + \int_{t_0}^t  \sigma\big(r,X^{(n)}_r\big) \d W^{(n)}_r,\\
X^{(n),2}_t = X^{(n),2}_0 + \int_{t_0}^t F_2\big(r,X^{(n)}_r\big) \d r,
\end{cases} \quad t\in[t_0,T], 
\end{equation}
with $X^{(n)}_0\sim \delta_{\x_0}$. 
Assume now that we can show
\begin{equation}\label{eq:dynamics_un}
u^{(n)} \big(t, X^{(n)}_t\big) = u^{(n)} \big(t_0, X^{(n)}_{t_0}\big) - \int_{t_0}^t b^{(n)}\big(r, X^{(n)}\big)\d r  +\int_{t_0}^t  \nabla_1 u^{(n)} \big(r, X^{(n)}_r\big) \sigma\big(r, X^{(n)}_r\big)  \d W^{(n)}_r.
\end{equation}
Combining the latter with \eqref{SDE_reg} yields
\begin{equation}\label{eq:dynamics_phi}
\phi^{(n)}\big(t, X^{(n)}_t\big) - \phi^{(n)}\big(s, X^{(n)}_s\big) = \int_{s}^t F_1\big(r,X^{(n)}_r\big)  \d r + \int_{s}^t  [(I_d +  \nabla_1 u^{(n)}  ) \sigma ]\big(r,X^{(n)}_r\big) \d W^{(n)}_r , 
\end{equation}
and thus Jensen's inequality, together with multidimensional BDG inequality 
, yields
\begin{align}
\mathbb{E}_{\mathbb{P}^{(n)}}\big[ \big| \phi^{(n)}\big(t, X^{(n)}_t\big) - \phi^{(n)}\big(s, X^{(n)}_s\big) \big|^4 \big] & \leq 8  (t -s)^3 \int_{s}^t \mathbb{E}_{\mathbb{P}^{(n)}}\big[ \big| F_1\big(r,X^{(n)}_r\big) \big|^4 \big] \d r \\
&\quad + 8 (t-s) \int_{s}^t \mathbb{E}_{\mathbb{P}^{(n)}}\big[ \big|  [(I_d +  \nabla_1 u^{(n)}  ) \sigma ]\big(r,X^{(n)}_r\big) \big|^{4} \big] \d r .
\label{eq:bound_phi}
\end{align}
Observe now that the Assumption {\bf [H-$F$]} yields
\begin{equation}
|F_1(t, \x)| + |F_2(t, \x)| \leq C (1 + |\x|) , \qquad (t,\x)\in[0,T]\times\Rdd. 
\end{equation}
Furthermore, by Theorems \ref{th:backward_CP_wellposedenss} and \ref{thm:regul_kolm} (in particular \eqref{eq:schauder_est_sing}-\eqref{eq:cont_propr}) and Assumption {\bf [H-$\sigma$]}, we obtain
\begin{equation}
\big|  [(I_d +  \nabla_1 u^{(n)}  ) \sigma ]\big(t, \x \big) \big| \leq C , \qquad (t,\x)\in[0,T]\times\Rdd. 
\end{equation}
{Applying the latter two inequalities to \eqref{SDE_reg}-\eqref{eq:bound_phi}, together with \eqref{eq:THM-deg_HK_1_reg}-\eqref{eq:HK-density_upper_lower}}, and with the fact that $P^{(n)}_{t_0,\x_0}$ is the law of $X^{(n)}$ under $\mathbb{P}^{(n)}$, yields \eqref{eq:bound_kolm}. 
{Notice that we utilized the density upper bounds \eqref{eq:THM-deg_HK_1_reg}-\eqref{eq:HK-density_upper_lower} to control the moments of $X^{(n)}$, uniformly in $n$. Alternatively, this control could also be obtained without density bounds, by employing Gronwall Lemma as in the proof of \cite[Lemma 5.7]{issoglio2024degenerate}.} 

To conclude, we need to show \eqref{eq:dynamics_un}. Let $(F^{(m)}_1, \sigma^{(m)})$, $m\in\N$, be the sequence defined by
\begin{equation}
F^{(m)}_1(t,\cdot):= \Phi_m \ast F_1(t,\cdot), \quad \sigma^{(m)}(t,\cdot) := \Phi_m \ast \sigma(t,\cdot),\qquad t\in [0,T].
\end{equation} 
First notice that Assumptions {\bf [H-$F$]} and {\bf [H-$\sigma$]} yield 
\begin{equation}\label{eq:conv_s_F_m_pre}
\| F_1 - F^{(m)}_1 \|_{L^{1}_{\text{loc}}([t_0,T] \times \Rdd)} + [ F_1 - F^{(m)}_1 ]_{{\bf d},T,0} + \| \sigma - \sigma^{(m)} \|_{\Linf{\beta+\nu}} \longrightarrow 0 \quad \text{as } m\to \infty,
\end{equation}
which in turn yields
\begin{equation}\label{eq:conv_s_F_m}
 \| \sigma - \sigma^{(m)} \|_{L^{\infty}([t_0,T] \times \Rdd)}  \longrightarrow 0 \quad \text{as } m\to \infty,
 \end{equation}
and, up to considering a subsequence,
\begin{equation}\label{eq:F1m_pointwise_conv}
F^{(m)}_1 \longrightarrow F_1,\quad \text{almost everywhere on } [t_0,T]\times\Rdd.
\end{equation}
For any $m,k\in\N\cup\{\infty\}$, let $u^{(n,m,k)}$ denote the unique mild solution to the Kolmogorov Cauchy problem
\begin{equation}\label{eq:kolm_eq_reg_zero__m}
	\begin{cases}
		\Lc^{(n,m,k)} u  = 
		- b^{(n)} \qquad \text{on } (t_0,T) \times \R^{2d}, \\
		u_T = 0,
	\end{cases}
\end{equation}
with
\begin{align}\label{eq:op_Lnm}
\Lc^{(n,m,k)} &= \Kc^{(m,k)} + \langle b^{(n)}_s ,  \nabla_1   \rangle, \\
\label{gen_m}
\Kc^{(m,k)} & = \partial_s + \frac 12{\rm Tr} \Big (\sigma^{(m)}\sigma^{(m)*}(s,\x) \nabla_1^2  \Big) + \langle  (F_1^{(k)}(s,\x) ,F_2(s,\x)  ) , \nabla \rangle,
\end{align}
and where we set $\sigma^{(\infty)}:=\sigma$, $F_1^{(\infty)}:=F_1$.

In light of Assumption {\bf [H]}, for any $m,k\in\N$, the coefficients, $(F^{(k)}_1 + b^{(n)})$, $F_2$, and $\sigma^{(m)}$ satisfy the assumptions in \cite[Theorem 6.11]{MR4124429}, up to performing and additional localization. Therefore, if $m,k\in\N$, then $u^{(n,m,k)}$ is a classical solution to \eqref{eq:kolm_eq_reg_zero__m} in the sense of \cite[Definition 6.1]{MR4124429}. In particular, $u^{(n,m,k)}$ has the sufficient regularity to apply It\^o formula, which yields
\begin{align}
 &u^{(n,m,k)}(t, X^{(n)}_t) = u^{(n,m,k)}({t_0}, X^{(n)}_{t_0}) + \int_{{t_0}}^t  \Lc^{(n)} u^{(n,m,k)}(r, X^{(n)}_r) \d r + \int_{{t_0}}^t [( \nabla_1  u^{(n,m,k)})  \sigma](r, X^{(n)}_r)  \d W^{(n)}_r 
\intertext{(as $u^{(n,m,k)}$ solves the first equation in \eqref{eq:kolm_eq_reg_zero__m} almost everywhere on $[t_0,T]\times \Rdd$, and the law of $ X^{(n)}_r$ is absolutely continuous w.r.t. Lebesgue measure for any $r\in [t_0,T]\times \Rdd$ by Proposition \ref{prop-HK-proxy-sde})}
& \qquad = u^{(n,m,k)}({t_0}, X^{(n)}_{t_0}) +  \int_{{t_0}}^t \big( \mathcal{A}^{(n,m)}u^{(n,m,k)} - b^{(n)}\big)(r, X^{(n)}_r) \d r  +  \int_{{t_0}}^t [( \nabla_1  u^{(n,m,k)})  \sigma](r, X^{(n)}_r)  \d W^{(n)}_r,\\
\label{eq:ito_u_nkm}
\end{align}
where
\begin{equation}\label{eq:def_Amk}
\mathcal{A}^{(m,k)} := \Lc^{(n)} - \Lc^{(n,m,k)} =  \frac 12{\rm Tr} \Big (\big[\sigma \sigma^* - \sigma^{(m)}\sigma^{(m)*} \big] (s,\x) \nabla_1^2  \Big) +  \big\langle F_1(s,\x) - F^{(k)}_1(s,\x)  , \nabla_1 \big\rangle.
\end{equation}
Now, existing results (see e.g. the passage to the limit from smooth to Hölder coefficients in \cite{chau:hono:meno:21}) yield
\begin{equation}\label{eq:conv_u_nm}
\| u^{(n,\infty,k)} - u^{(n,m,k)} \|_{\Linfb{2}{t_0}{T}} 
 \longrightarrow 0 \quad \text{as } m\to \infty.
\end{equation}
This 
implies the $\mathbb{P}^{(n)}$-almost sure convergence of
\begin{equation}
u^{(n,m,k)}(t, X^{(n)}_t) - u^{(n,m,k)}({t_0}, X^{(n)}_{t_0})   -  \int_{{t_0}}^t [( \nabla_1  u^{(n,m,k)})  \sigma]\big(r, X^{(n)}_r\big)  \d W^{(n)}_r
\end{equation}
to
\begin{equation}
u^{(n,\infty,k)}(t, X^{(n)}_t) - u^{(n,\infty,k)}({t_0}, X^{(n)}_{t_0})  -  \int_{{t_0}}^t [( \nabla_1  u^{(n,\infty,k)})  \sigma]\big(r, X^{(n)}_r\big)  \d W^{(n)}_r.
\end{equation}
and, owing as well to \eqref{eq:conv_s_F_m}, also
\begin{equation}
\int_{{t_0}}^t  \mathcal{A}^{(m,k)}u^{(n,m,k)}(r, X^{(n)}_r) \d r \longrightarrow \int_{{t_0}}^t  \mathcal{A}^{(k)}u^{(n,\infty,k)}(r, X^{(n)}_r) \d r  \quad \text{as } m\to \infty, \qquad  \mathbb{P}^{(n)}\text{-a.s.},
\end{equation}
with
\begin{equation}\label{eq:def_Ak}
\mathcal{A}^{(k)} : =  \big\langle F_1(s,\x) - F^{(k)}_1(s,\x)  , \nabla_1 \big\rangle.
\end{equation}
Thus, taking the limit for $m\to\infty$ in \eqref{eq:ito_u_nkm} yields
\begin{align}
u^{(n,\infty,k)}(t, X^{(n)}_t) &= u^{(n,\infty,k)}({t_0}, X^{(n)}_{t_0}) \\
&\quad +  \int_{{t_0}}^t \big( \mathcal{A}^{(k)}u^{(n,\infty,k)} - b^{(n)}\big)(r, X^{(n)}_r) \d r  +  \int_{{t_0}}^t [( \nabla_1  u^{(n,\infty,k)})  \sigma](r, X^{(n)}_r)  \d W^{(n)}_r.
\end{align}
The same results quoted above to state \eqref{eq:conv_u_nm} also give
\begin{equation}\label{eq:limit_unm_grad}
\| u^{(n)} - u^{(n,\infty,k)} \|_{L^{\infty}([t_0,T] \times \Rdd)} + \| \nabla_1 u^{(n)} - \nabla_1 u^{(n,\infty,k)} \|_{L^{\infty}([t_0,T] \times \Rdd)} 
 \longrightarrow 0 \quad \text{as } k\to \infty.
\end{equation}
Therefore, we have $\mathbb{P}^{(n)}$-almost sure convergence of
\begin{equation}
u^{(n,\infty,k)}(t, X^{(n)}_t) - u^{(n,\infty,k)}({t_0}, X^{(n)}_{t_0})   -  \int_{{t_0}}^t [( \nabla_1  u^{(n,\infty,k)})  \sigma]\big(r, X^{(n)}_r\big)  \d W^{(n)}_r
\end{equation}
to
\begin{equation}
u^{(n)}(t, X^{(n)}_t) - u^{(n)}({t_0}, X^{(n)}_{t_0})  -  \int_{{t_0}}^t [( \nabla_1  u^{(n)})  \sigma]\big(r, X^{(n)}_r\big)  \d W^{(n)}_r.
\end{equation}
To show \eqref{eq:dynamics_un}, and conclude the proof, we need to show
\begin{equation}\label{eq:conv_Anm_int}
\int_{{t_0}}^t  \mathcal{A}^{(k)}u^{(n,\infty,k)}(r, X^{(n)}_r) \d r \longrightarrow0 \quad \text{as } k\to\infty, \qquad  \mathbb{P}^{(n)}\text{-a.s.}\, .
\end{equation}
Owing to the continuity of the trajectories of $X^{(n)}$, and to the fact that $\mathcal{A}^{(k)}u^{(n,\infty,k)}$ is locally bounded on $[t_0,T]\times\Rdd$, uniformly in $m$, Bounded Converge Theorem yields \eqref{eq:conv_Anm_int}, so long as we can prove that
\begin{equation}\label{eq:limit_Anm_zero}
\mathcal{A}^{(k)}u^{(n,\infty,k)}(r, X^{(n)}_{r}) \longrightarrow 0 \quad \text{as } k\to\infty, \qquad \text{for almost every $r\in[t_0,T]$},
\end{equation}
$\mathbb{P}^{(n)}$-almost surely. 
By \eqref{eq:F1m_pointwise_conv} together with the fact that the law of $X^{(n)}_r$ is absolutely continuous w.r.t. the Lebesgue measure  for any $r\in [t_0,T]$, we have
\begin{equation}
\big(F_1 - F^{(k)}_1\big)(r, X^{(n)}_{r}) \longrightarrow 0\quad \text{as } k\to\infty, \quad \text{for almost every $r\in[t_0,T]$},
\end{equation}
$\mathbb{P}^{(n)}$-almost surely. This, together with \eqref{eq:limit_unm_grad}, proves \eqref{eq:limit_Anm_zero} $\mathbb{P}^{(n)}$-almost surely and completes the proof.
\end{proof}
We can now prove Proposition \ref{prop:tightness}.
\begin{proof}[Proof of Proposition \ref{prop:tightness}]
According to \cite[Theorem 4.10 in Chapter 2]{karatzasShreve} we need to prove the following Aldous criterion\footnote{We emphasize here that, with only equicontinuity in time at hand for $u^{(n)}$, this is the most reasonable tightness criterion to invoke. Indeed, the Kolmogorov type criteria would have required a quantitative Hölder-type time-regularity estimates.}: 
for every $\varepsilon>0$,
\begin{equation}\label{eq:KS2} 
\lim_{\delta\to0} \sup_{n\geq 1} P^{(n)}_{t_0,\x_0} \Big ( \sup_{\substack{
s,t \in [t_0,T] \\|s-t|\leq \delta }} |\x_t-\x_s|>\varepsilon   \Big ) =0.
\end{equation}
We have
\begin{align}
|\x_t - \x_s| &\leq |\phi^{(n)}(t, \x_t) - \phi^{(n)}(s, \x_s) | + |u^{(n)}(t, \x_t)- u^{(n)}(s, \x_s)| + |\x_{t}^2 -\x_{s}^2 |.
\end{align}
Furthermore,
\begin{align}
|u^{(n)}(t, \x_t)- u^{(n)}(s, \x_s)| & \leq |u^{(n)}(t, \x_t)- u^{(n)}(s, \x_t)| + |u^{(n)}(s, \x_t)- u^{(n)}(s, \x_{s}^1, \x_{t}^2)| \\
&\quad + |u^{(n)}(s, \x_{s}^1, \x_{t}^2) - u^{(n)}(s, \x_s)| 
\intertext{(by Theorems \ref{th:backward_CP_wellposedenss} and \ref{thm:regul_kolm}, in particular by \eqref{eq:schauder_est_sing}-\eqref{eq:cont_propr}, assuming w.l.o.g. that $T$ is small enough)}
& \leq |u^{(n)}(t, \x_t)- u^{(n)}(s, \x_t)| + \frac{1}{2} |\x_{t}^1 - \x_{s}^1| + \frac{1}{2}  |\x_{t}^2 - \x_{s}^2|^{\frac{1}{3}}.
\end{align}
Thus we obtain
\begin{equation}
|\x_t - \x_s| \leq  2  |\phi^{(n)}(t, \x_t) - \phi^{(n)}(s, \x_s) | + 2 |u^{(n)}(t, \x_t)- u^{(n)}(s, \x_t)| + 2 |\x_{t}^2 - \x_{s}^2|  +  |\x_{t}^2 - \x_{s}^2|^{\frac{1}{3}}.
\end{equation}
Therefore, to prove \eqref{eq:KS2} it is enough to prove
\begin{align}\label{eq:tight_1}
\lim_{\delta\to0} \sup_{n\geq 1} P^{(n)}_{t_0,\x_0} \Big ( \sup_{\substack{s,t \in [t_0,T] \\|s-t|\leq \delta }}  |\phi^{(n)}(t, \x_t) - \phi^{(n)}(s, \x_s) | >\varepsilon   \Big ) &=0, \\ \label{eq:tight_2}
\lim_{\delta\to0} \sup_{n\geq 1} P^{(n)}_{t_0,\x_0} \Big ( \sup_{\substack{s,t \in [t_0,T] \\|s-t|\leq \delta }}   |u^{(n)}(t, \x_t)- u^{(n)}(s, \x_t)|  >\varepsilon   \Big ) &=0, \\ \label{eq:tight_3}
\lim_{\delta\to0} \sup_{n\geq 1} P^{(n)}_{t_0,\x_0} \Big ( \sup_{\substack{s,t \in [t_0,T] \\|s-t|\leq \delta }} \big(  |\x_{t}^2 - \x_{s}^2| +  |\x_{t}^2 - \x_{s}^2|^{\frac{1}{3}} \big) >\varepsilon   \Big ) &=0,
\end{align}
for every $\varepsilon>0$. The limit in \eqref{eq:tight_2} stems directly from the equicontinuity of $u^{(n)}$ (see Remark \ref{rem:equi}), 
while \eqref{eq:tight_1} and \eqref{eq:tight_3} stem from Lemma \ref{lemm:kol} together with Garsia-Rodemich-Rumsey Lemma (see \cite[Section 3]{BarlowYor}).
\end{proof}
We can now prove the existence of the solution to $\text{MP}(b,\gF,\sigma;t_0, \delta_{\x_0})$. 
\begin{PROP}\label{prop:exist_delta}
Let $t_0\in [0,T)$ 
and 
$\x_0\in\Rdd$. We have
\begin{equation}\label{eq:limit_MP_delta}
P^{(n)}_{t_0,\x_0} \overset{\text{w}}{\longrightarrow} P_{t_0,\x_0}, \quad \text{as } n\to +\infty,
\end{equation}
where $P_{t_0,\x_0}$ is the (unique) solution to $\text{MP}(b,\gF,\sigma;t_0, \delta_{\x_0})$.
\end{PROP}
\begin{proof}
It is not restrictive to assume $T - t_0 $ is suitably small. By Proposition \ref{prop:tightness}, $P^{(n)}_{t_0,\x_0}$ converges weakly to some $P_{t_0,\x_0}$, up to a subsequence. We now show that $P_{t_0,\x_0}$ solves $\text{MP}(b,\gF,\sigma;t_0, \delta_{\x_0})$. Fix $g\in L^{\infty}_T \mathcal S$, $\ell \in \mathcal S$, and let $u$ and $u^{(n)}$ be the unique mild solutions (cf. Definition \ref{def:sol_mild_backward_CP}) to \eqref{eq:kolm_eq} and \eqref{eq:kolm_eq_reg}, respectively, with $t = T$ and $g^{(n)} = g$. 

As $P^{(n)}_{t_0,\x_0}$ is a solution to $\text{MP}(b^{(n)},\gF,\sigma;t_0, \delta_{\x_0})$, the process $M^{u^{(n)},t_0}$ ad defined in \eqref{eq:martingale} is a martingale under $P^{(n)}_{t_0,\x_0}$. Furthermore, by \eqref{eq:conv sequence} we can apply Theorem \ref{thm:regul_kolm}, in particular \eqref{eq:cont_propr}, and conclude that $u^{(n)}$ tends uniformly to $u$. Therefore, standard arguments show that $M^{u,t_0}$ is a martingale under $P_{t_0,\x_0}$. Moreover, trivially $X_{t_0} \sim \delta_{\x_0}$ under $P_{t_0,\x_0}$, and thus $P_{t_0,\x_0}$ is a solution to $\text{MP}(b,\gF,\sigma;t_0, \delta_{\x_0})$. As the solution is unique, $P_{t_0,\x_0}$ does not depend on the choice of the subsequence and thus we have \eqref{eq:limit_MP_delta}.
\end{proof}
We are now in the position to conclude the proof of existence for a general $\initlaw$. We employ a standard superposition argument. As each probability measure $P^{(n)}_{t_0,\x}$ is the (unique) solution to the Stroock-Varadhan martingale problem, it is known (\cite[Theorem 7.1.6]{stroock_multidimensional_1979}) that, for any $n\in \N$, the function
\begin{equation}
\Rdd \ni \x  \mapsto P^{(n)}_{t_0,\x} \in  \mathcal{P} (C_{[t_0,T]}(\R^{2d}), \mathcal{B})
\end{equation}
is 
Borel measurable, with $\mathcal{P} (C_{[t_0,T]}(\R^{2d}), \mathcal{B})$ being the space of probability measures on $(C_{[t_0,T]}(\R^{2d}), \mathcal{B})$, equipped with the weak topology. 
Therefore, the function $\Rdd \ni \x  \mapsto P_{t_0,\x}$ is also measurable for it is the limit of measurable functions (see Proposition \ref{prop:exist_delta}), and because the weak topology on $\mathcal{P} (C_{[t_0,T]}(\R^{2d}), \mathcal{B})$ is metrizable. We can then define a probability measure on $(C_{[t_0,T]}(\R^{2d}), \mathcal{B})$ as
\begin{equation}\label{eq:const_gen_law}
P_{t_0 , \initlaw} : = \int_{\Rdd} \initlaw(\d \x) P_{t_0,\x} .
\end{equation}
We need to show solves $\text{MP}(b,\gF,\sigma;t_0, \initlaw)$. 
Fix $g\in C_T \mathcal S$, $\ell \in \mathcal S$, and let $u$ be the unique mild solution (cf. Definition \ref{def:sol_mild_backward_CP}) to \eqref{eq:kolm_eq} with $t = T$. Let $t_0 \leq s <r \leq T$ and ${\bf f}_s$ be a continuous functional on $C_{[t_0,s]}(\R^{2d})$. We have
\begin{align}
\mathbb{E}_{P_{t_0 , \initlaw}}[ (M^{u,t_0}_r - M^{u,t_0}_s)  {\bf f}_s] & = \int_{C_{[t_0,s]}(\R^{2d})} (M^{u,t_0}_r - M^{u,t_0}_s)(\omega)  {\bf f}_s (\omega)  \int_{\Rdd} \initlaw(\d \x) P_{t_0,\x} (\d \omega) \\
& = \int_{\Rdd} \initlaw(\d \x) \int_{C_{[t_0,s]}(\R^{2d})} (M^{u,t_0}_r - M^{u,t_0}_s)(\omega)  {\bf f}_s (\omega)   P_{t_0,\x} (\d \omega) \\
& = \int_{\Rdd} \initlaw(\d \x)\, \mathbb{E}_{P_{t_0 , \x}}[ (M^{u,t_0}_r - M^{u,t_0}_s)  {\bf f}_s] = 0,
\end{align}
where we employed, in the last equality, that $P_{t_0 , \x}$ solves $\text{MP}(b,\gF,\sigma;t_0, \delta_{\x})$. This proves that $M^{u,t_0}$ is a martingale under $P_{t_0 , \initlaw}$. Furthermore, trivially $X_{t_0} \sim \initlaw$ under $P_{t_0,\initlaw}$, and thus $P_{t_0,\initlaw}$ solves $\text{MP}(b,\gF,\sigma;t_0, \initlaw)$. 

This concludes the proof of existence and thus the proof of Theorem \ref{thm-main-deg}-(i).
\qed

\subsection{Markov property and density estimates}

In this subsection we prove Theorem \ref{thm-main-deg}-(ii),(iii). We recall that $P(s , \x; t , H) : = P_{s, {\x}}(\x_t \in H)$, and also set
\begin{equation}\label{eq:trans_prob_moll}
P^{(n)}(s, \x; t , H) : = P^{(n)}_{s, {\x}}(\x_t \in H), \qquad  n\in\N, 
\end{equation}
for any $0\leq s \leq t\leq T $, $\x\in\Rdd$ and any Borel set $H \subset \Rdd$. 

We start by proving the existence of the density and its bounds. 
\begin{proof}[Proof of Theorem \ref{thm-main-deg}-(iii)]
Fix $0\leq s  < t\leq T$. By Theorem \ref{thm-main-mollified}, the families $\pn(s,\cdot,t,\cdot)$, $\nabla_{\x_1}\pn(s,\cdot,t,\cdot)$ are uniformly bounded and equicontinuous on $\Rdd\times\Rdd$. Therefore, by Arzel\`a-Ascoli Theorem, there exists a function $p=p(s,\cdot,t, \cdot)$ such that
\begin{equation}
\pn(s,\cdot,t,\cdot) \longrightarrow  p(s,\cdot,t,\cdot), \quad \nabla_{\x_1}\pn(s,\cdot,t,\cdot) \longrightarrow  \nabla_{\x_1}p(s,\cdot,t,\cdot), \quad \text{as }n\to \infty, 
\end{equation}
uniformly on compacts of $\Rdd\times\Rdd$, up to a subsequence. 
In particular, for any test function $\varphi\in C_0(\Rdd)$, we have
\begin{equation}
\int_{\Rdd} \varphi(y) P^{(n)}(s, \x, t , \d\y) = \int_{\Rdd} \varphi(y) \pn(s, \x, t , \y) \d\y \longrightarrow \int_{\Rdd} \varphi(y) p(s, \x, t , \y) \d\y, \quad \text{as } n\to \infty.
\end{equation} 
As $P^{(n)}(s, \x, t , \d\y)$ also tendes weakly to $P(s, \x, t , \d\y)$ (cf. Proposition \ref{prop:exist_delta}), we have 
\begin{equation}
P(s, \x, t , \d\y) = p(s, \x, t , \y) \d\y,
\end{equation}
namely $p(s,\x,t,\cdot)$ is the density of the probability measure $P(s , \x; t , \d\y)$. In particular, $p$ does not depend on the choice of the subsequence. 

Finally, the estimates in Theorem \ref{thm-main-mollified}, together with the Gaussian estimates \eqref{eq:HK-density_upper_lower}, yield estimates \eqref{eq:THM-deg_HK_1}-\eqref{eq:THM-deg_HK_5}. 
\end{proof}

\begin{lemma}
The transition density $p$ satisfies the Chapman-Kolmogorov equation 
\begin{equation}\label{eq:CK}
p(s,\x,t,\y) = \int_{\Rdd} p(s,\x,r,\z) p(r,\z,t,\y) \d\z ,   \qquad 0\leq s <r < t\leq T ,\quad  \x,\y\in\Rdd.
\end{equation}
In particular, $P(s,\x , t , \d\y )$ is a transition probability function in the sense of \cite[Ch. 2.2, p. 51-52]{stroock_multidimensional_1979}.
\end{lemma}

\begin{proof}
The identity is true for the transition densities $\pn$. By the estimates in Theorem \ref{thm-main-mollified} one can pass to the limit and obtain \eqref{eq:CK}.

In the end of Section \ref{sec:existence_MP}, we have already shown that the map $\x \mapsto P(s,\x , t , H)$ is Borel measurable, and thus \eqref{eq:CK} implies the Chapman-Kolmogorov identity for the transition probability function $P(s,\x , t , \d\y )$, which is then a transition probability function in the sense of \cite[Ch. 2.2, p. 51-52]{stroock_multidimensional_1979}.
\end{proof}
We can now prove the Markov property.
\begin{proof}[Proof of Theorem \ref{thm-main-deg}-(ii)] For any $t_0 \leq s < t\leq T$, the upper bound in \eqref{eq:THM-deg_HK_1} yields 
\begin{align}
\hspace{-20pt} 
 \int_{\Rdd}  |\x - \y|^{3} P(s, \x ; t , \d\y) &  \leq C  \int_{\Rdd}  |\x - \y|^{3} g_\lambda(t-s, \tilde \btheta_{t,s}(\x) - \y) \d\y
 \\
 &  \leq C \int_{\Rdd} \big(  |\tilde \btheta_{t,s}(\x) - \y|^{3} +  |\tilde \btheta_{t,s}(\x) - \x|^{3}   \big) g_\lambda(t-s, \tilde \btheta_{t,s}(\x) - \y) \d\y \intertext{(by definition of $|\cdot |_{\bd}$ and since $| \tilde \btheta_{t,s}(\x) - \x| \leq C (t-s) |\x|$)}
 & \leq C \int_{\Rdd} \big(|\tilde \btheta_{t,s}(\x) - \y|_{\bd}^{3} +  |\tilde \btheta_{t,s}(\x) - \y|_{\bd}^{9} +  (t-s)^{3}| \x|^{3}   \big) g_\lambda(t-s, \tilde \btheta_{t,s}(\x) - \y) \d\y \hspace{-20pt} 
 \intertext{(by usual Gaussian estimates)}
 & \leq C   (t-s)^{3/2}| \x|^{3},
\end{align}
for any $\x \in \Rdd$, and thus
\begin{align}
\int_{\Rdd} P( t_0 , \z ; s, \d \x ) \int_{\Rdd}  |\x - \y|^{3+\eps} P(s, \x ; t , \d\y) &   \leq  C   |t-s|^{3/2} \int_{\Rdd} | \x|^{3} P( t_0 , \z ; s, \d \x )  
\intertext{(by proceeding as above, i.e. employing \eqref{eq:THM-deg_HK_1}, $| \tilde \btheta_{t,s}(\z)| \leq C |\z|$ and usual Gaussian estimates)}
 & \leq  C   |t-s|^{3/2} | \z|^{3}, \label{eq:estim_markov}
\end{align}
for any $z \in\Rdd$. Thus, by \cite[Theorem 2.2.4, p. 54]{stroock_multidimensional_1979}, for any $ \z \in \Rdd$ there exists a unique Markov probability measure $\tilde P_{t_0,\z}$ on $(C_{[t_0,T]}(\R^{2d}), \mathcal{B})$ with transition kernel $P$ and initial law $\delta_{\z}$, namely such that
\begin{align}
\tilde P_{t_0,\z}(\x_0 = \z ) = 1 , && \tilde P_{t_0,\z}(\x_t \in H | \Fc^{t_0}_s) 
= P(s,\x_s; t,  H) , \quad P_{t_0,\z}\text{-a.s.},
\end{align}
for any $t_0\leq s\leq t \leq T$ and any Borel set $H\subset \Rdd$. {Notice that the application of \cite[Theorem 2.2.4, p. 54]{stroock_multidimensional_1979} here is not direct but requires a trivial modification, as the bound in \eqref{eq:estim_markov} depends on $\z$, unlike the one in the assumptions of in \cite[Theorem 2.2.4, p. 54]{stroock_multidimensional_1979}. However, the independence of $\z$ is essential only when the initial distribution is general, whereas we are considering, in this step, a deterministic initial condition.}

We now show that, necessarily, $\tilde P_{t_0,\z} = P_{t_0,\z}$. First note that the finite dimensional distributions of $\tilde P_{t_0,\z}$ write as 
\begin{align}
\tilde P_{t_0,\z}( \x_{s_1} \in H_1 , \dots ,  &\x_{s_n} \in H_n ) \\ 
&
= \int_{H_1 \times H_{n-1}}   \tilde P_{t_0,\z}( \x_{s_1} \in \d \y_1 , \dots ,  \x_{s_{n-1}} \in  \d \y_{n-1} )   \int_{H_n} P(t_{n-1}, \y_{n-1} ; t_n , \d \y_n),
\end{align}
for any choice of $t_0 \leq s_1 < \dots < s_n \leq T $ and $H_1, \dots, H_n$ Borel subsets of $\Rdd$. As the Markov property holds for the mollified solutions $P^{(n)}_{t_0,\z}$ (with transition probability function $P^{(n)}$ as in \eqref{eq:trans_prob_moll}), the same representation holds for the finite dimensional distributions of $P^{(n)}_{t_0,\z}$. Therefore, Proposition \ref{prop:exist_delta} implies that $P^{(n)}_{t_0,\z} \longrightarrow \tilde P_{t_0,\z}$ weakly. As $P^{(n)}_{t_0,\z} $ also converges to $P_{t_0,\z} $, then $\tilde P_{t_0,\z} = P_{t_0,\z}$. This concludes the proof for $\initlaw = \delta_{\z}$. 

For a general $\initlaw$, it is enough to recall that $P_{t_0,\initlaw}$ is represented as in \eqref{eq:const_gen_law} and to check that the Markov property carries on through the integration with respect to $\initlaw$. 
\end{proof}

\bibliographystyle{alpha}
\bibliography{bibli}

\appendix

\section{Proof of Lemma \ref{lemma-besov-convo} (Besov norm of a convolution of densities)} \label{app:proof_lemma_normal}
	Let us first prove estimate \eqref{maj-convo-besov}. 
Recalling the thermic characterization of anisotropic Besov spaces (Definition \ref{DEF_BESOV_THERMIC_ANISO}), and that $\beta\in(-1/2, 0)$, we have to bound
	\begin{align*}
		\mathcal{T}^{-\beta}_{1,1}[\pn(0,\x, r,\cdot)\nabla_{1} p_{\gF,\sigma}( r ,\cdot,t,\y)]
		&=\int_0^{t-r} v^{\frac{\beta}{2}} \Vert \partial_v  g (v,\cdot) * \pn(0,\x,r,\cdot)\nabla_{1} p_{\gF,\sigma}(r,\cdot,t,\y) \Vert_{L^1} \d v\\
		&\qquad +\int_{t-r}^1 v^{\frac{\beta}{2}} \Vert \partial_v  g (v,\cdot) * \pn(0,\x,r,\cdot)\nabla_{1} p_{\gF,\sigma}(r,\cdot,t,\y) \Vert_{L^1} \d v\\
		&=:\mathcal{T}^{- \beta}_{1,1,L}+\mathcal{T}^{- \beta}_{1,1,U},
	\end{align*}
	where we assumed, without loss of generality, that $t-r<1$. We also recall that the notation $\nabla_{1} p_{\gF,\sigma}(r,\cdot,t,\y) $ above means that we consider the gradient with respect to the first $d$ variables associated with the ``$\cdot $''. For the upper part, using the $L^1-L^1$ Young convolution inequality (see \eqref{eq:young_r} below), we have
	\begin{align*}
		\mathcal{T}^{-\beta}_{1,1,U} & \lesssim \int_{t-r}^1 v^{\frac{\beta}{2}} \Vert \partial_v  g (v,\cdot) \Vert_{L^1}\Vert  \pn(0,\x,r,\cdot)\nabla_{1} p_{\gF,\sigma}(r,\cdot,t,\y) \Vert_{L^1} \d v.
	\end{align*} 
	Employing now $|\partial_v  g (v,\cdot) | \lesssim v^{-1}  g^{ {\bf d}}_{2}(v,\cdot)$ (see \eqref{BD_THERM_DEG_DER_TEMP}) and the fact that $ g^{ {\bf d}}_{2}(v,\cdot)$ is a probability density, together with the gradient bound \eqref{HK-gradient-back} 
	and the definition of the normalization term in \eqref{THE_NOR}, yields
	\begin{align}
		\mathcal{T}^{-\beta}_{1,1,U} & \lesssim h_{\x}^{\etaforward,n}(r) \int_{t-r}^1 v^{-1+\frac{\beta}{2}} (t-r)^{-\frac{1}{2}} \int _{\Rdd} p_{\gF,2 \lambda}(0,\x,r,\z) p_{\gF,\lambda}(r,\z,t,\y) \d \z\d v 
		\intertext{(by \eqref{HK-density_upper} and applying Chapman-Kolmogorov identity for $p_{\gF,2 \lambda} $)}
		& \lesssim h_{\x}^{\etaforward,n}(r) p_{\gF,2\lambda}(0,\x,t,\y)(t-r)^{\frac{-1+\beta}{2}}
		.\label{maj-upper-thermic}
	\end{align}
	Let us now turn to the lower part. Notice the following cancellation argument:
	\begin{align*}
\int_{\Rdd} \partial_v  g (v,\z-\w)  \pn(0,\x,r,\z)\nabla_{\z_1} p_{\gF,\sigma}(r,\z,t,\y) \d \w=0, \qquad \z \in \R^{2d}.
	\end{align*}
	This will allow us to take advantage of the H\"older regularity of $\pn(0,\x,r,\cdot)\nabla_1 p_{\gF,\sigma}(r,\cdot,t,\y)$. Namely, we can write
	\begin{equation}
		\mathcal{T}^{-\beta}_{1,1,L} =\int_0^{t-r} v^{\frac{\beta}{2}} \int_{\Rdd} \Big|\int_{\Rdd}  \partial_v  g (v,\z-\w)  I(\x;r,\w,\z;t,\y)  \d \w\Big| \d \z\d v,\label{cancellation}
	\end{equation}
	with 
	\begin{equation}
I(\x;r,\w,\z;t,\y): =  \pn(0,\x,r,\w)\nabla_{\w_1} p_{\gF,\sigma}(r,\w,t,\y) -\pn(0,\x,r,\z)\nabla_{\z_1} p_{\gF,\sigma}(r,\z,t,\y).
\end{equation}
	Next, we bound $I(\x;r,\w,\z;t,\y)$ by considering two different time regimes.

\underline{Diagonal case}: $|\z-\w|_{\bf d}\leq (t-r)^{\frac{1}{2}}$.
		\begin{align}
I(\x;r,\w,\z;t,\y)			
			& \leq \pn(0,\x,r,\w)|\nabla_{\w_1} p_{\gF,\sigma}(r,\w,t,\y) -\nabla_{\z_1} p_{\gF,\sigma}(r,\z,t,\y)|\\ & \quad  +|\pn(0,\x,r,\w)-\pn(0,\x,r,\z)||\nabla_{\z_1} p_{\gF,\sigma}(r,\z,t,\y)|\nonumber.
		\end{align}
By \eqref{HK-holder-gradient-back}, \eqref{HK-gradient-back}, and by the definition of $h_{\x}^{\etaforward,n}(r)$ in in \eqref{THE_NOR}, we have
		\begin{align}
		I(\x;r,\w,\z;t,\y)	& \lesssim h_{\x}^{\etaforward,n}(r)p_{\gF,{2 \lambda}}(0,\x,r,\w) \frac{|\z-\w|^{\etaforward}_{\bf d}}{(t-r)^{\frac{{\etaforward}}{2}}} \frac{1}{(t-r)^{\frac{1}{2}}}\big( p_{\gF,{\lambda}} (r,\z,t,\y)+p_{\gF,{\lambda}} (r,\w,t,\y)\big)\nonumber\\
			& \quad +h_{\x}^{\etaforward,n}(r) \big( p_{\gF,{2 \lambda}}(0,\x,r,\w)+p_{\gF,{2 \lambda}}(0,\x,r,\z)\big)  \frac{|\z-\w|^{\etaforward}_{\bf d}}{r^{\frac{\etaforward}{2}}}   \frac{1}{(t-r)^{\frac{1}{2}}} p_{\gF,{\lambda}}(r,\z,t,\y) 
			\intertext{(by \eqref{HK-density_upper} and employing $|\z-\w|_{\bf d} \leq (t-r)^{\frac{1}{2}}$)}
 			& \lesssim  h_{\x}^{\etaforward,n}(r) \frac{|\z-\w|^{\etaforward}_{\bf d}}{ (t-r)^{\frac{1}{2}}}\bigg[  \bigg(   \frac{1}{(t-r)^{\frac{\etaforward}{2}}} +    \frac{1}{r^{\frac{\etaforward}{2}}} \bigg) p_{\gF,{2 \lambda}}(0,\x,r,\w) p_{\gF,{2 \lambda}} (r,\w,t,\y) \\
	& \qquad\qquad\qquad \qquad\qquad+ \frac{1}{r^{\frac{\etaforward}{2}}}p_{\gF,{2 \lambda}}(0,\x,r,\z) p_{\gF,{2 \lambda}} (r,\z,t,\y)
			 \bigg] \label{pivot-diag}
			.
		\end{align}

\underline{Off-diagonal case}: $|\z-\w|_{\bf d}> (t-r)^{\frac{1}{2}}$. Employing triangular inequality and \eqref{HK-density_upper}, we obtain 
		\begin{equation}\label{pivot-off-diag}
				I(\x;r,\w,\z;t,\y)	 
			 \lesssim h_{\x}^{\etaforward,n}(r) \frac{|\z-\w|^{\etaforward}_{\bf d}}{(t-r)^{\frac{1+\etaforward}{2}}} 
			 \big( p_{\gF,{2 \lambda}}(0,\x,r,\w)p_{\gF,{2 \lambda}} (r,\w,t,\y)+p_{\gF,{2 \lambda}}(0,\x,r,\z)p_{\gF,{2 \lambda}} (r,\z,t,\y) \big). 
		\end{equation}
	By \eqref{pivot-diag} and \eqref{pivot-off-diag}, employing \eqref{spatial-moments-kolm-kern} (moments estimates on the gaussian kernel) and applying Chapman-Kolmogorov identity to $p_{\gF,{2 \lambda}} $ yields 
\begin{equation}
\int_{\Rdd} \int_{\Rdd} \big|  \partial_v  g (v,\z-\w)  I(\x;r,\w,\z;t,\y)   \big| \d \z \d \w  \lesssim h_{\x}^{\etaforward,n}(r) v^{\frac{\etaforward}{2}-1} \frac{1}{(t-r)^{\frac{1}{2}}}\bigg( \frac{1}{(t-r)^{\frac{\etaforward}{2}}}+\frac{1}{r^{\frac{\etaforward}{2}}}\bigg) p_{{2 \lambda}}(0,\x,t,\y).
\end{equation}	
Plugging this into \eqref{cancellation} proves
\begin{align}
\frac{\mathcal{T}^{-\beta}_{1,1,L}}{h_{\x}^{\etaforward,n}(r)} &  \lesssim \frac{1}{(t-r)^{\frac{1}{2}}}\bigg( \frac{1}{(t-r)^{\frac{\etaforward}{2}}}+\frac{1}{r^{\frac{\etaforward}{2}}}\bigg) p_{{2 \lambda}}(0,\x,t,\y)  \int_0^{{t-r}} v^{\frac{\beta + \etaforward}{2}-1} \d v 
\intertext{{(recall that we have assumed $\beta + \etaforward>0$) 
}
}
&  \lesssim \frac{1}{(t-r)^{\frac{1{-\beta -\etaforward}}{2}}}\bigg( \frac{1}{(t-r)^{\frac{\etaforward}{2}}}+\frac{1}{r^{\frac{\etaforward}{2}}}\bigg) p_{{2 \lambda}}(0,\x,t,\y),
\end{align}
which, together with \eqref{maj-upper-thermic}, proves
	\begin{equation}\label{thermic-part-bound}
		\mathcal{T}^{- \beta}_{1,1}[\pn(0,\x,r,\cdot)\nabla_{1} p_{\gF,\sigma}(r,\cdot,t,\y)]\lesssim h_{\x}^{\etaforward,n}(r) \frac{1}{(t-r)^{\frac{1{-\beta -\etaforward}}{2}}}\bigg( \frac{1}{(t-r)^{\frac{\etaforward}{2}}}+\frac{1}{r^{\frac{\etaforward}{2}}}\bigg) p_{{2 \lambda}}(0,\x,t,\y)
	\end{equation}
	for the thermic part. 
Similar computations yield
\begin{equation}
\Vert \pn(0,\x,r,\cdot)\nabla_{1} p_{\gF,\sigma}(r,\cdot,t,\y)\Vert_{L^1}\lesssim h_{\x}^{\etaforward,n}(r)p_{\gF,{2 \lambda}}(0,\x,t,\y),
\end{equation}
which proves \eqref{maj-convo-besov} for $\delta=0$. To obtain the case $\delta=1$, one simply needs to apply a gradient w.r.t $\x_1$ from the beginning of the proof. To prove \eqref{maj-convo-besov-2}, one first has to expand the Besov norm as above and then use \eqref{HK-holder-gradient-forward} on the difference $\nabla_{1} p_{\gF,\sigma}(r,\cdot,t,\y)-\nabla_{1} p_{\gF,\sigma}(r,\cdot,t,\y')$ appearing therein. Estimate \eqref{maj-convo-besov-2} then follows by reproducing the previous computations and by employing \eqref{maj-convo-besov} in the off-diagonal case $|\y-\y'|_d>(t-r)^{\frac 12} $.$\qed $
\begin{REM}[About the estimates in the non degenerate case] As we stated in the main body of the paper, the non-degenerate case could  be retrieved by integration of the degenerate variable when considering the kinetic model associated to an autonomous non-degenerate equation. Alternatively, the above computations could be reproduced almost \textit{mutatis-mutandis} directly for the non-degenerate SDE perturbed by a singular drift, i.e. \eqref{NON_DEG_MARG}, replacing $p_{\gF,\sigma}$ by $p_{F,\sigma}$ 
and exploiting the bounds in \cite{meno:pesc:zhan:21}, instead of those in Proposition \ref{prop-HK-proxy-sde}.
\end{REM}	
	
\section{Sensitivity of the proxy with respect to the forward spatial variable} \label{AUX_EST}	

In this section we prove the regularity estimates \eqref{HK-holder-gradient-forward}-\eqref{eq:ste_new}, with respect to the forward spatial variable, in Proposition \ref{prop-HK-proxy-sde}.

Assumptions {\bf [H-H\"or]}, {\bf [H-${\bf F}$]} and {\bf [H-$\sigma$]}, with $(\mathcal{O}, {\bf o}): = (\Rdd, {\bd})$, are in force throughout this section. Furthermore, we also assume that 
\begin{equation}\label{eq:new_assumpt_grad}
\|\nabla \gF\|_{L^\infty}<+\infty,
\end{equation}
without loss of generality. This indeed allows to exploit the existence of a \textit{usual} flow associated with the drift. The estimates we will obtain below will only depend on the parameters appearing in $\Theta_T $ and \textbf{not} on $\|\nabla \gF\|_{L^\infty}$. 
Eventually, the estimate without the assumption \eqref{eq:new_assumpt_grad} stems from Arzelà-Ascoli type compactness arguments (as in Section 4 of  \cite{chau:meno:pesc:zhan:23}). Namely, $\gF $ in this section must be thought of as a mollification of the initial $\gF$. For the sake of notational simplicity, we will still denote the coefficient by $\gF$ instead of $\gF_m=\gF * \phi_m $, for an appropriate mollifier $\phi_m $, as it should be.  

\subsection{Preliminaries}


We need to provide a specific Duhamel type representation for the density $p_{\gF,\sigma} $. which will allow to investigate the regularity of the density with respect to its spatial forward variable.

Fix $(\t,\bxi)\in [0,+\infty)\times \R^{2d}$ as \textit{freezing parameters} to be chosen later on define the corresponding dynamics:
\begin{equation}
\dot \btheta_{t,\t}(\bxi)=\gF(t,\btheta_{t,\t}(\bxi)), \qquad t\geq0, \  \btheta_{\t,\t}(\bxi)=\bxi.
\end{equation}
The associated linearized stochastic dynamics $(\wt \X_{t,s}^{(\tau,\bxi)})_{t\geq s} $ then writes:
\begin{equation}\label{FROZ}
\wt \X_{t,s}^{(\tau,\bxi)} (\x)=\x+\int^t_s \big[ \gF(r,\btheta_{r,\t}(\bxi))+ \bA(r,\btheta_{r,\t}(\bxi)) \big(\wt \X_{r,s}^{(\tau,\bxi)}(\x)-\btheta_{r,\t}(\bxi)\big) \big]\dif r +
\int^t_s \bsigma(r,\btheta_{r,\t}(\bxi)) \dif W_r, 
 \end{equation} 
where
\begin{equation}
\bA(r,\x)=\begin{pmatrix}0_{d\times d} & 0_{d\times d}\\
\nabla_{x_1}F_2(r,\x) & 0_{d\times d} \end{pmatrix},\quad
\bsigma(r,\x)=\begin{pmatrix}\sigma(r,\x) \\
0_{d\times d}  \end{pmatrix},
\qquad \x=(x_1,x_2)\in \R^{2d}.
\end{equation}
The resolvent  $({\gR}^{(\tau,\bxi)}_{t,s})_{t\geq s}$ associated with $\bA(r,\btheta_{r,\t}(\bxi))$  is explicitly given by
\begin{align}\label{Res00}
\gR^{(\tau,\bxi)}_{t,s}=\begin{pmatrix}\mI_{d\times d} & 0_{d\times d}\\
\int^t_s \nabla_{x_1}F_2(r,\btheta_{r,\t}(\bxi))\dif r & \mI_{d\times d} \end{pmatrix}.
\end{align}
Defining now
\begin{align}\label{RR2}
\bvtheta_{t,s}^{(\tau,\bxi)}(\x)& :={\gR}^{(\tau,\bxi)}_{t,s}\x+
\int^t_s {\gR}^{(\tau,\bxi)}_{t,r}\Big(\gF(r,\btheta_{r,\t}(\bxi))-\bA(r,\btheta_{r,\t}(\bxi))\btheta_{r,\t}(\bxi)\Big)\dif r, \\ 
\bTheta_r^{(\tau,\bxi)}& :=\bsigma(r,\btheta_{r,\t}(\bxi)), \label{SB5}
\end{align}
 the variation  of constants formula yields that $\wt  \X_{t,s}^{(\tau,\bxi)}(\x)$ is explicitly given by
 \begin{align}
\wt  \X_{t,s}^{(\tau,\bxi)}(\x) = \bvtheta_{t,s}^{(\tau,\bxi)}(\x)+ \int^t_s {\gR}^{(\tau,\bxi)}_{t,r}\bTheta_r^{(\tau,\bxi)} \dif W_r.\label{INTEGRATED}
 \end{align}
Consequently,  $\wt  \X_{t,s}^{(\tau,\bxi)}(\x)$
 is a Gaussian random variable, with density $\wt p^{(\tau,\bxi)}(s,\x,t,\cdot) $ given by
\begin{equation}\label{CORRESP}
\wt p^{(\tau,\bxi)}(s,\x,t,\y)=\frac{1}{(2\pi)^{d}\det(\K_{t,s}^{(\tau,\bxi)})^{\frac 12}}
\exp\left( -\frac12\Big\langle\big(\K_{t,s}^{(\tau,\bxi)}\big)^{-1} \big(\bvtheta_{t,s}^{(\tau,\bxi)}(\x)-\y\big), \bvtheta_{t,s}^{(\tau,\bxi)}(\x)-\y \Big\rangle \right),
\end{equation}
where
\begin{equation}\label{KK1}
\K_{t,s}^{(\tau,\bxi)}:=\int^t_s {\gR}^{(\tau,\bxi)}_{t,r}\bTheta_r^{(\tau,\bxi)}({\gR}^{(\tau,\bxi)}_{t,r}\bTheta_r^{(\tau,\bxi)})^* \dif r.
\end{equation}
Importantly, $\wt p^{(\tau,\bxi)}(s,\x,t,\y)$ satisfies
\begin{align}\label{Kolmogorov_frozen}
\wt \Kc^{(\tau,\bxi)}_{s,\x} p^{(\tau,\bxi)}(s,\x,t,\y)&: = \partial_s \wt p^{(\tau,\bxi)}(s,\x,t,\y)+\wt \Ac^{(\tau,\bxi)}_{s,\x}\wt p^{(\tau,\bxi)}(s,\x,t,\y)=0,\\
\wt p^{(\tau,\bxi)}(s,\x,t,\y)&\longrightarrow \delta_\y(\cdot)\; \text{weakly as}\;  s\uparrow t,
\end{align}
where,  for $a=\sigma\sigma^*/2$,
\begin{equation}\label{frozen_gen}
\wt \Ac^{(\tau,\bxi)}_{s,\x}=\tr\big(a(s,\btheta_{s,\tau}(\bxi))\cdot \nabla^2_{x_1}\big)+
\big\langle \big(\gF(s,\btheta_{s,\t}(\bxi))+ \bA(s,\btheta_{s,\t}(\bxi)\big)\big(\x-\btheta_{s,\t}(\bxi)\big),{\nabla}\big\rangle
\end{equation}
is the generator of the diffusion with frozen coefficients in \eqref{FROZ}.\\

We introduce the scale matrix
\begin{equation}\label{eq:scale_mat}
\T_u=\left(\begin{array}{cc}u^{\frac 12}\mathbb I_{d\times d}& 0_{d\times d}\\
0_{d\times d}& u^{\frac 32} \mathbb I_{d\times d}\end{array} \right), \qquad u\geq 0,
\end{equation}
which reflects the scales of the two $d$-dimensional components, and observe that the densities $g^{ {\bf d}}_\lambda$, $\lambda>0$, in \eqref{THE_DEF_GAUSS_DEG} can be written as
\begin{equation}\label{eq:g_lam_T}
g^{ {\bf d}}_\lambda(u,\z)=\frac{1}{(2\pi \lambda)^du^{2d}}\exp\left(- \lambda^{-1} |  \T_u^{-1} \z|^2 \right), \qquad u>0,\ \z\in \Rdd.
\end{equation}
Throughout the remainder of this section, we will drop the superscript ${\bf d}$ in $g^{ {\bf d}}_\lambda(u,\z)$ to ease the notation.

To proceed with the analysis, we need additional controls on the Gaussian densities with frozen coefficients. Importantly, we state the following scaling properties for the covariance matrix (we refer to \cite{chau:meno:pesc:zhan:23} for a proof).
\begin{lemma}[Scaling Lemma]\label{Le25}
There is a constant $\kappa\geq1$ depending only on $\Theta_T$ such that
\begin{align}\label{TT1}
|(\K_{t,s}^{(t,\y)})^{-1/2}\x|^2=\<(\K_{t,s}^{(t,\y)})^{-1}\x,\x\> & \asymp_{\kappa}|\T^{-1}_{t-s}\x|^2,\\
\label{TT0}
|\T^{-1}_{t-s}\gR_{t,s}^{(t,\y)}\x| & \asymp_{\kappa} |\T^{-1}_{t-s}\x|,
\end{align}
for all $0\leq s<t\leq T$ and $\x,\y \in\Rdd$.
\end{lemma}
We also recall the identity
\begin{align}\label{LIN_NON_LIN_FLOW}
\bvtheta_{t,s}^{(t,\y)}(\x)-\y={\gR}^{(t,\y)}_{t,s}(\x-\btheta_{s,t}(\y)).
\end{align}


The lemma below was proved in  Section 2.1 of \cite{chau:meno:pesc:zhan:23} (Lemma 2.2) under even weaker assumptions than those considered in this section. 

\begin{lemma}[Equivalence of the forward and backward flows]\label{lemme:bilipflow_FIRST}
There is a constant $\kappa\geq1$, depending only on $\Theta_T$, such that
\begin{equation}
\label{EQ_EQUIV_FLOW}
\kappa^{-1}\big(|\T^{-1}_{t-s}(\x-\btheta_{v,t}(\y))|-1\big)\leq |\T^{-1}_{t-s}(\btheta_{t,v}(\x)-\y)|
\leq \kappa\big( |\T^{-1}_{t-s}(\x-\btheta_{v,t}(\y))|+ 1\big),
\end{equation}
for any $0\leq s\leq v<t\leq T $ and $\x,\y\in\Rdd$, with the scale matrix $\T_u $ defined in \eqref{eq:scale_mat}. In particular, for any $\lambda>0$, by \eqref{eq:g_lam_T} we have
\begin{equation}\label{eq:change_flow}
e^{-\kappa} g_{\kappa^{-1}\lambda}\big(t-s,\x-\btheta_{v,t}(\y)\big)\leq g_{\lambda}\big(t-s,\btheta_{t,v}(\x)-\y\big)  \leq  e^{\kappa}g_{\kappa\lambda}\big(t-s,\x-\btheta_{v,t}(\y)\big).
\end{equation}
\end{lemma}

We introduce the following notation 
For a multi-index $j\in \N_0^d$, we set the multi-derivative operator
\begin{equation}
D_z^j : = \partial_{z_1}^{j_1} \partial_{z_2}^{j_2} \cdots \partial_{z_d}^{j_d} ,
\end{equation}
and, with a slight abuse of notation, the height of the multi-index
\begin{equation}
| j | : = j_1 + \cdots + j_d.
\end{equation}
Furthermore, for a multi-index ${\bf j}=(j_1,j_2)\in \N_0^d \times \N_0^d$, we set $D_\z^{\mathbf j}:=D_{z_1}^{j_1}D_{z_2}^{j_2}$.

The following proposition is a direct consequence of expression \eqref{CORRESP} and Lemma \ref{Le25}, together with the identity \eqref{LIN_NON_LIN_FLOW}. 
\begin{PROP}[A priori controls for the frozen Gaussian densities]\label{PROP_Proxy}
For any multi-index ${\bf j}=(j_1,j_2)\in \N_0^d \times \N_0^d$, there are constants
$\lambda_0, \lambda_{\bf j}, C_0, C_{\bf j}\geq 1$ depending only on $\Theta_T$ such that
 \begin{align}
C_0^{-1}g_{\lambda_0^{-1}}\big(t-s,{\bvtheta_{s,t}^{(\tau,\bxi)}(\x)-\y}\big)&\leq
\wt p^{(\tau,\bxi)}(s,\x,t,\y)\leq C_0g_{\lambda_0}\big(t-s,{\bvtheta_{s,t}^{(\tau,\bxi)}(\x)-\y}\big),\label{lower_proxy_bis}\\
|D_\y^{\mathbf j} \wt p^{(\tau,\bxi)}(s,\x,t,\y)| + |D_\x^{\mathbf j} \wt p^{(\tau,\bxi)}(s,\x,t,\y)|&\leq C_{\bf j}(t-s)^{-\frac{|j_1|+3|j_2|}{2}}g_{\lambda_{\bf j}}\big(t-s,{\bvtheta_{s,t}^{(\tau,\bxi)}(\x)-\y}\big), \label{upper_proxy_bis}
\end{align}
for any $0\leq s<t\leq T$, $\tau\in[0,T]$ and $\x,\y,\bxi\in\R^{2d}$.

In particular,  if for $\z\in \R^{2d} $, $|\y-\z|_{\bd}\le(t-s)^{\frac 12} $, it holds that:
 \begin{align}
C_0^{-1}g_{\lambda_0^{-1}}\big(t-s,{\bvtheta_{s,t}^{(\tau,\bxi)}(\x)-\y}\big)&\leq
\wt p^{(\tau,\bxi)}(s,\x,t,{\z})\leq C_0g_{\lambda_0}\big(t-s,{\bvtheta_{s,t}^{(\tau,\bxi)}(\x)-\y}\big),\label{lower_proxy_bis_CONVEXITY_DIAG}\\
|D_{{\z}}^{\mathbf j} \wt p^{(\tau,\bxi)}(s,\x,t,{\z})| + |D_\x^{\mathbf j} \wt p^{(\tau,\bxi)}(s,\x,t,{\z})|&\leq C_{\bf j}(t-s)^{-\frac{|j_1|+3|j_2|}{2}}g_{\lambda_{\bf j}}\big(t-s,{\bvtheta_{s,t}^{(\tau,\bxi)}(\x)-\y}\big), \label{upper_proxy_bis_CONVEXITY_DIAG}
\end{align}
and from \eqref{LIN_NON_LIN_FLOW}
 \begin{align}
C_0^{-1}g_{\lambda_0^{-1}}\big(t-s,{\x-\btheta_{t,s}(\y)}\big)\leq
\wt p^{(t,\y)}(s,\x,t,{\boldsymbol{\eta}})&\leq C_0g_{\lambda_0}\big(t-s,{\x-\btheta_{t,s}(\y)}\big),\label{lower_proxy_bis_CONVEXITY_DIAG_SPEC}\\
 |D_{{\boldsymbol{\eta}}}^{\mathbf j} \wt p^{(t,\y)}(s,\x,t,{\boldsymbol{\eta}})| + |D_\x^{\mathbf j} \wt p^{(t,\y)}(s,\x,t,{\boldsymbol{\eta}})&\leq C_{\bf j}(t-s)^{-\frac{|j_1|+3|j_2|}{2}}g_{\lambda_{\bf j}}\big(t-s,{\x-\btheta_{t,s}(\y)}\big).\label{upper_proxy_bis_CONVEXITY_DIAG_SPEC}
\end{align}
\end{PROP}
The last preliminary result is the following lemma, whose proof is deferred to Section \ref{Lem_sensitivity}.
\begin{lemma}[Sensitivity of the proxy w.r.t. the freezing parameters]\label{Lem_sensitivity}
There exist two positive constants, $\lambda, C$ depending only on $\Theta_T$, repsectively, such that
\begin{equation}\label{THE_CTR_SENSI_DIFF_PARAM_Y}
\big|\big(\wt{p}^{(t,\y)}-\wt{p}^{(t,\y')}\big)(r,\z;t,\y')
\leq C |\y-\y'|_{\bd}^{\nu+\beta} g_{\lambda}\big(t-r,\z-\btheta_{r,t}(\y')\big), \qquad r:=t-|\y-\y'|_{\bd}^2,
\end{equation}
for any $\z,\y,\y'\in\R^{2d}$ and $t\in [0,T]$ with $|\y-\y'|_{\bd}^2 \leq t$.
\end{lemma}

The starting point of our analysis is the following Duhamel type representation formula, which readily follows, in the current \textit{smooth-coefficients} setting, from \eqref{Kolmogorov_frozen}-\eqref{frozen_gen} and from the Kolmogorov equation satisfied by the density, i.e.
\begin{equation}
\Kc p_{\gF,\sigma}(\cdot,\cdot,t,\y) = 0,\qquad \text{on } [0,t) \times \R^{2d}, 
\end{equation}
with $\Kc$ as in \eqref{gen}.
For all $0 \leq r <t\le T$, $\tau\in[0,T]$ and $\z,\y,\bxi\in\R^{2d}$, we have:
\begin{equation}
p_{\gF,\sigma}(r,\z,t,\y)
=\wt{p}^{(\tau,\bxi)}(r,\w,t,\y)+\int^t_r\int_{\R^{2d}}p_{\gF,\sigma}(r,\z,v,\w)\big(\Ac_{v,\w}-\wt\Ac^{(\tau,\bxi)}_{v,\w}\big){\wt p}^{(\tau,\bxi)}(v,\w, t,\y)\dif\w \dif v.\label{D1}
\end{equation}
The expansion 
\eqref{D1} corresponds 
to the 
\textit{backward} Duhamel representation of the density for suitable freezing points. 

\subsection{Proof of estimates in the forward variable}
We first prove estimate \eqref{HK-holder-gradient-forward}. It is clear that it suffices to consider the case $|\y-\y'|_{\bf d}\le (t-s)^{\frac 12} $ (intrinsic diagonal regime), as \eqref{HK-holder-gradient-forward} readily stems from \eqref{HK-density_upper} and \eqref{HK-gradient-back} in the non-diagonal regime.
Following the approach adopted in \cite{chau:meno:pesc:zhan:23} and \cite{meno:pesc:zhan:21} for the backward variable (see as well \cite{chau:hono:meno:21} for a related approach in the setting of Schauder estimates), the idea is to exploit \eqref{D1} by choosing suitable time and spatial freezing points in order to compare $p_{\gF,\sigma}(r,\z,t,\y) $ and $p_{\gF,\sigma}(r,\z,t,\y')$. 


While globally we work in the diagonal regime, in the time-convolution in \eqref{D1} we distinguish two local regimes: a diagonal one, i.e. $|\y-\y'|_{\bd}^{2} \leq t-v$, where it seems \textit{reasonable} to expand the two densities to be compared around the same freezing point; and an off-diagonal one, i.e. $|\y-\y'|_{\bd}^{2} > t-v$, where the corresponding contributions cannot be compared and the sought sensitivity in $|\y-\y'|_{\bd}^{\etaforward} $ must come  form time integration. For this to be possible, the freezing points must be chosen accordingly.


From \eqref{D1} with $(\tau,\bxi)=(t,\y) $ we have
\begin{align}
p_{\gF,\sigma}(s,\x,t,\y)&= {\wt p}^{(t , \y)}(s,\x,t,\y)+\int^t_s\int_{\R^{2d}}p_{\gF,\sigma}(s,\x,v,\w)(\Ac_{v,\w}-\wt\Ac^{( t ,\y)}_{v,\w}){\wt p}^{(t ,\y)}(v,\w,t,\y)\dif\w \dif v,\\
p_{\gF,\sigma}(s,\x,t,\y')&= {\wt p}^{(t , \y)}(s,\x,t,\y')+\int^t_s\int_{\R^{2d}}p_{\gF,\sigma}(s,\x,v,\w)(\Ac_{v,\w}-\wt\Ac^{( t ,\y)}_{v,\w}){\wt p}^{(t ,\y)}(v,\w,t,\y')\dif\w \dif v.
\end{align}
Taking the difference and rearranging terms yields
\begin{equation}\label{DIFF_Y_YP}
p_{\gF,\sigma}(s,\x,t,\y) - p_{\gF,\sigma}(s,\x,t,\y') =  \Delta_1+\Delta_2+\Delta_3,
\end{equation}
with
\begin{align}
 \Delta_1 & = {\wt p}^{(t , \y)}(s,\x,t,\y) - {\wt p}^{(t , \y)}(s,\x,t,\y') ,\\
 \Delta_2& = \int^{t - |\y-\y'|_{\bd}^{2}}_s\int_{\R^{2d}}p_{\gF,\sigma}(s,\x,v,\w)\big(\Ac_{v,\w}-\wt\Ac^{( t ,\y)}_{v,\w}\big)\big({\wt p}^{(t ,\y)}(v,\w,t,\y)-{\wt p}^{(t ,\y)}(v,\w,t,\y')\big)\dif\w \dif v ,\\
 \Delta_3& =  \int_{t - |\y-\y'|_{\bd}^{2}}^t\int_{\R^{2d}}p_{\gF,\sigma}(s,\x,v,\w)\big(\Ac_{v,\w}-\wt\Ac^{( t ,\y)}_{v,\w}\big) {\wt p}^{(t ,\y)}(v,\w,t,\y)\dif\w \dif v\\
 &\quad - \int_{t - |\y-\y'|_{\bd}^{2}}^t\int_{\R^{2d}}p_{\gF,\sigma}(s,\x,v,\w)\big(\Ac_{v,\w}-\wt\Ac^{( t ,\y)}_{v,\w}\big){\wt p}^{(t ,\y)}(v,\w,t,\y')\dif\w \dif v.
 \end{align}
 
The terms $\Delta_1$, $\Delta_2$ and $\Delta_3$ correspond, respectively, to the global diagonal case, the local diagonal case and the local off-diagonal case. The second addend in $\Delta_3$ requires further manipulation because the spatial freezing point in the frozen operator is $\y$ while forward variable in the proxy density is $\y'$. Therefore, we write $\Delta_3 =\Delta_{3,1}  + \Delta_{3,2}$ with
  \begin{align}
 \Delta_{3,1}& =  \int_{t - |\y-\y'|_{\bd}^{2}}^t\int_{\R^{2d}}p_{\gF,\sigma}(s,\x,v,\w)\big(\Ac_{v,\w}-\wt\Ac^{( t ,\y)}_{v,\w}\big) {\wt p}^{(t ,\y)}(v,\w,t,\y)\dif\w \dif v\\
 &\quad - \int_{t - |\y-\y'|_{\bd}^{2}}^t\int_{\R^{2d}}p_{\gF,\sigma}(s,\x,v,\w)\big(\Ac_{v,\w}-\wt\Ac^{( t ,\y')}_{v,\w}\big){\wt p}^{(t ,\y')}(v,\w,t,\y')\dif\w \dif v,\\
 \Delta_{3,2}& = \int_{t - |\y-\y'|_{\bd}^{2}}^t\int_{\R^{2d}}p_{\gF,\sigma}(s,\x,v,\w)\big(\Ac_{v,\w}-\wt\Ac^{( t ,\y')}_{v,\w}\big){\wt p}^{(t ,\y')}(v,\w,t,\y')\dif\w \dif v \\
 & \quad  - \int_{t - |\y-\y'|_{\bd}^{2}}^t\int_{\R^{2d}}p_{\gF,\sigma}(s,\x,v,\w)\big(\Ac_{v,\w}-\wt\Ac^{( t ,\y)}_{v,\w}\big){\wt p}^{(t ,\y)}(v,\w,t,\y')\dif\w \dif v .
\end{align}

\paragraph{Analysis of the contributions in \eqref{DIFF_Y_YP}}
In the sequel, to ease the notation, we will denote indistinctly by $\lambda$ any positive constant that depends at most on $\Theta_T$. 

Concerning $\Delta_1$, we have
\begin{equation}
\Delta_1= \int_0^1 \frac{\d}{\d \rho}  {\wt p}^{(t , \y)}\big(s,\x,t,\y' + \rho (\y - \y')\big)    \d \rho 
 = \sum_{k=1,2} \int_0^1 \big\langle y_k - y_k' , \nabla_{z_k} {\wt p}^{(t , \y)}\big(s,\x,t,\z\big)  \big\rangle   \d \rho \Big|_{\z = \y' + \rho (\y - \y')}.
\end{equation}
and thus
that for the first term of the former expansion:
\begin{align}
|\Delta_1 | &\leq \sum_{k=1,2} \int_0^1 \big  |y_k - y_k'| \times \big| \nabla_{z_k} {\wt p}^{(t , \y)}\big(s,\x,t,\z\big)  \big|   \d \rho\, \Big|_{\z = \y' + \rho (\y - \y')}
\intertext{(by \eqref{upper_proxy_bis} in Proposition \ref{PROP_Proxy}, together with \eqref{eq:change_flow} in Lemma \ref{lemme:bilipflow_FIRST})}
&\lesssim \Bigg( \frac{|y_1 - y_1'|}{(t-s)^{\frac{1}{2}}} + \frac{|y_2 - y_2'|}{(t-s)^{\frac{3}{2}}} \Bigg) \int_0^1 g_{\lambda}\big(t-s,\btheta_{t,s}(\x)-(\y' + \rho (\y - \y'))\big) \d \rho\\
& \lesssim   \frac{|\y-\y'|_{\bd}^{\etaforward}}{(t-s)^{\frac{\etaforward}{2}}} \, g_{\lambda}\big(t-s,\btheta_{t,s}(\x)-\y\big), \label{eq:est_D_1}
\end{align}
where we employed $|\y-\y'|_{\bd}\le (t-s)^{\frac 12} $ for the last inequality.

For the term $\Delta_2$ a similar argument applies. Precisely, by exploiting the local diagonal regime for the current time integration variable, i.e. $|\y-\y'|_{\bd}\le (t-v)^{\frac 12} $, we proceed as we did above and obtain
\begin{align}
\big| \nabla^{\delta}_{\w_1} {\wt p}^{(t ,\y)}(v,\w,t,\y)- \nabla^{\delta}_{\w_1} {\wt p}^{(t ,\y)}(v,\w,t,\y') \big| & \lesssim  \frac{|\y-\y'|_{\bd}^{\etaforward}}{(t-v)^{\frac{\delta+\etaforward}{2}}} \, g_{\lambda}\big(t-v,\w-\btheta_{v,t}(\y)\big), \quad \delta=1,2, \\
\big| \nabla_{\w_2} {\wt p}^{(t ,\y)}(v,\w,t,\y)- \nabla_{\w_2} {\wt p}^{(t ,\y)}(v,\w,t,\y') \big| & \lesssim  \frac{|\y-\y'|_{\bd}^{\etaforward}}{(t-v)^{\frac{3+\etaforward}{2}}} \, g_{\lambda}\big(t-v,\w-\btheta_{v,t}(\y)\big).
\end{align}
Note that, this time, we put the flow on the variable $\y$ in order to leave $\w$ free to be integrated. By the estimates above, and by observing that
\begin{align}\label{eq:oper_diff}
\big(\Ac_{v,\w}-\wt\Ac^{( t ,\y)}_{v,\w}\big) &= {\rm Tr}\Big(\big(a(v,\w)-a(v,\btheta_{v,t}(\y))\big) \nabla^{2}_{\w_1}\Big) \\
& \quad + \big\langle \gF(v,\w)-\gF(v,\btheta_{v,t}(\y))- \bA(v,\btheta_{v,t}(\y))\big(\w-\btheta_{v,t}(\y)\big),\nabla_{\w} \big\rangle,
\end{align}
we obtain
\begin{align}
&\big| \big(\Ac_{v,\w}-\wt\Ac^{( t ,\y)}_{v,\w}\big) \big({\wt p}^{(t ,\y)}(v,\w,t,\y)-{\wt p}^{(t ,\y)}(v,\w,t,\y')\big) \big|\\
& \quad \lesssim  \frac{|\y-\y'|_{\bd}^{\etaforward}}{(t-v)^{\frac{\etaforward}{2}}} \Bigg[ \frac{|\w-\btheta_{v,t}(\y)|_{\bd}^{\nu+\beta}}{t-v}+\textcolor{black}{\frac{1+|\w-\btheta_{v,t}(\y)|_{\bf d}}{(t-v)^{\frac 12}}}+\frac{|\w-\btheta_{v,t}(\y)|_{\bf d}^{1+\nu+\beta}}{(t-v)^{\frac 32}}\Bigg] g_{\lambda}\big(t-v,\w-\btheta_{v,t}(\y)\big),
\end{align}
where we exploited the assumptions {\bf [H-${\bf F}$]} and {\bf [H-$\sigma$]}, including the fact that $\gF_1 
$ has linear growth. Now, the Gaussian upper bound of the density $p_{\gF,\sigma} $ (see \eqref{HK-density_upper} together with \eqref{eq:HK-density_upper_lower}) and the bound in \eqref{BD_THERM_DEG_DER_TEMP} yield
\begin{align}
| \Delta_2 |  &\lesssim |\y-\y'|_{\bd}^{\etaforward} \int_s^{t}\int_{\R^{2d}} g_{\lambda}(v - s,\btheta_{v,s}(\x)-\w)g_{\lambda}(t-v,\w-\btheta_{v,t}(\y)) (t-v)^{\frac{\nu+\beta}2-\frac{\etaforward}{2}-1}\d \w \d v
\intertext{(by the Chapman-Kolmogorov identity for $g_{\lambda}$)}
 &\lesssim|\y-\y'|_{\bd}^{\etaforward} \int_s^{t } g_{\lambda}(t-s,\btheta_{v,s}(\x)-\btheta_{v,t}(\y)) (t-v)^{\frac{\nu+\beta}2-\frac{\etaforward}{2}-1} \d v
 \intertext{(by \eqref{eq:change_flow} and by integrating in $\d v$)}
 &\lesssim |\y-\y'|_{\bd}^{\etaforward} (t-s)^{\frac{\nu+\beta - \etaforward}{2}} g_{\lambda}(t-s,\btheta_{t,s}(\x)-y).\label{CTR_DELTA_3}
\end{align}

Let us now turn to the term $\Delta_{3,1}$. 
We recall that the off-diagonal regime holds for the current time integration variable, i.e. $|\y-\y'|_{\mathbf d}^2> t-v $. Thus it suffices to write:
\begin{align}
|\Delta_{3,1}|\le& \int_{t-|\y-\y'|_{\mathbf d}^2}^t\int_{\R^{2d}}p_{\gF,\sigma}(s,\x,v,\w)\Big|(\cL_{v,\w}-\wt\cL^{(t,\y)}_{v,\w}){\wt p}^{(t,\y)}(v,\w, t,\y)\Big|\dif\w \dif v\\
&+\int_{t-|\y-\y'|_{\mathbf d}^2}^t\int_{\R^{2d}}p_{\gF,\sigma}(s,\x , v,\w)\Big|(\cL_{v,\w}-\wt\cL^{(t,\y')}_{v,\w}){\wt p}^{(t,\y')}(v,\w, t,\y')\Big|\dif\w \dif v. \label{CTR_DELTA_4_PREAL}
\end{align}
By the identity \eqref{eq:oper_diff} together with assumptions {\bf [H-${\bf F}$]} and {\bf [H-$\sigma$]}, and by combining \eqref{upper_proxy_bis} 
with \eqref{eq:change_flow}, 
we obtain, for $\mathfrak y\in \{\y,\y' \} $, 
\begin{align}
&\big| \big(\Ac_{v,\w}-\wt\Ac^{( t ,\mathfrak y)}_{v,\w}\big) {\wt p}^{(t ,\mathfrak y)}(v,\w,t,\mathfrak y) \big|\\
& \qquad \lesssim   \Bigg[ \frac{|\w-\btheta_{v,t}(\mathfrak y)|_{\bd}^{\nu+\beta}}{t-v}+\textcolor{black}{\frac{1+|\w-\btheta_{v,t}(\mathfrak y)|_{\bf d}}{(t-v)^{\frac 12}}}+\frac{|\w-\btheta_{v,t}(\mathfrak y)|_{\bf d}^{1+\nu+\beta}}{(t-v)^{\frac 32}}\Bigg] g_{\lambda}\big(t-v,\w-\btheta_{v,t}(\mathfrak y)\big)
\intertext{(by \eqref{BD_THERM_DEG_DER_TEMP})}
& \qquad \lesssim (t-v)^{\frac{\nu+\beta}2-1}g_{\lambda}(t-v,\w-\btheta_{v,t}(\fy)).
\end{align}

Plugging this estimate into \eqref{CTR_DELTA_4_PREAL} and proceeding as in \eqref{CTR_DELTA_3}, namely applying first the Gaussian upper bounds \eqref{HK-density_upper}-\eqref{eq:HK-density_upper_lower} and the Chapman-Kolmogorov identity, then \eqref{eq:change_flow} and finally integrating in $\d v$, yields
\begin{equation}
|\Delta_{3,1}|
\lesssim |\y-\y'|^{\nu+\beta} \big(g_{\lambda}(t-s,\btheta_{t,s}(\x)-\y)+g_{\lambda}(t-s,\btheta_{t,s}(\x)-\y')\big).
\label{CTR_DELTA_4}
\end{equation}

It remains to bound $\Delta_{3,2}$ which is associated with the change of freezing point. Applying Chapman-Kolmogorov identity and Fubini Theorem 
yields
\begin{align}
\Delta_{3,2} & =\int_{\R^{2d}} p_{\gF,\sigma}(s,\x,r,\z) \int_{r}^t\int_{\R^{2d}}p_{\gF,\sigma}(r,\z,v,\w)\big(\Ac_{v,\w}-\wt\Ac^{( t ,\y')}_{v,\w}\big){\wt p}^{(t ,\y')}(v,\w,t,\y')\dif\w \dif v \d \z \\
 & \quad  -\int_{\R^{2d}} p_{\gF,\sigma}(s,\x,r,\z)  \int_{r}^t\int_{\R^{2d}}p_{\gF,\sigma}(r,\z,v,\w)\big(\Ac_{v,\w}-\wt\Ac^{( t ,\y)}_{v,\w}\big){\wt p}^{(t ,\y)}(v,\w,t,\y')\dif\w \dif v \d \z 
 \intertext{(by Duhamel formula \eqref{D1})}
 & = \int_{\R^{2d}} p_{\gF,\sigma}(s,\x,r,\z) \big( {\wt p}^{(t ,\y)}(r,\z,t,\y') -{\wt p}^{(t ,\y')}(r,\z,t,\y')   \big) \dif \z, \label{eq:delta_32_trick}
\end{align}
where we set $r: = t - |\y-\y'|_{\bd}^{2}$. Note that we applied Fubini Theorem to swap the integral in $\d \w \d v$ with the one in $\d \z$. This was possible because the coefficients were regularized, and so derivatives w.r.t. $\w$ could be moved from ${\wt p}$ to $p_{\gF,\sigma}$ and to the coefficients of $\Ac$, thus avoiding non integrable singularities. Critically, the control of the final term in \eqref{eq:delta_32_trick} will only depend on the constants in $\Theta_T$, and not on the regularization constants.

Therefore, by applying first Lemma \ref{Lem_sensitivity}, then the Gaussian upper bounds \eqref{HK-density_upper}-\eqref{eq:HK-density_upper_lower} and the Chapman-Kolmogorov identity, and finally \eqref{eq:change_flow}, we obtain
\begin{equation}
|\Delta_{3,2}| \lesssim |\y-\y'|_{\bd}^{\nu+\beta} g_\lambda \big(t-s,\btheta_{t,s}(\x)-\y\big).\label{CTR_DELTA_2}
\end{equation}

Plugging the controls \eqref{CTR_DELTA_2}, \eqref{CTR_DELTA_4}, \eqref{CTR_DELTA_3}, \eqref{eq:est_D_1} into \eqref{DIFF_Y_YP} eventually yields 
\begin{equation}
| p_{\gF,\sigma}(s,\x,t,\y)- p_{\gF,\sigma}(s,\x,t,\y')| \lesssim  \frac{|\y-\y'|^{\etaforward}_{\bf d}}{(t-s)^{\frac{{\etaforward}}{2}}}   \Big( g_\lambda \big(t-s,\btheta_{t,s}(\x)-\y\big)+g_\lambda \big(t-s,\btheta_{t,s}(\x)-\y'\big)\Big),
\end{equation}
which, by the Gaussian lower bound in \eqref{eq:HK-density_upper_lower}, proves \eqref{HK-holder-gradient-forward} for $\delta=0 $. The case $\delta=1 $ is handled similarly from \eqref{DIFF_Y_YP}, by employing the upper bound on the gradient $\nabla_{x_1} p_{\gF,\sigma} $ provided by \eqref{HK-gradient-back}.

Eventually, the bounds \eqref{eq:ste_new}, \eqref{eq:ste_M_DEG_new} would be proved by proceeding along the same lines, namely by taking the former expansions   and different starting point $\x,\x' $ in \eqref{DIFF_Y_YP} (that would be differentiated w.r.t. the non-degenerate variable when $\delta=1 $ in \eqref{eq:ste_new}), exploiting the Hölder control \eqref{HK-holder-gradient-back} for the perturbative terms (associated with $\nabla_{x_1}^{{\delta}} \Delta_i (\x) - \nabla_{x'_1}^{{\delta}} \Delta_i (\x') $, $i=2,3$
) and the direct Gaussian controls of Proposition \ref{PROP_Proxy} for the main one (associated with $\nabla_{x_1}^{{\delta}} \Delta_1 (\x) - \nabla_{x'_1}^{{\delta}} \Delta_1 (\x')$).


\subsection{Proof of Lemma \ref{Lem_sensitivity}}
For  notational simplicity, we introduce
$$
K_1:=\K^{(t,\y)}_{t,r},\ K_2:=\K^{(t,\y')}_{t,r},\ \w_1:=\bvtheta_{t,r}^{(t,\y)}(\z)-\y',\ \w_2:=\bvtheta^{(t,\y')}_{t,r}(\z)-\y',
$$ 
and
$$
\cA:=|\T^{-1}_{t-r}(\x-\btheta_{r,t}(\y))|+1.
$$
Using the above notations and by definition \eqref{CORRESP}, we have
\begin{equation}
\big|\big(\wt{p}^{(t,\y')}-\wt{p}^{(t,\y)}\big)(r,\z,t,\y')\big|=(2\pi)^{-d}\Big|\det( K_2)^{-\frac 12}
\exp\Big( -\tfrac12\big| K_2^{-\frac12}\w_2\big|^2\Big)
-\det( K_1)^{-\frac 12}
\exp\Big( -\tfrac12\big| K_1^{-\frac12}\w_1\big|^2\Big)\Big|.
\end{equation}
In order to show Lemma \ref{Lem_sensitivity}, it is sufficient to establish that for some $\lambda>0$, which only depends on $\Theta_T$, 
we have
\begin{align}
\big|(\det  K_1)^{-\frac 12}-(\det  K_2)^{-\frac 12}\big|&\lesssim (t-r)^{-2d}|\y-\y'|_{\bd}^{\nu+\beta}
,\label{e12}\\
\big|\exp\big(-\tfrac{1}{2}| K_1^{-\frac12}\w_1|^2\big)-\exp\big(-\tfrac{1}{2}| K_2^{-\frac12}\w_2|\big)\big|
&\lesssim |\y-\y'|_{\bd}^{{\nu+\beta}}\exp\left(-\lambda\cA^2\right).\label{e13}
\end{align}

\paragraph{Proof of \eqref{e12}}
Observe now that, by \eqref{KK1},
%
%
%
%
\begin{align}\label{TT6}
K_i=\T_{t-r}\hat K_i \T_{t-r}, \qquad i=1,2,
\end{align}
where
\begin{align}
\hat K_1:=\int^1_0\hat\gR^{(t,\y)}_{1,v}\bTheta^{(t,\y)}_{{\eta(v)}}\big(\hat\gR^{(t,\y)}_{1,v}\bTheta^{(t,\y)}_{{\eta(v)}}\big)^*\dif v,&&
\hat K_2:=\int^1_0 \hat\gR^{(t,\y')}_{1,v}\bTheta^{(t,\y')}_{{\eta(v)}}\big(\hat\gR^{(t,\y')}_{1,v}\bTheta^{(t,\y')}_{{\eta(v)}}\big)^*\dif v,
\end{align}
with 
\begin{equation}
 \eta(v):=r+(t-r)v,
\end{equation}
and where
\begin{equation}\label{SCALED_RESOLVENT}
\hat\gR^{(t,\xi)}_{1,v} := \T^{-1}_{t-r} {\gR^{(t,\xi)}_{t,\eta (v)}} \T_{t-r} = \begin{pmatrix}\mI_{d\times d} & 0_{d\times d}\\
\int^1_v \nabla_{x_1}F_2\big(\eta(u),\btheta_{\eta(u),t}(\bxi)\big)\dif u & \mI_{d\times d} \end{pmatrix}  .
\end{equation}
Recalling 
{\bf [H-H\"or]}, {\bf [H-${\bf F}$]} and {\bf [H-$\sigma$]}, and owing to Lemma \ref{lemme:bilipflow_FIRST} and 
$r=t-|\y-\y'|_{\bd}^2$, we have, for any $v\in[0,1]$,
\begin{align}
\big|\bTheta^{(t,\y)}_{{\eta(v)}}-\bTheta^{(t,\y')}_{{\eta(v)}}\big| & \lesssim|\btheta_{\eta(v),t}(\y)-\btheta_{\eta(v),t}(\y')|_{\bf d}^{\nu+\beta}
\intertext{(by exploiting, in particular, \eqref{eq:sub_linear_growth})}
&\lesssim |\y-\y'|_{\bd}^{{\nu+\beta}} +(t-r)^{\frac {\nu+\beta} 2}\lesssim |\y-\y'|_{\bd}^{{\nu+\beta}}.
\end{align}
and, following the same arguments,
\begin{align}\label{DIFF_RESOLVENT}
|\hat\gR^{(t,\y)}_{{1,v}} - \gR^{(t,\y')}_{{1,v}}|\lesssim |\y-\y'|_{\bd}^{{\nu+\beta}}.
\end{align}
Hence, by \eqref{TT1}, we obtain
\begin{align}\label{TT9}
|\hat K_1-\hat K_2|\lesssim |\y-\y'|_{\bd}^{\nu+\beta},\ \ |\hat K_i|\lesssim 1,\ \ \det K_i\asymp 1,
\end{align}
and
\begin{align*}
|(\det  K_1)^{-\frac 12}-(\det  K_2)^{-\frac 12}|&\lesssim (t-{r})^{-2d}|(\det \hat K_1)^{-\frac 12}-(\det \hat K_2)^{-\frac 12}|\\
&\lesssim (t-{r})^{-2d}|\det \hat K_1-\det \hat K_2|\\
&\lesssim (t-{r})^{-2d}|\hat K_1-\hat K_2|\lesssim (t-r)^{-2d}|\y-\y'|_{\bd}^{{\nu+\beta}}.
\end{align*}
This proves \eqref{e12}. 

\paragraph{Proof of \eqref{e13}}
Without loss of generality, we may assume
$$
| K_1^{-\frac12}\w_1|\leq | K_2^{-\frac12}\w_2|.
$$
Then by $1-\exp(-x)\leq x$, we have
\begin{align*}
&\big|\exp\big(-\tfrac{1}{2}| K_1^{-\frac12}\w_1|^2\big)-\exp\big(-\tfrac{1}{2}| K_2^{-\frac12}\w_2|^{\textcolor{black}{2}}\big)\big|\leq 
\tfrac12(| K_2^{-\frac12}\w_2|^2-| K_1^{-\frac12}\w_1|^2)\exp\big(-\tfrac{1}{2}| K_1^{-\frac12}\w_1|^2\big).
\end{align*}
By \eqref{TT6} we have
\begin{align}
| K_2^{-\frac12}\w_2|^2-| K_1^{-\frac12}\w_1|^2 & = \<\hat K_2^{-1}\T^{-1}_{t-r}\w_2,\T^{-1}_{t-r}\w_2\>-\<\hat K_1^{-1}\T^{-1}_{t-r}\w_1,\T^{-1}_{t-r}\w_1\>\\
&\leq|\hat K_1^{-1}||\T^{-1}_{t-r}(\w_2-\w_1)|\big(|\T^{-1}_{t-r}\w_2|+|\T^{-1}_{t-r}\w_1|\big)+|\hat K_2^{-1}-\hat K_1^{-1}||\T^{-1}_{t-r}\w_2|^2.\\ \label{DIFF_SENSI_FW}
\end{align}
Now, \eqref{TT9}, \eqref{LIN_NON_LIN_FLOW}, and \eqref{SCALED_RESOLVENT} with $v=0$, yield
\begin{align}\label{CTR_W2}
|\T^{-1}_{t-r}\w_2| \lesssim | \hat\gR^{(t,\y')}_{1,0} \T_{t-r}^{-1} (\z-\btheta_{r,t}(\y'))| \lesssim |\T_{t-r}^{-1} (\z-\btheta_{r,t}(\y'))|. 
\end{align}
Assume now that
\begin{align}
\label{DIFF_W1W2}
|\T^{-1}_{t-r}(\w_2-\w_1)|\lesssim |\y-\y'|_\bd^{\nu+\beta} \Big( 1+ \big|\T_{t-r}^{-1}\big(\z-\btheta_{r,t}(\y')  \big) \big|\Big). 
\end{align}
Then, from \eqref{TT9}, \eqref{DIFF_W1W2} and \eqref{DIFF_SENSI_FW} we derive 
 \eqref{e13} using Lemma \eqref{lemme:bilipflow_FIRST}, 
 as well a convexity inequality for the exponential term, recalling also that $|\y-\y'|_\bd\le (t-s)^{\frac 12} $. 
 Therefore, to complete the proof we need to prove \eqref{DIFF_W1W2}. 

We can write
 \begin{align*}
\T_{t-r}^{-1}(\w_2-\w_1)=\T_{t-r}^{-1} \gR_{t,r}^{(t,\y')}\Big(\z+ m_{r,t}^{(t,\y')}-\gR_{r,t}^{(t,\y')}\y'  \Big)-\T_{t-r}^{-1} \gR_{t,r}^{(t,\y)}\Big(\z+ m_{r,t}^{(t,\y)} - \gR_{r,t}^{(t,\y)}\y'  \Big),
 \end{align*}
 where we set
 \begin{equation}
m_{r,t}^{(\tau,\bxi)} :=   \int^t_r {\gR}^{(\tau,\bxi)}_{r,v}\Big(\gF(v,\btheta_{v,\tau}(\bxi))-\bA\big(v,\btheta_{v,\tau}(\bxi)\big)\btheta_{v,\tau}(\bxi)\Big)\dif v    .
\end{equation}
%
 
  Employing the following identity (see the proof of \cite[Lemma 3.2]{chau:meno:pesc:zhan:23}):
\begin{equation}
\btheta_{r,t}(\y') = \gR_{r,t}^{(t,\y')}\y' -  m_{r,t}^{(t,\y')}  ,
  \end{equation}
and setting
\begin{equation}
\bGamma_{r,t}^{(\tau,\bxi)} :=   m_{r,t}^{(\tau,\bxi)}-\gR_{r,t}^{(\tau,\bxi)}\y' ,
\end{equation}
yields
  \begin{align}
|\T_{t-r}^{-1}(\w_2-\w_1)|
&\le \big| \big( \T_{t-r}^{-1} \gR_{t,r}^{(t,\y')}-\T_{t-r}^{-1} \gR_{t,r}^{(t,\y)} \big) [\z-\btheta_{r,t}(\y')  ] \big| +\big| \T_{t-r}^{-1} \gR_{t,r}^{(t,\y)}\big( \bGamma_{r,t}^{(t,\y)} - \bGamma_{r,t}^{(t,\y')} \big)\big| \\
&\leq \big| \hat \gR_{1,0}^{(t,\y')}- \hat \gR_{1,0}^{(t,\y)}\big|\, \big|\T_{t-r}^{-1}\big(\z-\btheta_{r,t}(\y')  \big)\big|+ \big| \T_{t-r}^{-1} \gR_{t,r}^{(t,\y)}\big(\bGamma_{r,t}^{(t,\y)}-\bGamma_{r,t}^{(t,\y')} \big)\big| \\
 & \lesssim |\y-\y'|_{\bd}^{\nu+\beta} \big| \T_{t-r}^{-1}\big( \z-\btheta_{r,t}(\y')  \big) \big|+  \big| \T_{t-r}^{-1} \big(\bGamma_{r,t}^{(t,\y)}-\bGamma_{r,t}^{(t,\y')} \big)\big| 
 \label{THE_DIFF_W}
 \end{align}
 using \eqref{DIFF_RESOLVENT}  and \eqref{SCALED_RESOLVENT} for the last inequality. To investigate the difference $\big(\bGamma_{r,t}^{(t,\y)}-\bGamma_{r,t}^{(t,\y')} \big)$ we go back to the corresponding differential dynamics, namely
 \begin{equation}
\bGamma_{r,t}^{(\tau,\bxi)} = \y'-\int_{r}^t \Big[\gF(v,\btheta_{v,\tau}( \bxi))+\bA(v,\btheta_{v,\tau}( \bxi))\big(\bGamma_{v,t}^{(\tau,\bxi)}- \btheta_{v,\tau}( \bxi)\big)\Big] \d v.
\end{equation} 
In particular, observe again that $\bGamma_{v,t}^{(t,\y')}=\btheta_{v,t}(\y'),\ v\in [r,t] $. Thus
\begin{align}
\bGamma_{r,t}^{(t,\y)}-\bGamma_{r,t}^{(t,\y')}
& =
\int_{r}^t \Big(\gF\big(v,\btheta_{v,t}(\y')\big)-\gF\big(v,\btheta_{v,t}(\y)\big)-\bA\big(v,\btheta_{v,\tau}( \y)\big)\big(\btheta_{v,\tau}(\y')-\btheta_{v,\tau}(\y)\big)\\
&\qquad\qquad+\bA\big(v,\btheta_{v,\tau}( \y)\big)\big(\bGamma_{v,t}^{(t,\y')}-\bGamma_{v,t}^{(t,\y)}\big) \Big) \d v. 
\end{align}
This yields
  \begin{align}
\Big|  \Big( \T_{t-r}^{-1}\big(\bGamma_{r,t}^{(t,\y)}-\bGamma_{r,t}^{(t,\y')} \big)\Big)_1 \Big|&\leq (t-r)^{-\frac 12}\int_{r}^t \big| \gF_1\big(v,\btheta_{v,t}({\y'})\big)-\gF_1\big(v,\btheta_{v,t}({\y})\big) \big| \d v\\
 & \lesssim (t-r)^{-\frac 12}\int_{r}^t 
  (1+|\btheta_{v,t}({\y'})-\btheta_{v,t}({\y})|_\bd )\d v\\
  &\lesssim (t-r)^{\frac 12}+ (t-r)^{-\frac{1}{2}} \int_{r}^ t \big( (t-v)^{\frac 12}+|\y-\y'|_\bd \big)\d v \lesssim |\y-\y'|_\bd,\label{CTR_DIFF_1}
 \end{align}
  where we used $|\y-\y'|_\bd =  (t-r)^{\frac 12}$ in the last inequality, {as well Lemma \ref{lemme:bilipflow_FIRST} for the second last inequality. On the other hand,
\begin{align}
\Big|\Big( \T_{t-r}^{-1}\big(\bGamma_{r,t}^{(t,\y)}-\bGamma_{r,t}^{(t,\y')} \big)\Big)_2\Big|
&\lesssim (t-r)^{-\frac 32}\int_{r}^t \Big( |\btheta_{v,t}({\y'})-\btheta_{v,t}({\y})|_\bd^{1+{\nu+\beta}}+
\big|\big(\bGamma_{v,t}^{(t,\y)}-\bGamma_{v,t}^{(t,\y')}\big)_1\big| \Big)\d v
\intertext{(proceeding as in \eqref{CTR_DIFF_1})}
&\lesssim (t-r)^{-\frac 32}\int_{r}^t \Big( \big((t-v)^{\frac 12}+|\y-\y'|_\bd\big)^{1+{\nu+\beta}}+
 (t-v)^{\frac{1}{2}} |\y-\y'|_\bd \Big)\d v
 \intertext{(as $|\y-\y'|_\bd =  (t-r)^{\frac 12}$)}
 & \lesssim |\y-\y'|_{\bd}^{\nu+\beta},
\end{align}
which, together with \eqref{CTR_DIFF_1} and \eqref{THE_DIFF_W}, proves \eqref{DIFF_W1W2}. \qed


\section{Equivalence between the thermic and the Littlewood-Paley characterization in the kinetic anisotropic setting}
\label{EQUIV_CAR_ANIS_BESOV_SPACES}
We recall here the class of anisotropic Besov spaces $\bB^{\beta,q}_{\bbp;\bba}$, defined via Littlewood-Paley decomposition as in the classical reference \cite[Chapter 5]{triebel06}, and considered as well in \cite{hao:rock:zhan:26}. For simplicity, we keep the notations considered in the latter reference and, when it is useful, make the connection with the notations used in the main body of the article.

For $z=(x,v)$ and $z'=(x',v')$ in $\R^{2d}$,  the anisotropic distance 
\begin{equation}\label{AnisoDist}
	|z-z'|_\bba:=|x- x'|^{1/(1+\alpha)}+|v-v'|, \qquad \alpha\in (1,2],
\end{equation}
is introduced. For $\alpha=2 $ it does actually correspond to $|\cdot|_{\mathbf d} $ defined in \eqref{HOMO_METRIC}, with the association $\x = (\x_1, \x_2) = (v , x )$.


In order to consider a Littlewood-Paley type decomposition, we introduce the related metric balls. Namely,
for $r>0$ and $z\in\R^{2d}$,  the ball centred at $z$ and with radius $r$ with respect to $|\cdot |_\bba$ is defined 
as follows:
$$
B^\bba_r(z):=\{z'\in\R^{2d}:|z'-z|_\bba\leq r\},\ \ B^\bba_r:=B^\bba_r(0).
$$
Let $\chi^\bba_0$ be  a symmetric $C^{\infty}$-function  on $\R^{2d}$ with
$$
\chi^\bba_0(\xi)=1\ \mathrm{for}\ \xi\in B^\bba_1\ \mathrm{and}\ \chi^\bba_0(\xi)=0\ \mathrm{for}\ \xi\notin B^\bba_2.
$$
In order to consider the corresponding dyadic partition, define:
$$
\phi^\bba_j(\xi):=
\left\{
\begin{aligned}
	&\chi^\bba_0(2^{-j\bba}\xi)-\chi^\bba_0(2^{-(j-1)\bba}\xi),\ \ &j\geq 1,\\
	&\chi^\bba_0(\xi),\ \ &j=0,
\end{aligned}
\right.
$$
where for $s\in\R$ and $\xi=(\xi_1,\xi_2)$, we define the dilation
$$
2^{s\bba }\xi:=(2^{(1+\alpha)s}\xi_1, 
2^{s}\xi_2).
$$
Note that
\begin{align}\label{Cx8}
	{\rm supp}(\phi^\bba_j)\subset {B^\bba_{2^{j+1}} \setminus B^\bba_{2^{j-1}}}
	=:C_{2^j}^{\bba},\ j\geq 1,&& {\rm supp}(\phi^\bba_0)\subset B^\bba_2 =: C_{1}^{\bba},
\end{align}
and
\begin{align}\label{AA13}
	\sum_{j\geq 0}\phi^\bba_j(\xi)=1,\ \ \forall\xi\in\R^{2d}.
\end{align}

As in Section \ref{SEC_DEF_BESOV_ANIS}, let $\cS$ be the space of all Schwartz functions on $\R^{2d}$, and $\cS'$ be its dual space (tempered distributions). 
For given $j\geq 0$, the  dyadic anisotropic block operator  $\mathcal{R}^\bba_j$ is defined on $\cS'$ by
\begin{align}\label{Ph0}
	\mathcal{R}^\bba_jf(z):=(\phi^\bba_j\hat{f})\check{\ }(z)=\check{\phi}^\bba_j*f(z), \qquad {f\in \cS'},
\end{align}
where the convolution is understood in the distributional sense, {and where the signs $\hat{{}}$ and $\check{{}}$ denote the Fourier transform and anti-transform, respectively.} Note that, by scaling, we have
\begin{align}\label{SX4}
	\check{\phi}^\bba_j(z)=2^{(j-1)(2+\alpha)d}\, \check{\phi}^\bba_1\big(2^{(j-1)\bba}z\big),\quad j\geq 1.
\end{align}
Now we introduce the following anisotropic Besov spaces.
\begin{definition}[Anisotropic Besov spaces through Littlewood-Paley]\label{bs}
	Let $\vartheta\in\R$, $q\in[1,\infty]$ and $\bbp=(p_x,p_v)\in[1,\infty]^2$. The  anisotropic Besov space is defined by
	\begin{equation}
	\mathbb{B}^{\vartheta,q}_{\bbp;\bba}:=\Big\{f\in \cS'\textcolor{black}{(\R^{2d})}\ : \  \|f\|_{\mathbb{B}^{\vartheta,q}_{\bbp;\bba}}
	<\infty\Big\},
\end{equation}
with 
\begin{equation}
\|f\|_{\mathbb{B}^{\vartheta,q}_{\bbp;\bba}}
	:= \begin{cases}
	\Big(\sum_{j\geq0}\big(2^{ j\vartheta}\|\cR^\bba_{j}f\|_{\bbp}\big)^q\Big)^{1/q}, & q\in [1,\infty), \\
	\sup_{j\geq0} \big\{ 2^{ j\vartheta}\|\cR^\bba_{j}f\|_{\bbp} \big\}, & q = \infty,
	\end{cases}
\end{equation}
and	where  $\mL^{\bbp}=L^{p_v}(\R^d;L^{p_x}(\R^d))=\{f:\R^{2d}\rightarrow\R \ \text{Borel measurable}:\,\|f\|_{\mL^\bbp}<\infty\}$ for \textcolor{black}{$L^p(\mathbb R^d)$ denoting the Lebesgue space on $\mathbb R^d$ and}
\begin{equation}\label{LP1}
	\|f\|_{\mL^\bbp}:=\|f\|_{\bbp}:=
\begin{cases}	
	\left(\int_{\R^d}\|f(\cdot,v)\|_{p_x}^{p_v} d  v\right)^{1/p_v}, & p_v \in [1,\infty), \\
	\sup_{v\in \Rd} \big\{ \|f(\cdot,v)\|_{p_x}^{p_v}  \big\}, & p_v = \infty .
	\end{cases}
\end{equation}
Finally, for $p \in [1, \infty]$, we set
\begin{equation}
\|f\|_{\mathbb B_{p;\bba}^{\vartheta,q}}:=\|f\|_{\mathbb B_{(p,p);\bba}^{\vartheta,q}}  = \begin{cases}
	\Big(\sum_{j\geq0}\big(2^{ j\vartheta}\|\cR^\bba_{j}f\|_{L^p(\Rdd)}\big)^q\Big)^{1/q}, & q\in [1,\infty) ,\\
	\sup_{j\geq0} \big\{ 2^{ j\vartheta}\|\cR^\bba_{j}f\|_{L^p(\Rdd)} \big\}, & q = \infty.
	\end{cases}
\end{equation}		
\end{definition}	

We will actually prove that, when $p_x=p_v=p$ and $\alpha=2 $, the norms associated with Definition  \ref{bs} and the former Definition \ref{DEF_BESOV_THERMIC_ANISO} are equivalent. We proceed somehow as in the usual isotropic case, see e.g. Lemarié-Rieusset \cite{lema:02} or Triebel \cite{trie:83}. 

\begin{lemma}[Equivalence of the norms in the anisotropic case]
Let $\vartheta \in \R$ and $p,q\in [1,+\infty]$. The Littlewood-Paley norm $\|\cdot\|_{\mathbb{B}^{\vartheta,q}_{p;\bba}} $ of Definition \ref{bs}, with $\alpha = 2$, and the thermic norm $\|\cdot\|_{\mathcal{H}_{p,q,\bf d}^{\vartheta}}$ of Definition \ref{DEF_BESOV_THERMIC_ANISO} are equivalent.
\end{lemma}

\begin{proof}
For brevity, we only prove the equivalence for the case $\vartheta<2$, which means that we can set $n=1$ in Definition \ref{DEF_BESOV_THERMIC_ANISO}. This covers all the regularity indices that are considered in the main body of the paper. The general proof is a straightforward modification.

We will first dominate the thermic norm of Definition \ref{DEF_BESOV_THERMIC_ANISO} by the Littlewood-Paley norm of Definition \ref{bs}. Namely, we establish that there exists a positive $C=C(p,q,\vartheta)$ such that
\begin{align}\label{HEAT_DOM_BY_LP}
\|f\|_{\mathcal{H}_{p,q,\bf d}^{\vartheta}}\le C\|f\|_{\mathbb B_{p;\mathbf a}^{\vartheta,q}}.
\end{align}

To this end, write for $j\ge 1$:
\begin{align}\label{SWAP_DYADIC_BLOCK}
\|\mathcal R_j \partial_v^{n}g(v,\cdot)* f\|_{L^p}=\| \partial_v^{n}g(v,\cdot)* \mathcal R_jf\|_{L^p}, \qquad v\in[0,1].
\end{align}
Observe now that for an arbitrary function $u$ such that $\widehat u $ has support on $C_{2^j}^\bba$, for all $\x\in \R^{2d} $ :
\begin{align*}
g(v,\cdot)* u(\x)&=\mathcal F^{-1}\Bigg(\bxi\mapsto \varphi_{2^j}(\bxi) \exp\Big(-v|\bxi_1|^2-v^3|\bxi_2|^2\Big) \widehat u(\bxi)\Bigg)(\x)\\
&=\Big(\frac{1}{(2\pi)^{2d}}\int_{\R^{2d}}\exp(i\langle \cdot, \bxi\rangle)\exp\Big(-v |\bxi_1|^2-v^3|\bxi_2|^2\Big)\varphi_{2^j}(\bxi) d\bxi\Big)* u(\x)=:g_j(v,\cdot)* u(\x),
\end{align*}
where $\varphi_{2^j}\in \mathcal C^{\infty}_0(B^\bba_{2^{j+1}}\backslash \{0\}) $ such that: 
\begin{itemize}
\item $\varphi_{2^j}=1 $ on $ C_{2^j}^\bba$ and $\varphi_{2^j}=0$ on $C_{2^{j-2}}^\bba$, for $\ j\ge 2$;
\item there exists $C>0$ such that, for any multi-index $\alpha$ with $|\alpha|\le d $, $\|\partial^\alpha \varphi_{2^j}\|\le C $ for any $\ j\ge 2$.
\end{itemize}
In particular, from Young's inequality:	
\begin{equation}\label{eq:young_r}
\|g(v,\cdot)* u(\x)\|_{L^p}\le 	\|g_j(v,\cdot)\|_{L^1}\|u\|_{L^p}.
\end{equation}
We will thus focus on:
\begin{align*}
g_j(v,\x)&=\prod_{{k}=1}^2\frac{1}{(1+\frac{|\x_{{k}}|^2}{v^{2{k}-1}})^{d}}\frac{1}{(2\pi)^{2d}}\int_{\R^{2d}}\prod_{{k}=1}^2(1+\frac{|\x_{{k}}|^2}{v^{2{k}-1}})^{d}\exp(i\langle \x, \bxi\rangle)\exp\Big(-v |\bxi_1|^2-v^3|\bxi_2|^2\Big)\varphi_{2^j}(\bxi) d\bxi\\
&=\prod_{{k}=1}^2\frac{1}{(1+\frac{|\x_{{k}}|^2}{v^{2{k}-1}})^{d}}\frac{1}{(2\pi)^{2d}}\int_{\R^{2d}} \exp\Big(-v |\bxi_1|^2-v^3|\bxi_2|^2\Big)\varphi_{2^j}(\bxi)\prod_{{k}=1}^2(I-\frac{\Delta_{\bxi_{{k}}}}{v^{2{k}-1}})^{d}\exp(i\langle \x, \bxi\rangle) d\bxi
\intertext{{(integrating by parts)}}
&=\prod_{{k}=1}^2\frac{1}{(1+\frac{|\x_{{k}}|^2}{v^{2{k}-1}})^{d}}\frac{1}{(2\pi)^{2d}}\int_{\R^{2d}}\exp(i\langle \x, \bxi\rangle)\prod_{{k}=1}^2(I-\frac{\Delta_{\bxi_{k}}}{v^{2{k}-1}})^{d}\Big(\exp\Big(-v |\bxi_1|^2-v^3|\bxi_2|^2\Big)\varphi_{2^j}(\bxi)\Big) d\bxi.
\end{align*}
Hence, from  the support of $\varphi_{2^j}$, we derive:
\begin{align*}
|g_j(v,{\x})|&\le C\prod_{{k}=1}^2\frac{1}{(1+\frac{|\x_{{k}}|^2}{v^{2{k}-1}})^{d}}\int_{{\rm supp}(\varphi_{2^j})} (1+v)^d(1+v^3)^{d}\exp\left(- c\big(v 2^{2(j-1)}+v^3 2^{6(j-1)}\big)\right)  d\bxi\\
&\le C \prod_{{k}=1}^2\frac{1}{(1+\frac{|\x_{{k}}|^2}{v^{2{k}-1}})^{d}} 2^{4(j-1)d}(1+v)^d(1+v^3)^{d}\exp\left(- {c}\big(v 2^{2(j-1)}+v^3 2^{6(j-1)}\big)\right)\\
&\le C \prod_{{k}=1}^2 \frac{1}{v^{\frac {2{k}-1}{2}d}}\frac{1}{(1+\frac{|\x_{{k}}|^2}{v^{2{k}-1}})^{d}} \exp\left(- c\big(v 2^{2(j-1)}+v^3 2^{6(j-1)}\big)\right),
\end{align*}
and eventually, by integrating in $\x$,
\begin{align*}
\|g_j(v,\cdot)\|_{L^1}\le C \exp\left(- c(v 2^{2(j-1)}+v^3 2^{6(j-1)})\right).
\end{align*}
Observe now that $ g(v,\x)={g_{\bf e}(v,\x_1)g_{\bf e}(v^3,\x_2)} $, where $g_{\bf e}(v,z) $ denotes the density of the standard $d $-dimensional  Brownian motion at time $v$ and point $z\in \R^d$, which solves the heat equation $\partial_v g_{\bf e}(v,z)=\frac 12 \Delta g_{\bf e}(v,z) $. We then derive 
\begin{align*}
\partial_v g(v,\x)&=\partial_v g_{\bf e}(v,\x_1)g_{\bf e}(v^3,\x_2)+g_{\bf e}(v,\x_1)\partial_v g_{\bf e}(v^3,\x_2)\\
&=\frac 12 \Delta_{\x_1} g_{\bf e}(v,\x_1)g_{\bf e}(v^3,\x_2)+g_{\bf e}(v,\x_1)\frac 32 v^2 \Delta_{\x_2}g_{\bf e}(v^3,\x^2).
\end{align*}
Hence, similarly to the previous computations, we derive
for an arbitrary function $u$ such that $\widehat u $ has support on $C_{2^j}^\bba$, for all ${\x}\in \R^{2d}$:
\begin{align*}
&\partial_v g(v,\cdot)* u(\x)\\
&= \mathcal F^{-1}\Bigg(\bxi\mapsto - \varphi_{2^j}(\bxi) \Big(\frac 12|\bxi_1|^2+\frac 32 v^2 |\bxi_2|^2\Big)\exp\Big(-v|\bxi_1|^2-v^3|\bxi_2|^2\Big) \widehat u(\bxi)\Bigg)(\x)\\
&= \Big(-\frac{1}{(2\pi)^{2d}}\int_{\R^{2d}}\Big(\frac 12|\bxi_1|^2+\frac 32 v^2 |\bxi_2|^2\Big)\exp(i\langle \cdot, \bxi\rangle)\exp\Big(-v |\bxi_1|^2-v^3|\bxi_2|^2\Big)\varphi_{2^j}(\bxi) d\bxi\Big)* u(\x)\\
&=:\partial_v g_j(v,\cdot)* u(\x).
\end{align*}
Reproducing the same previous arguments, one then derives:

\begin{equation}
\|\partial_v g_j(v,\cdot)\|_{L^1}\leq C(2^{2(j-1)}+v^2 2^{6(j-1)}) \exp\left(- c(v 2^{2(j-1)}+v^3 2^{6(j-1)})\right)\le C2^{2(j-1)} \exp\left(- cv 2^{2(j-1)}\right).\label{PARTIAL_TIME_LOC_FREQ}
\end{equation}
Observe that \eqref{PARTIAL_TIME_LOC_FREQ} holds true also for $j=0$ (initial localization on the centered metric ball).

For any $f\in \cS'$, we have
\begin{equation}
\| \partial_vg(v,\cdot)* f\|_{L^p}\le \sum_{j\ge 0}\|\mathcal R_j \partial_v g(v,\cdot)* f\|_{L^p}\le 
\sum_{j\ge 0}\| \partial_v g(v,\cdot)* \mathcal R_jf\|_{L^p}\le C\sum_{j\ge 0}2^{2j} \exp\left(- cv 2^{2j}\right)\|\mathcal R_j f\|_{L^p},\label{PRIMA_DI_HOLDER}
\end{equation}
using \eqref{SWAP_DYADIC_BLOCK}, and  \eqref{eq:young_r}-\eqref{PARTIAL_TIME_LOC_FREQ} for the last inequality. Now, consider first the case $q<+\infty $. From the Hölder inequality and Definition \ref{bs}, we obtain:
\begin{align*}
\| \partial_vg(v,\cdot)* f\|_{L^p}\le C \|f\|_{\mathbb B_{p;\mathbf a}^{\vartheta,q}} \sum_{j\ge 0} 2^{2j(1-\frac \vartheta 2)}\exp\left(- cv 2^{2j}\right)c_{q,j},
\end{align*}
where $(c_{q,j})_{j\in \N} \in \ell^q(\N)$ is such that $\|c_q\|_{\ell^q}=1$ and, by duality, 
$$\|j\in \N\mapsto  2^{2j(1-\frac \vartheta 2)}\exp\left(- cv 2^{2j}\right)\|_{\ell^{q'}(\N)}= \sum_{j\ge 0} 2^{2j(1-\frac \vartheta 2)}\exp\left(- cv 2^{2j}\right)c_{q,j}.$$
Write then:
\begin{align}
\int_0^1 \frac {\d v}v  v^{(1-\frac \vartheta2)q}\|\partial_v g(v,\cdot)* f\|_{L^p}^q 
& \leq C\|f\|_{\mathbb B_{p;\mathbf a}^{\vartheta,q}}^q
\int_0^1 \frac {\d v}v \Big( \sum_{j\ge 0} (v2^{2j})^{1-\frac \vartheta2}\exp\left(- v 2^{2j}\right)c_{q,j}\Big)^q\notag\\
 & \leq C\|f\|_{\mathbb B_{p;\mathbf a}^{\vartheta,q}}^q
\int_0^1 \frac {\d v}v  I(v)^{q-1} \sum_{j\ge 0} (v2^{2j})^{1-\frac \vartheta2}\exp\left(- v 2^{2j}\right)c_{q,j}^q
.\label{INTER_STEP_1}
\end{align}
where
\begin{equation}
I(v):= \sum_{j\ge 0} (v2^{2j})^{1-\frac \vartheta2}\exp\left(- v 2^{2j}\right)\le v^{1-\frac\vartheta2}\exp(-v)+\sum_{j\ge 1}\int_{[2^j,2^{j+1}]} (vz^2)^{1-\frac \vartheta2}\exp\left(- v \frac{z^2}4\right)\frac{dz}{2^{j+1}-2^j}.
\end{equation}
Since 
\begin{equation}
\frac{1}{2^{j+1}-2^j}\le \frac{C}{2^{j+1}} \iff 2^{j+1}\le C(2^{j+1}-2^j)\iff 2 \leq 2C - C  
 \iff C\ge 2 ,
\end{equation}
we have
\begin{align}
I(v)&\leq v^{1-\frac\vartheta2}\exp(-v)+ {2} \int_{2}^{\infty} (vz^2)^{1-\frac \vartheta2}\exp\left(- v \frac{z^2}4\right)\frac{dz}{z}
\intertext{(setting $\tilde z=v^{\frac 12}z$)}
&= v^{1-\frac\vartheta2}\exp(-v)+2 \int_{2v^{\frac 12}}^{\infty} \tilde z^{2 (1-\frac \vartheta2 )}\exp\left(-  \frac{\tilde z^2}4\right)\frac{d\tilde z}{\tilde z}\\
&\le v^{1-\frac\vartheta2}\exp(-v)+2\int_{0}^{\infty} {\tilde z}^{1- \vartheta}\exp\left(-  \frac{\tilde z^2}4\right)d\tilde z \leq C_{\vartheta}, \label{BOUND_WIEGHTED_SUM}
\end{align}
with $C_{\vartheta}>0$ that is independent of $v>0$.
Hence, \eqref{INTER_STEP_1} and changing variables, $\tilde v=v2^{2j}$, yields:
\begin{equation}
\int_0^1 \frac {\d v}v v^{(1-\frac \vartheta2)q}\|\partial_v g(v,\cdot)* f\|_{L^p}^q \leq  C\|f\|_{\mathbb B_{p;\mathbf a}^{\vartheta,q}}^q\sum_{j\ge 0}c_{q,j}^q\int_0^{2^{2j}} \tilde v^{1-\frac \vartheta 2}\exp(-c \tilde v) \d \tilde v\le C\|f\|_{\mathbb B_{p;\mathbf a}^{\vartheta,q}}^q \Gamma\Big(1-\frac \vartheta2\Big).
\end{equation}
By proceeding as above, let us now bound:
\begin{align}
\|g(1,\cdot)* f\|_{L^p}&\leq \sum_{j\ge 0}\|g(1,\cdot)* \mathcal R_j f\|_{L^p}\le \|\mathcal R_0 f\|_{L^p}+\sum_{j\ge 1}\exp(-c2^{2j}) \|\mathcal R_j f\|_{L^p}\\
&\le \|\mathcal R_0 f\|_{L^p}+\Big(\sum_{j\ge 1} 2^{-j\vartheta\frac{q}{q-1}}\exp\big(-\frac{q}{q-1}c2^{2j}\big)\Big)^{\frac {q-1}{q}}\|f\|_{\mathbb B_{p;\mathbf a}^{\vartheta,q}}\le C \|f\|_{\mathbb B_{p;\mathbf a}^{\vartheta,q}}.
\end{align}

We thus have  proved \eqref{HEAT_DOM_BY_LP} for $q<+\infty $. For the case $q=+\infty $, we restart from \eqref{PRIMA_DI_HOLDER} and write:
\begin{align*}
\|f\|_{\mathcal{H}_{p,\infty,\bf d}^{\vartheta}} & =\|g(1,\cdot)* f\|_{L^p}+\sup_{v\in (0,1]}v^{1-\frac \vartheta2}\| \partial_vg(v,\cdot)* f\|_{L^p}\\
&\leq \|\mathcal R_0 f\|_{L^p}+\Big(\sum_{j\ge 1} 2^{-j\vartheta}\exp(-c2^{2j})\Big)\Big(\sup_{j\in \N} \big\{ 2^{j\vartheta} \|\mathcal R_j f\|_{L^p} \big\} \Big)\\
&\quad +C\Big(\sup_{j\in \N} \big\{ 2^{j\vartheta} \|\mathcal R_j f\|_{L^p} \big\} \Big)  \sup_{v\in (0,1]}\sum_{j\ge 0} (v2^{2j})^{1-\frac \vartheta 2}\exp\left(- v 2^{2j}\right)
\le  C\|f\|_{\mathbb B_{p;\mathbf a}^{\vartheta,\infty}},
\end{align*}
using the Definition \ref{bs} and \eqref{BOUND_WIEGHTED_SUM} for the last inequality. Hence, \eqref{HEAT_DOM_BY_LP} still holds for $q=+\infty $.

\vspace{5pt}

Let us now prove the converse, namely 

By exploiting
\begin{equation}
\frac{1}{\Gamma(2-\frac \vartheta 2)}\int_{0}^{+\infty} t^{1-\frac \vartheta 2}\exp(-t) \d t =1,
\end{equation}
and changing variables, $ t=v |\bxi_1|^2+v^3|\bxi_2|^2 $, one has
\begin{align}
1 
&=\frac{1}{\Gamma(2-\frac \vartheta 2)}\int_{0}^{+\infty} (v |\bxi_1|^2+v^3|\bxi_2|^2)^{1-\frac \vartheta 2}\exp\Big(-(v|\bxi_1|^2+v^3|\bxi_2|^2)\Big) (|\bxi_1|^2+3v^2|\bxi_2|^2)\d v\\
&=\frac{1}{\Gamma(2-\frac \vartheta 2)}\int_{0}^{+\infty}v^{1-\frac \vartheta2}(|\bxi_1|^2+v^2|\bxi_2|^2)^{1-\frac \vartheta 2}\exp\Big(-(v|\bxi_1|^2+v^3|\bxi_2|^2)\Big) (|\bxi_1|^2+3v^2|\bxi_2|^2)\d v .
\end{align}
Therefore,
\begin{align}
{\mathcal R_j f}(\x)
&=\frac{1}{\Gamma(2-\frac \vartheta 2)}\int_{0}^{+\infty}v^{1-\frac \vartheta2}\big(-(\Delta_{\x_1}+v^2\Delta_{\x_2})\big)^{1-\frac \vartheta 2}g(v/2,\cdot)* \partial_v g(v/2,\cdot)* {\mathcal R_j f}(\x)\d v,
\end{align}
and, employing the $L^1-L^p $ Young inequality and the arguments that led to \eqref{PARTIAL_TIME_LOC_FREQ}, we obtain
\begin{align}
\|{\mathcal R_j f}\|_{L^p}&\le C\int_0^{+\infty} v^{1-\frac \vartheta 2}(2^{2j}+v^22^{6j})^{1-\frac{\vartheta}2}\exp(-c (v2^{2j}+v^32^{6j}))\|\partial_v g(v/2,\cdot)*{\mathcal R_j f}\|_{L^p}\d v\\
&\le C\int_0^{+\infty} v^{1-\frac \vartheta 2}2^{2j(1-\frac{\vartheta}2)}\exp(-c v2^{2j}) \|\partial_v g(v,\cdot)*{\mathcal R_j f}\|_{L^p}\d v.
\end{align}
As
\begin{equation}
 \|\partial_v g(v,\cdot)*{\mathcal R_j f}\|_{L^p} =  \| \mathcal R_j  \partial_v g(v,\cdot)* f\|_{L^p} \leq C \|  \partial_v g(v,\cdot)* f\|_{L^p}, 
\end{equation}
we finally have
\begin{equation}
\|{\mathcal R_j f}\|_{L^p} \leq C\int_0^{+\infty} v^{1-\frac \vartheta 2}2^{2j(1-\frac{\vartheta}2)}\exp(-c v2^{2j})\|\partial_v g(v,\cdot)* f\|_{L^p}\d v.
\end{equation}

We now again consider separately the cases $q=\infty $ and $q<+\infty $. For the first one write:
\begin{align}
\|f\|_{\mathbb B_{p;\mathbf a}^{\vartheta,\infty}}=&\sup_{j\ge 0}\Big( 2^{j\vartheta}\|{\mathcal R_j f}\|_{L^p}\Big)\\
&\le C\sup_{j\ge 0}\Big(\sup_{v\in (0,1]} \Big\{ v^{1-\frac \vartheta 2}\|\partial_v g(v,\cdot)* f\|_{L^p} \Big\} \int_0^{1} 2^{2j}\exp(-cv2^{2j})\d v \Big) \\
&\qquad\qquad\quad+\int_1^{+\infty}2^{2j}\exp(-cv2^{2j})\underbrace{\|g(c v,\cdot )* f\|_{L^p}}_{\leq \|g(1,\cdot )* f\|_{L^p}}  \d v\Big)\\
 &\le C\Big(\sup_{v\in (0,1]}v^{1-\frac \vartheta 2}\|\partial_v g(v,\cdot)* f\|_{L^p}+\frac{\exp(-c)}c \|g(1,\cdot )* f\|_{L^p}  \d v\Big)
 \leq C\|f\|_{\mathcal H_{p,\infty,\bf d}^{\vartheta}}
.
\end{align}
The case $q<+\infty $ is analogous. We leave the details to the reader, for the sake of brevity.
\end{proof}

\end{document}